\documentclass[12pt]{amsart}

\usepackage{tikz,pgf,amscd,enumerate,multicol}
\usepackage[pagebackref,colorlinks,linkcolor=red,citecolor=blue,urlcolor=blue,hypertexnames=false]{hyperref}
\usepackage{amsmath,amssymb,tikz-cd,tkz-graph}

\usepackage[margin=1.2cm]{geometry}
\usetikzlibrary{decorations.markings,arrows}

\tikzset{midarrow/.style = {postaction=decorate, decoration={markings,mark=at position .5 with \arrow{stealth}}}}

\pgfdeclarelayer{background}
\pgfsetlayers{background,main}

\usepackage[all,2cell]{xy}

\newcommand{\ZZ}{\mathbb{Z}}

\newcommand{\N}{\mathbf{n}}
\newcommand{\M}{\mathbf{m}}

\newcommand{\Hom}{\operatorname{Hom}}
\newcommand{\Obj}{\operatorname{Obj}}

\newcommand{\MCE}{\operatorname{MCE}}
\newcommand{\KP}{\operatorname{KP}}

\newcommand{\Pgr}[1][]{{\operatorname{{Pgr}^{#1}\!-}}}

\newtheorem{thm}{Theorem}[section]

\newtheorem{lem}[thm]{Lemma}
\newtheorem{prop}[thm]{Proposition}

\newtheorem{prob}[thm]{Problem}
\newtheorem{ques}[thm]{Question}

\theoremstyle{definition}
\newtheorem{dfn}[thm]{Definition}

\theoremstyle{remark}
\newtheorem{rmk}[thm]{Remark}
\newtheorem{rmks}[thm]{Remarks}
\newtheorem{example}[thm]{Example}

\newcommand{\image}{\operatorname{Im}}
\newcommand{\Ker}{\operatorname{Ker}}
\newcommand{\Aut}{\operatorname{Aut}}

\newcommand{\Ext}{\operatorname{Ext}}
\newcommand{\End}{\operatorname{End}}
\newcommand{\lspan}{\operatorname{span}}

\def\r{{\textsf r}}

\newcommand{\Typ}{\operatorname{Typ}}
\newcommand{\typ}{\operatorname{typ}}
\newcommand{\gr}{\operatorname{gr}}
\newcommand{\Gr}{\operatorname{\textbf{Gr}}}
\newcommand{\Mod}{\operatorname{\textbf{Mod}}}

\allowdisplaybreaks

\begin{document}

\title{Higher-rank graphs and the graded $K$-theory of Kumjian--Pask Algebras}

\author[Hazrat, Mukherjee, Pask and Sardar]{{ Roozbeh Hazrat, } { Promit Mukherjee, } { David Pask, } and { Sujit Kumar Sardar}}
\address{ (R. Hazrat) Centre for Research in Mathematics and Data Science, Western Sydney University,  NSW 2150, Australia}
\email{r.hazrat@westernsydeny.edu.au}
\address{ (P. Mukherjee) Department of Mathematics, Jadavpur University, Kolkata-700032, India}
\email{promitmukherjeejumath@gmail.com}
\address{ (D. Pask) School of Mathematics and Applied Statistics, University of Wollongong,  NSW 2522, Australia}
\email{david.a.pask@gmail.com}
\address{ (S. K. Sardar) Department of Mathematics, Jadavpur University, Kolkata-700032, India}
\email{sujitk.sardar@jadavpuruniversity.in, sksardarjumath@gmail.com}   



\begin{abstract}
This paper lays out the foundations of graded $K$-theory for Leavitt algebras associated with higher-rank graphs, also known as Kumjian--Pask algebras, establishing it as a potential tool for their classification. 

For a row-finite $k$-graph $\Lambda$ without sources, we show that there exists a $\mathbb{Z}[\mathbb{Z}^k]$-module isomorphism between the graded zeroth (integral) homology $H_0^{\gr}(\mathcal{G}_\Lambda)$ of the infinite path groupoid $\mathcal{G}_\Lambda$ and the graded Grothendieck group $K_0^{\gr}(\KP_\mathsf{F}(\Lambda))$ of the Kumjian--Pask algebra $\KP_\mathsf{F}(\Lambda)$, which respects the positive cones (i.e., the talented monoids).

We demonstrate that the $k$-graph moves of \emph{in-splitting} and \emph{sink deletion} defined by Eckhardt et al. (Canad. J. Math. 2022) preserve the graded $K$-theory of associated Kumjian--Pask algebras and produce algebras which are graded Morita equivalent, thus providing evidence that graded $K$-theory may be an effective invariant for classifying certain Kumjian--Pask algebras.

We also determine a natural sufficient condition regarding the fullness of the graded Grothendieck group functor. More precisely, for two row-finite $k$-graphs $\Lambda$ and $\Omega$ without sources and with finite object sets, we obtain a sufficient criterion for lifting a pointed order-preserving $\mathbb{Z}[\mathbb{Z}^k]$-module homomorphism between $K_0^{\gr}(\KP_\mathsf{F}(\Lambda))$ and $K_0^{\gr}(\KP_\mathsf{F}(\Omega))$ to a unital graded ring homomorphism between $\KP_\mathsf{F}(\Lambda)$ and $\KP_\mathsf{F}(\Omega)$. For this, we adopt, in the setting of $k$-graphs, the \emph{bridging bimodule} technique recently introduced by Abrams, Ruiz and Tomforde (Algebr. Represent. Theory 2024).  
\end{abstract}

\maketitle

\tableofcontents

\section{Introduction}\label{sec intro}
The triangle of symbolic dynamics, operator algebras, and noncommutative algebras has seen a notable surge of activity in recent years. The classical dynamical notions such as conjugacy, shift equivalence  and flow equivalence can be captured  by both analytical and discrete combinatorial algebras (here, graph $C^*$-algebras \cite{KPR} and Leavitt path algebras \cite{Abrams-Monograph, Ara}) arising from underlying dynamics. On the other hand, conjecturally,  these algebras can be classified by certain variants of $K$-theory, thus enabling one to interpret the dynamical behaviour in algebraic terms, and vice versa \cite{willie, eilersxyz}.

To be precise, for a finite graph $E$ with no sinks,  the graded Grothendieck group $K_0^{\gr}$ of the Leavitt path algebra $L_\mathsf{F}(E)$, coincides with the equivariant $K$-theory $K_0^{\mathbb T}$ of the graph $C^*$-algebra $C^*(E)$ and they both coincide with Krieger's dimension group of the shift of finite type associated to the adjacency matrix of the graph $E$ \cite{Hazrat-main, Lind-Marcus}. The long-standing Graded Classification Conjecture states that the pointed pre-ordered $\mathbb Z[x^{-1}, x]$-module $K_0^{\gr}$ classifies the class of Leavitt path algebras (and similarly the graph $C^*$-algebras). A consequence of the conjecture is that two shifts of finite type are eventually conjugate or, shift equivalent if and only if the associated Leavitt path algebras are graded Morita equivalent. 

In an effort to settle the conjecture, it has been shown that the pointed $K_0^{\gr}$ is a full functor for the category of Leavitt path algebras; namely a pointed order-preserving $\mathbb Z[x^{-1},x]$-module homomorphism $\phi:K_0^{\gr}(L_\mathsf{F}(E)) \rightarrow K_0^{\gr}(L_\mathsf{F}(F))$ lifts to a unital graded homomorphism  $L_\mathsf{F}(E)\rightarrow L_\mathsf{F}(F)$ \cite{Arnone, Vas}. And that this functor is invariant under out- and in-splitting graph moves, which gives the graded Morita equivalent algebras~\cite{Cordeiro}. We refer the readers to \cite{ART, AraPardo2014, Arnone, guidohom, guidowillie, Bock, Brix, bilich, eilers, Vas, Vasdisjoint} for recent works on the Graded Classification Conjecture and \cite{willie} for a comprehensive survey. 

This paper aims to initiate a study of graded $K$-theory in the context of a generalised class of Leavitt path algebras, namely those associated with higher-rank graphs (or, $k$-graphs); commonly referred to as Kumjian–Pask algebras~\cite{Pino}. Higher-rank graphs (see Definition \ref{def higher-rank graph} below) were introduced by Kumjian and the third named author in \cite{Kumjian--Pask} to provide a combinatorial framework for higher-rank Cuntz--Krieger algebras \cite{RS}. Soon it appeared that the $C^*$-algebras of higher-rank graphs cover a broader class of $C^*$-algebras, including the well-known graph $C^*$-algebras. Moreover, the $C^*$-algebras and Kumjian--Pask algebras associated to higher-rank graphs include many examples which are not graph $C^*$-algebras and Leavitt path algebras (see \cite{Pask, Pino}); also it is observed recently in \cite{HMPS} that there exists a finitely generated commutative refinement monoid which cannot be realised as a graph monoid but can be realised as the monoid of a row-finite higher-rank graph. Higher-rank graphs, being naturally equipped with delicate categorical structure, often demand innovative and subtle ideas to establish results related to their algebras. 

For a higher-rank $k$-graph $\Lambda$, and its associated Kumjian--Pask algebra $\KP_\mathsf{F}(\Lambda)$ with coefficients in a field $\mathsf{F}$, one can describe the graded Grothendieck group $K_0^{\gr}(\KP_\mathsf{F}(\Lambda))$ directly from the underlying $k$-graph $\Lambda$. The positive cone of this group coincides with the so-called \emph{talented monoid} of $\Lambda$, studied in~\cite{HMPS}. 
Parallel to the theory of Leavitt path algebras, it is natural to ask whether this invariant captures similar structural or categorical properties  in the context of Kumjian–Pask algebras. 
 
As illustrated by Example~\ref{ex non-isomorphic KP-algebras with isomorphic talented monoids}, the $2$-graphs arising from two different factorisations of the following bi-colored skeleton 
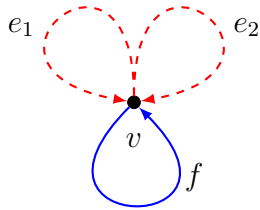
\[
\begin{tikzpicture}[scale=1]
\node[circle,draw,fill=black,inner sep=0.5pt] (p11) at (0, 0) {$.$} 

edge[-latex, red,thick,loop, dashed, in=10, out=90, min distance=70] (p11)
edge[-latex, red,thick, loop, dashed, out=90, in=170, min distance=70] (p11)

edge[-latex, blue,thick, loop, out=225, in=-45, min distance=70] (p11);
                                        
\node at (0, -0.5) {$v$};
\node at (-1.5,1) {$e_1$};
\node at (1.5,1) {$e_2$};
\node at (0.8,-1) {$f$};

\end{tikzpicture}
\]
give rise to two non-graded isomorphic Kumjian--Pask algebras: $L_\mathsf{F}(1,2)[t^{-1},t]$, the Laurent polynomial algebra  with coefficients from the Leavitt path algebra $L_\mathsf{F}(1,2)$ of the two petal rose $R_2$, where $t$ commutes with $L_\mathsf{F}(1,2)$ and $L_\mathsf{F}(1,2)[t^{-1},t;\varphi]$, the skew Laurent polynomial algebra where $\varphi:L_\mathsf{F}(1,2)\rightarrow L_\mathsf{F}(1,2)$ is the ring automorphism switching the petals. However, one can observe that these algebras have the same graded Grothendieck group $\mathbb Z[1/2]$.

This observation naturally suggests the following problem.

\begin{prob}\label{pp11}
Characterise the class of $k$-graphs for which the existence of a (pointed) order-preserving $\mathbb{Z}[\mathbb{Z}^k]$-module isomorphism between the graded Grothendieck groups implies that the respective Kumjian--Pask algebras are graded (isomorphic) Morita equivalent. 
\end{prob}

One should note that, for $1$-graphs, Problem~\ref{pp11} is precisely the Graded Classification Conjecture for Leavitt path algebras. Although the above problem is yet unsolved, it appears, at least from the discrepancy mentioned above, that the $K$-theoretic classification of Kumjian--Pask algebras of higher-rank graphs might be much more involved than that of Leavitt path algebras. This motivates us to investigate how far the graded Grothendieck group can capture the substantial information about this algebra. In this direction, we here accomplish three objectives, each of which contributes to our quest in spite of being independent in their treatment. 

We start by showing that the newly developed moves on $k$-graphs~\cite{EFGGGP}, namely in-splitting and sink deletion, produce Kumjian--Pask algebras that are graded Morita equivalent (see Propositions \ref{pro in-splitting gives graded isomorphic KP algebras} \& \ref{pro sink deletion giving graded Morita equivalent KP-algberas}). We further directly establish in Theorems \ref{th in-splitting preserves talented monoid} and \ref{th sink deletion preserves talented monoid} that the graded $K$-theories of these algebras remain isomorphic under the said moves. As a continuation of \cite{HMPS}, this gives further evidence that the talented monoid may be an effective invariant for classifying certain Kumjian--Pask algebras up to graded Morita equivalence. It is interesting to note that the moves of in- and out-splitting appear extensively in the study of dynamical systems, especially shifts of finite type. It is known that two-dimensional shifts of finite type are best modelled by using textile systems (see \cite{JM}) introduced by Nasu \cite{Nasu}. However, finite $2$-graphs can also be used to form certain shifts of finite type. The connections between $2$-graph in-splitting (\cite{EFGGGP}) and textile in-splitting (introduced by Johnson and Madden \cite{JM}) have been explored very recently in \cite{WoA}, where it is shown that every $2$-graph in-splitting can be broken down into multiple textile in-splittings and inversions; consequently, $2$-graph in-splitting gives conjugacy of the associated one-sided two-dimensional shifts of finite type. In dimension $1$, conjugacy (eventual conjugacy) of edge shifts (essentially shifts of finite type) is characterized by strong shift equivalence (shift equivalence) of the corresponding adjacency matrices. In Remark \ref{rem matrix connection of in-splitting}, we observe that the vertex matrices of a row-finite $k$-graph $\Lambda$ and its in-split $\Lambda_I$ are naturally connected via a certain matrix relation, which looks like a higher-dimensional analogue of shift equivalence on the class of non-negative integral square matrices. This, together with the preceding discussion, gives some hope to derive in future work, some possible matrix conditions describing conjugacy of two-dimensional shifts of finite type as desired at the end of \cite{JM}. 

Our second objective is to realise the graded $K$-theory of Kumjian--Pask algebras as the graded homology \cite{H-Li} of the infinite path groupoid (see \S \ref{ssec path groupoid} below) associated with the underlying $k$-graph. The seminal work of Renault \cite{Renault} showcased topological groupoids as a potential tool to model operator algebras, especially $C^*$-algebras. In the past two decades, through several works (\cite{Bosa, ABBL, Ara-Hazrat, CHR, Clark, H-Li, Kumjian--Pask, Matui, MatAdv} and many more) it has become evident that groupoids play the role of a crucial junction point connecting combinatorial dynamical systems, their $C^*$-algebras, and their algebraic counterparts. The homology theory of \'{e}tale groupoids, introduced by Crainic and Moerdijk \cite{Crainic}, is significant in studying groupoid $C^*$-algebras \cite{Matui, MatAdv}. In this context, Matui proved that for any \'{e}tale groupoid $\mathcal{G}$ arising from a one-sided, one-dimensional shift of finite type, $H_0(\mathcal{G})\cong K_0(C_r^*(\mathcal{G}))$ (see \cite[Theorem 4.14]{Matui}), where $C_r^*(\mathcal{G})$ is the reduced $C^*$-algebra of $\mathcal{G}$. As the majority of interesting groupoids can be naturally graded, an exciting program is to relate graded combinatorial invariants for graded groupoids (e.g., graded homology) to their Steinberg and $C^*$-algebraic counterparts (here graded $K$-theory). In line with this program, we prove here, in two different ways (see Theorem \ref{th relationship of graded K-theory and graded homology} and Remark \ref{rem alternative proof incorporating type semigroup}), that there is an order-preserving $\ZZ[\ZZ^k]$-module isomorphism between the graded zeroth homology $H_0^{\gr}(\mathcal{G}_\Lambda)$ of the infinite path groupoid $\mathcal{G}_\Lambda$ of a row-finite $k$-graph $\Lambda$ without sources and the graded $K$-theory of the corresponding Kumjian--Pask algebra $\KP_\mathsf{F}(\Lambda)$. This extends \cite[Theorem 6.6]{H-Li} to higher dimensions. Moreover, Theorem \ref{th relationship of graded K-theory and graded homology}, together with Remark \ref{rem Connecting K-theory of Kp-algebra and C*-algebra} and \cite[Corollary 3.5 $(i)$]{Kumjian--Pask}, enables us to obtain the following relationships: \[H_0(\mathcal{G}_{\Lambda\times_d \ZZ^k})\cong H_0(\mathcal{G}_\Lambda\times_{\Tilde{d}}\ZZ^k)=H_0^{\gr}(\mathcal{G}_\Lambda)\cong K_0^{\gr}(\KP_\mathsf{F}(\Lambda))\cong K_0(C^*(\Lambda\times_d \ZZ^k))\cong K_0(C_r^*(\mathcal{G}_{\Lambda\times_d \ZZ^k})),\] extending Matui's result mentioned above, to skew-product $k$-graph groupoids which possibly do not arise from shifts of finite type.

The third and final goal of this paper is to seek an answer to the following question:
\begin{ques}\label{ques brief lifting question}
Given two row-finite $k$-graphs $\Lambda$ and $\Omega$ without sources and with finite vertex sets, when can one lift a pointed order-preserving $\ZZ[\ZZ^k]$-module homomorphism between $K_0^{\gr}(\KP_\mathsf{F}(\Lambda))$ and $K_0^{\gr}(\KP_\mathsf{F}(\Omega))$ to a unital graded ring homomorphism between $\KP_\mathsf{F}(\Lambda)$ and $\KP_\mathsf{F}(\Omega)$?
\end{ques}
The above question is framed in categorical terms in Question \ref{ques lifting question}, and is indeed related to the fullness of the graded Grothendieck group functor. The analogue of the above problem on the level of Leavitt path algebras has been solved independently by several authors, viz. Va\v{s} \cite{Vas}, Arnone \cite{Arnone} and Brix et al. \cite{Brix}. The answer is that we can always lift a homomorphism between graded $K$-theories of Leavitt path algebras to a graded algebra homomorphism between the algebras concerned. The authors in \cite{Brix} used the idea of \emph{bridging bimodules}, recently invented by Abrams, Ruiz and Tomforde \cite{ART}, to prove the fullness of the functor $K_0^{\gr}$. Motivated by this approach, in order to find an answer to the lifting question, we develop machinery to adopt the bridging bimodule technique in our higher-rank setting. However, we face some technical hurdles. In a higher-rank graph, a path can be made of edges having different degrees and this simple fact can cause a serious compatibility issue in the process of forming a well-defined conjugacy between two $k$-graphs $\Lambda$ and $\Omega$, especially since any such conjugacy should respect the internal factorization rules of $\Lambda$ and $\Omega$ at the time of intertwining the $\mathsf{F}\Lambda^0-\mathsf{F}\Lambda^0$-bimodule $\mathsf{F}\Lambda^\N$ with the $\mathsf{F}\Omega^0-\mathsf{F}\Omega^0$-bimodule $\mathsf{F}\Omega^\N$ (see diagram (\ref{comd4}) of Lemma \ref{lem towards specified conjugacy}), where $\mathsf{F}$ is a field. To overcome this, we work on the level of polymorphisms and introduce a certain condition on matrices $R$ in $\mathbb{M}_{\Lambda^0\times \Omega^0}(\mathbb{N})$ satisfying $A_{\mathbf{e}_i}R=RB_{\mathbf{e}_i}$ for all $i=1,2,\ldots,k$ ($A_{\mathbf{e}_i}$ and $B_{\mathbf{e}_i}$ are the vertex matrices of $\Lambda$ and $\Omega$ respectively), which is actually a coherence condition taking care of compatibility between the mutual commutativity of the polymorphism of $R$ with the coordinate polymorphisms of $\Lambda$, $\Omega$ and the individual factorization rules of $\Lambda$ and $\Omega$. In Definition \ref{def the bridging matrix and polymorphism}, we call a matrix $R$ satisfying such a coherence condition, a \emph{bridging} matrix. We give enough justification, in Remarks \ref{rem explaining the bridging graph}, for the significance of the bridging criterion and also describe its similarity with the associativity condition required to form a $k$-graph out of a $k$-colored skeleton (see \cite[Theorem 4.4]{Hazlewood}). Though the bridging criterion is vacuously satisfied in dimension $1$, it is not at all obvious while moving to higher dimensions; Examples \ref{ex matrix commutativity is not enough} and \ref{ex positive example of a bridging matrix} justify this for $k=2$. 

Coming back to the lifting question, in Lemma \ref{lem the connecting matrix R} $(ii)$, we show that any pointed order-preserving $\ZZ[\ZZ^k]$-module homomorphism $\mathfrak{h}:K_0^{\gr}(\KP_\mathsf{F}(\Lambda))\longrightarrow K_0^{\gr}(\KP_\mathsf{F}(\Omega))$ gives rise to a matrix $R\in \mathbb{M}_{\Lambda^0\times \Omega^0}(\mathbb{N})$ whose polymorphism $E_R$ makes the diagram
\begin{equation*}\label{comd0}
\begin{tikzcd}[row sep=huge, column sep=huge]
		\Omega^0 \arrow [r, "E_{\mathbf{e}_i}^\Omega"] \arrow [d, "E_R"'] & \Omega^0 \arrow [d, "E_R"] \\
		\Lambda^0 \arrow [r, "E_{\mathbf{e}_i}^\Lambda"] & \Lambda^0
\end{tikzcd}    
\end{equation*}
commute in the category of polymorphisms for each $i=1,2,\ldots,k$, where $E_{\mathbf{e}_i}^\Lambda$, $E_{\mathbf{e}_i}^\Omega$ are the $i^{\text{th}}$-coordinate graphs of $\Lambda$ and $\Omega$ respectively (see \S \ref{sec lifting problem} for more details). In addition, if $R$ is a bridging matrix, then by a step-by-step construction of a bridging bimodule $M(R)$ (a suitable graded $\KP_\mathsf{F}(\Lambda)-\KP_\mathsf{F}(\Omega)$-bimodule), we finally show in Theorem \ref{the lifting theorem}, that there exists a graded homomorphism $\psi:\KP_\mathsf{F}(\Lambda)\longrightarrow \KP_\mathsf{F}(\Omega)$ such that $K_0^{\gr}(\psi)=\mathfrak{h}$. In this way, we establish the bridging criterion as a sufficient condition for lifting a homomorphism on the level of graded $K$-theories to the level of Kumjian--Pask algebras.  

Throughout this paper, all our $k$-graphs are row-finite without sources.

The paper is structured as follows. After this introductory section, we recall in \S \ref{sec basics} the very basic concepts related to higher-rank graphs and Kumjian--Pask algebras. This includes fundamental examples of $k$-graphs (Examples \ref{ex a finite grid graph}, \ref{ex pullback k-graphs}, \ref{ex skew-product k-graphs}), the description of the infinite path groupoid and most importantly a brief overview of the graded $K$-groups of Kumjian--Pask algebras. The preliminary section is followed by \S \ref{sec moves}--\ref{sec lifting problem}, where we collect the main findings of the paper. In \S \ref{sec moves}, we describe the $k$-graph moves of in-splitting and sink-deletion separately in two subsections (see \S \ref{ssec in-slitting} and \S \ref{ssec sink deletion}). We obtain Theorems \ref{th in-splitting preserves talented monoid} and \ref{th sink deletion preserves talented monoid} which state that the talented monoid of a $k$-graph remains $\mathbb{Z}^k$-isomorphic under each of these moves. \S \ref{sec homology} starts with a brief introduction to the homology theory of ample groupoids as well as the graded homology of a $\Gamma$-graded groupoid, where $\Gamma$ is a discrete abelian group. In the rest of this section, we focus on the path groupoid of a $k$-graph $\Lambda$ and prove necessary results (Lemma \ref{lem key for showing well-definedness}, Proposition \ref{pro connecting path group with SA of path groupoid}) leading to Theorem \ref{th relationship of graded K-theory and graded homology}, which exhibits a desired isomorphism between $H_0^{\gr}(\mathcal{G}_\Lambda)$ and $K_0^{\gr}(\KP_\mathsf{F}(\Lambda))$. We also provide some interesting connections between the (graded) homology of $\mathcal{G}_\Lambda$ and that of $\mathcal{G}_{f^*(\Lambda)}$ where $f^*(\Lambda)$ is the pullback of $\Lambda$ with respect to a surjective monoid homomorphism $f:\mathbb{N}^\ell\longrightarrow \mathbb{N}^k$. \S \ref{sec dimension group} is dedicated to the study of the dimension group associated with a $k$-graph $\Lambda$ which is actually another realisation of the graded $K$-theory. The main result of this section is Theorem \ref{th isomorphism of TM in terms of matrices}, which describes the existence of an isomorphism between graded Grothendieck groups of Kumjian--Pask algebras via a specific matrix condition connecting the vertex matrices of the underlying $k$-graphs. In \S \ref{sec lifting problem}, we take up the lifting question (Question \ref{ques lifting question}) and through the introduction of bridging matrices, determine a natural sufficient condition in Theorem \ref{the lifting theorem}, under which the lifting can be done. We conclude the paper in \S \ref{sec future} discussing a few problems which have not been explored here and might be interesting for further research. 

\section{Basic concepts}\label{sec basics}
\subsection{Higher-rank graphs and Kumjian--Pask algebras}\label{ssec k-graphs}
For any positive integer $k$, $\mathbb{N}^k$ stands for the commutative monoid of $k$-tuples of non-negative integers under coordinatewise addition. It is simultaneously viewed as a small category with one object.
\begin{dfn}\label{def higher-rank graph}
A \emph{higher-rank graph} (also called a \emph{$k$-graph}) is a countable small category $\Lambda$ together with a functor $d:\Lambda\longrightarrow \mathbb{N}^k$ (called the \emph{degree functor}) which satisfies the \emph{unique factorization property}: if $\lambda\in \Lambda$ and $d(\lambda)=\mathbf{n}+\mathbf{m}\in \mathbb{N}^k$, then there exist unique $\alpha,\beta\in \Lambda$ such that $d(\alpha)=\mathbf{n},d(\beta)=\mathbf{m}$ and $\lambda=\alpha\beta$.    
\end{dfn} 

The \emph{range} and \emph{source} of any morphism $\lambda\in \Lambda$ are usually denoted as $r(\lambda)$ and $s(\lambda)$. For any $\mathbf{n}\in \mathbb{N}^k$ and $u,v\in \Obj(\Lambda)$, we denote $\Lambda^\mathbf{n}:=d^{-1}(\mathbf{n})$, $u\Lambda^\mathbf{n}:=r^{-1}(u)\cap \Lambda^\mathbf{n}$, $\Lambda^\mathbf{n} v:=\Lambda^\mathbf{n}\cap s^{-1}(v)$ and $u\Lambda^\mathbf{n} v:=u\Lambda^\mathbf{n} \cap \Lambda^\mathbf{n} v$. Also, we write $\Lambda^\mathbf{1}=\displaystyle{\bigcup_{i=1}^{k}}~\Lambda^{\mathbf{e}_i}$ and $v\Lambda^\mathbf{1}=\displaystyle{\bigcup_{i=1}^{k}}~v\Lambda^{\mathbf{e}_i}$, where $\{ \mathbf{e}_i~|~ i = 1 , \ldots , k \}$ is the canonical basis for $\mathbb{N}^k$. Using the factorization property, one can show that $\Obj(\Lambda)=\Lambda^0$. In this paper, we confine ourselves to \emph{row-finite $k$-graphs without sources}, i.e., $k$-graphs $\Lambda$ such that $0<|v\Lambda^\mathbf{n}|< \infty$ for all $v\in \Lambda^0$ and $\N\in \mathbb{N}^k$. Let $\Lambda$ be any $k$-graph and $\lambda\in \Lambda$. Suppose $\mathbf{m},\mathbf{n}\in \mathbb{N}^k$ are such that $0\le \mathbf{m}\le \mathbf{n}\le d(\lambda)$. By applying the unique factorization, we have unique paths $\mu_1,\mu_2,\mu_3\in \Lambda$ such that 
\begin{center}
    $d(\mu_1)=\mathbf{m},d(\mu_2)=\mathbf{n}-\mathbf{m}, d(\mu_3)=d(\lambda)-\mathbf{n}$ and $\lambda=\mu_1\mu_2\mu_3$.
\end{center}

We denote the paths $\mu_1,\mu_2$ and $\mu_3$ by $\lambda(0,\mathbf{m}),\lambda(\mathbf{m},\mathbf{n})$ and $\lambda(\mathbf{n},d(\lambda))$ respectively. Let $\lambda,\mu\in \Lambda$. A path $\tau\in \Lambda$ is called a \emph{minimal common extension} of $\lambda$ and $\mu$ if 

\begin{center}
    $d(\tau)=d(\lambda)\vee d(\mu),$ $\tau(0,d(\lambda))=\lambda$ and $\tau(0,d(\mu))=\mu$.
\end{center}
Define 
\begin{center}
$\MCE(\lambda,\mu):=\{\tau~|~\tau$ is a minimal common extension of $\lambda,\mu\}$.
\end{center}

For a subset $E\subseteq \Lambda$ and $\lambda\in \Lambda$, we denote
\begin{center}
$\Ext(\lambda;E):=\displaystyle{\bigcup_{\mu\in E}} \{\alpha~|~\lambda \alpha=\mu\beta\in \MCE(\lambda,\mu)$ for some $\beta\in \Lambda\}$.
\end{center}

To each $k$-graph $\Lambda$, we can associate an (edge) colored directed graph $G_\Lambda$ called the \emph{$1$-skeleton} of $\Lambda$. Define the set of vertices $G_\Lambda^0:=\Lambda^0$ and set of edges $G_\Lambda^\mathbf{1}:=\Lambda^\mathbf{1}$. Range and source maps are the restrictions of $r$ and $s$ to $\Lambda^\mathbf{1}$. Fix $k$ different colors, each of which corresponds to a generator of $\mathbb{N}^k$. If $\lambda\in G_\Lambda^\mathbf{1}$ has degree $\mathbf{e}_i$ and $c_i$ is the color corresponding to $\mathbf{e}_i$, then we color the edge $\lambda$ by $c_i$. On the other hand, given a $k$-colored directed graph $G$, one can obtain a $k$-graph by defining, on the path category $G^*$, an equivalence relation $\sim$ satisfying some specific properties (see \cite[Theorems 4.4 and  4.5]{Hazlewood}) which encodes the factorisation property. This correspondence between $k$-graphs and $k$-colored directed graphs with certain equivalence relations gives a general mechanism to perform constructions on $k$-graphs and to establish results by focusing on the $1$-skeleton.

\begin{example}\label{ex a finite grid graph}
(\textbf{The $k$-dimensional grid graph}) Let $k$ be any positive integer and $\mathbf{n}\in \mathbb{N}^k$. Denote by $\mathbf{n}^{\downarrow}$, the principal down-set of $\mathbf{n}$ in the poset $\mathbb{N}^k$ (with respect to coordinatewise ordering). Then, $\mathbf{n}^{\downarrow}$ is also a poset. Let $\Omega_{k,\mathbf{n}}$ be the corresponding poset category, i.e., \[\Omega_{k,\mathbf{n}}:=\{(\mathbf{p},\mathbf{q})\in \mathbf{n}^{\downarrow}\times \mathbf{n}^{\downarrow}~|~\mathbf{p}\le \mathbf{q}\},\] with $r((\mathbf{p},\mathbf{q}))=\mathbf{p}$, $s((\mathbf{p},\mathbf{q}))=\mathbf{q}$ and the composition is defined by \[(\mathbf{p},\mathbf{q})(\mathbf{q},\mathbf{t}):=(\mathbf{p},\mathbf{t}).\] Then $\Omega_{k,\mathbf{n}}$ can be viewed as a $k$-graph with respect to the degree functor $d_\mathbf{n}:\Omega_{k,\mathbf{n}}\longrightarrow \mathbb{N}^k$; $d_\mathbf{n}((\mathbf{p},\mathbf{q})):=\mathbf{q}-\mathbf{p}$. Instead of choosing a principal down-set, if we consider the poset $\mathbb{N}^k$ as a whole, then the corresponding poset category, denoted as $\Omega_k$, is also a $k$-graph with the degree functor $d:\Omega_k\longrightarrow \mathbb{N}^k$ defined by $d((\mathbf{p},\mathbf{q})):=\mathbf{q}-\mathbf{p}$ for all $(\mathbf{p},\mathbf{q})\in \Omega_k$. 

Below, we show the $1$-skeleton of $\Omega_{3,(2,2,2)}$. For the three basic $3$-tuples $\mathbf{e}_1,\mathbf{e}_2$ and $\mathbf{e}_3$ we use the colors blue (solid), green (dotted) and red (dashed) respectively. 

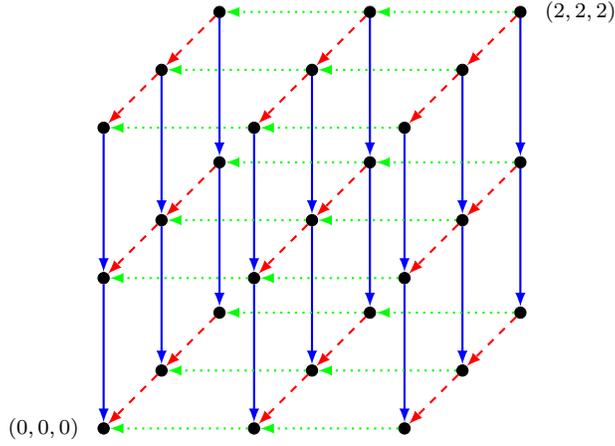
\[
\begin{tikzpicture}[scale=4]
\node[inner sep=1.5pt, circle,draw,fill=black] (A1) at (0,0,0) {};
\node[inner sep=1.5pt, circle,draw,fill=black] (A2) at (1,0,0) {};
\node[inner sep=1.5pt, circle,draw,fill=black] (A3) at (1,0,1) {};
\node[inner sep=1.5pt, circle,draw,fill=black] (A4) at (0,0,1) {};

\node[inner sep=1.5pt, circle,draw,fill=black] (B1) at (0,1,0) {};
\node[inner sep=1.5pt, circle,draw,fill=black] (B2) at (1,1,0) {};
\node[inner sep=1.5pt, circle,draw,fill=black] (B3) at (1,1,1) {};
\node[inner sep=1.5pt, circle,draw,fill=black] (B4) at (0,1,1) {};

\node[inner sep=1.5pt, circle,draw,fill=black] (C1) at (0,1,0.5) {};
\node[inner sep=1.5pt, circle,draw,fill=black] (C2) at (0.5,1,0) {};
\node[inner sep=1.5pt, circle,draw,fill=black] (C3) at (1,1,0.5) {};
\node[inner sep=1.5pt, circle,draw,fill=black] (C4) at (0.5,1,1) {};
\node[inner sep=1.5pt, circle,draw,fill=black] (C5) at (0.5,1,0.5) {};

\node[inner sep=1.5pt, circle,draw,fill=black] (D1) at (0,0,0.5) {};
\node[inner sep=1.5pt, circle,draw,fill=black] (D2) at (0.5,0,0) {};
\node[inner sep=1.5pt, circle,draw,fill=black] (D3) at (1,0,0.5) {};
\node[inner sep=1.5pt, circle,draw,fill=black] (D4) at (0.5,0,1) {};
\node[inner sep=1.5pt, circle,draw,fill=black] (D5) at (0.5,0,0.5) {};

\node[inner sep=1.5pt, circle,draw,fill=black] (E1) at (0,0.5,1) {};
\node[inner sep=1.5pt, circle,draw,fill=black] (E2) at (0,0.5,0) {};
\node[inner sep=1.5pt, circle,draw,fill=black] (E3) at (0,0.5,0.5) {};

\node[inner sep=1.5pt, circle,draw,fill=black] (F1) at (1,0.5,1) {};
\node[inner sep=1.5pt, circle,draw,fill=black] (F2) at (1,0.5,0) {};
\node[inner sep=1.5pt, circle,draw,fill=black] (F3) at (1,0.5,0.5) {};

\node[inner sep=1.5pt, circle,draw,fill=black] (F) at (0.5,0.5,1) {};
\node[inner sep=1.5pt, circle,draw,fill=black] (B) at (0.5,0.5,0) {};
\node[inner sep=1.5pt, circle,draw,fill=black] (G) at (0.5,0.5,0.5) {};

\path[->, red, dashed, >=latex,thick] (B1) edge [] node[]{} (C1);
\path[->, red, dashed, >=latex,thick] (C1) edge [] node[]{} (B4);
\path[->, red, dashed, >=latex,thick] (A1) edge [] node[]{} (D1);
\path[->, red, dashed, >=latex,thick] (D1) edge [] node[]{} (A4);
\path[->, red, dashed, >=latex,thick] (B2) edge [] node[]{} (C3);
\path[->, red, dashed, >=latex,thick] (C3) edge [] node[]{} (B3);
\path[->, red, dashed, >=latex,thick] (A2) edge [] node[]{} (D3);
\path[->, red, dashed, >=latex,thick] (D3) edge [] node[]{} (A3);
\path[->, red, dashed, >=latex,thick] (C2) edge [] node[]{} (C5);
\path[->, red, dashed, >=latex,thick] (C5) edge [] node[]{} (C4);
\path[->, red, dashed, >=latex,thick] (D2) edge [] node[]{} (D5);
\path[->, red, dashed, >=latex,thick] (D5) edge [] node[]{} (D4);
\path[->, red, dashed, >=latex,thick] (E2) edge [] node[]{} (E3);
\path[->, red, dashed, >=latex,thick] (E3) edge [] node[]{} (E1);
\path[->, red, dashed, >=latex,thick] (F2) edge [] node[]{} (F3);
\path[->, red, dashed, >=latex,thick] (F3) edge [] node[]{} (F1);
\path[->, red, dashed, >=latex,thick] (B) edge [] node[]{} (G);
\path[->, red, dashed, >=latex,thick] (G) edge [] node[]{} (F);

\path[->,blue, >=latex,thick] (B1) edge [] node[]{} (E2);
\path[->,blue, >=latex,thick] (E2) edge [] node[]{} (A1);
\path[->,blue, >=latex,thick] (B4) edge [] node[]{} (E1);
\path[->,blue, >=latex,thick] (E1) edge [] node[]{} (A4);
\path[->,blue, >=latex,thick] (B2) edge [] node[]{} (F2);
\path[->,blue, >=latex,thick] (F2) edge [] node[]{} (A2);
\path[->,blue, >=latex,thick] (B3) edge [] node[]{} (F1);
\path[->,blue, >=latex,thick] (F1) edge [] node[]{} (A3);
\path[->,blue, >=latex,thick] (C1) edge [] node[]{} (E3);
\path[->,blue, >=latex,thick] (E3) edge [] node[]{} (D1);
\path[->,blue, >=latex,thick] (C3) edge [] node[]{} (F3);
\path[->,blue, >=latex,thick] (F3) edge [] node[]{} (D3);
\path[->,blue, >=latex,thick] (C4) edge [] node[]{} (F);
\path[->,blue, >=latex,thick] (F) edge [] node[]{} (D4);
\path[->,blue, >=latex,thick] (C2) edge [] node[]{} (B);
\path[->,blue, >=latex,thick] (B) edge [] node[]{} (D2);
\path[->,blue, >=latex,thick] (C5) edge [] node[]{} (G);
\path[->,blue, >=latex,thick] (G) edge [] node[]{} (D5);

\path[->,green, dotted, >=latex,thick] (B2) edge [] node {} (C2);
\path[->,green,  dotted, >=latex,thick] (C2) edge [] node {} (B1);
\path[->,green,  dotted, >=latex,thick] (B3) edge [] node {} (C4);
\path[->,green,  dotted, >=latex,thick] (C4) edge [] node {} (B4);
\path[->,green,  dotted, >=latex,thick] (C3) edge [] node {} (C5);
\path[->,green,  dotted, >=latex,thick] (C5) edge [] node {} (C1);
\path[->,green,  dotted, >=latex,thick] (A2) edge [] node {} (D2);
\path[->,green,  dotted, >=latex,thick] (D2) edge [] node {} (A1);
\path[->,green,  dotted, >=latex,thick] (A3) edge [] node {} (D4);
\path[->,green,  dotted, >=latex,thick] (D4) edge [] node {} (A4);
\path[->,green,  dotted, >=latex,thick] (D3) edge [] node {} (D5);
\path[->,green,  dotted, >=latex,thick] (D5) edge [] node {} (D1);
\path[->,green,  dotted, >=latex,thick] (F2) edge [] node {} (B);
\path[->,green,  dotted, >=latex,thick] (B) edge [] node {} (E2);
\path[->,green,  dotted, >=latex,thick] (F1) edge [] node {} (F);
\path[->,green,  dotted, >=latex,thick] (F) edge [] node {} (E1);
\path[->,green,  dotted, >=latex,thick] (F3) edge [] node {} (G);
\path[->,green,  dotted, >=latex,thick] (G) edge [] node {} (E3);

\node at (-0.2,0,1) {\tiny $(0,0,0)$};
\node at (1.2,1,0) {\tiny $(2,2,2)$};

\end{tikzpicture}
\]
\end{example}

\begin{example}\label{ex pullback k-graphs}
(\textbf{Pullback of a $k$-graph}) Suppose we have a $k$-graph $\Lambda$ with degree functor $d$. Let $\ell$ be any positive integer. A common way to produce an $\ell$-graph out of $\Lambda$ is to consider a monoid homomorphism $f:\mathbb{N}^\ell \longrightarrow \mathbb{N}^k$. Set \[f^*(\Lambda):=\{(\lambda,\N)~|~d(\lambda)=f(\N)\},\] \[s((\lambda,\N)):=(s(\lambda),0),~~r((\lambda,\N)):=(r(\lambda),0),~~ d^*((\lambda,\N)):=\N.\] Define a composition as \[(\lambda,\N)(\mu,\M):=(\lambda\mu,\N+\M).\]
Then $(f^*(\Lambda),d^*)$ is an $\ell$-graph, which is often called the \emph{pullback of $\Lambda$ along $f$}  (see \cite[Definition 1.9]{Kumjian--Pask}). 

Pulling back is a standard procedure to transit between higher-rank graphs of different dimensions. Below, we provide two simple instances; in each case, we only draw the skeletons of the $k$-graphs and mention the relevant monoid homomorphism along which the pullback is performed. 

$(i)$ $1$-graph $\rightsquigarrow$ $2$-graph:

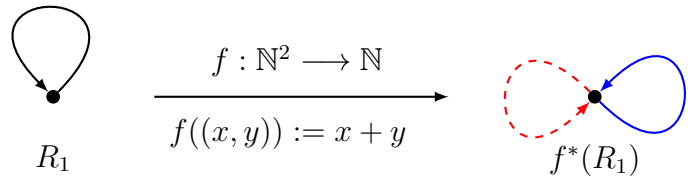
\[
\begin{tikzpicture}[scale=1.2]

\node[circle,draw,fill=black,inner sep=0.5pt] (p11) at (0, 0) {$.$} 
edge[-latex,thick,loop, out=45, in=135, min distance=50] (p11);

\node[circle,draw,fill=black,inner sep=0.5pt] (p11) at (6, 0) {$.$} 
edge[-latex, blue,thick,loop, out=-45, in=45, min distance=50] (p11)
edge[-latex, red,thick, loop, dashed, out=135, in=225, min distance=50] (p11);
                                                            
\node at (0, -0.68) {$R_1$};
\node at (6, -0.7) {$f^*(R_1)$};

\node[] (A) at (1,0) {};
\node[] (B) at (4.5,0) {};

\path[->,black, >=latex,thick] (A) edge [] node[]{} (B);

\node at (2.7,0.4) {$f:\mathbb{N}^2\longrightarrow \mathbb{N}$};
\node at (2.6, -0.4) {$f((x,y)):=x+y$};

\end{tikzpicture}
\]

$(ii)$ $2$-graph $\rightsquigarrow$ $1$-graphs: 
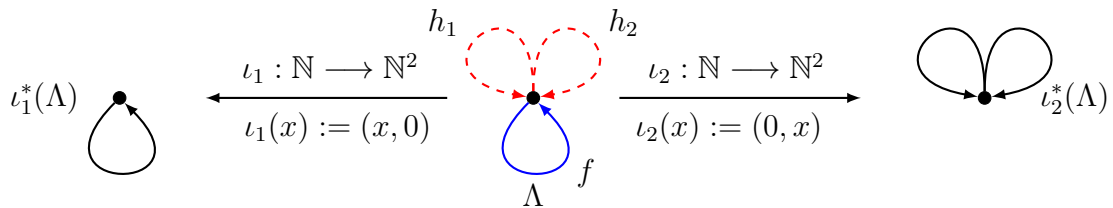
\[
\begin{tikzpicture}[scale=1]

\node[circle,draw,fill=black,inner sep=0.5pt] (p11) at (0, 0) {$.$} 

edge[-latex, red,thick,loop, dashed, in=10, out=90, min distance=50] (p11)
edge[-latex, red,thick, loop, dashed, out=90, in=170, min distance=50] (p11)

edge[-latex, blue,thick, loop, out=225, in=-45, min distance=50] (p11);
                                          
\node at (-1.2,1) {$h_1$};
\node at (1.2,1) {$h_2$};
\node at (0.7,-1) {$f$};

\node[circle,draw,fill=black,inner sep=0.5pt] (p11) at (-5.5, 0) {$.$} 
edge[-latex, thick,loop, out=225, in=-45, min distance=50] (p11);

\node[circle,draw,fill=black,inner sep=0.5pt] (p11) at (6, 0) {$.$} 
edge[-latex, thick,loop, in=10, out=90, min distance=50] (p11)
edge[-latex, thick, loop, out=90, in=170, min distance=50] (p11);

\node[] (A) at (1,0) {};
\node[] (B) at (4.5,0) {};

\path[->,black, >=latex,thick] (A) edge [] node[]{} (B);

\node at (2.7,0.4) {$\iota_2:\mathbb{N}\longrightarrow \mathbb{N}^2$};
\node at (2.6, -0.4) {$\iota_2(x):=(0,x)$};

\node[] (C) at (-1,0) {};
\node[] (D) at (-4.5,0) {};

\path[->,black, >=latex,thick] (C) edge [] node[]{} (D);

\node at (-2.7,0.4) {$\iota_1:\mathbb{N}\longrightarrow \mathbb{N}^2$};
\node at (-2.6, -0.4) {$\iota_1(x):=(x,0)$};

\node at (0, -1.3) {$\Lambda$};
\node at (7.2, 0) {$\iota_2^*(\Lambda)$};
\node at (-6.5, 0) {$\iota_1^*(\Lambda)$};
\end{tikzpicture}
\]
The edges $h_1,h_2$ in $\Lambda$ have degree $(0,1)$ and the edge $f$ has degree $(1,0)$. The factorizations of paths of degree $(1,1)$ are given as $h_1 f=f h_1$, $h_2 f=f h_2$. 
\end{example}

In Example \ref{ex pullback k-graphs} $(ii)$, the two directed graphs $\iota_1^*(\Lambda)$ and $\iota_2^*(\Lambda)$ are known as the coordinate graphs of the $2$-graph $\Lambda$. More generally, given a $k$-graph $\Lambda$, the \emph{$i^{\text{th}}$ coordinate graph} associated with $\Lambda$ is a directed graph $E_{\mathbf{e}_i}^\Lambda$ with set of vertices $\Lambda^0$, set of edges $\Lambda^{\mathbf{e}_i}$, and range and source maps are the restrictions of $r,s$ to $\Lambda^{\mathbf{e}_i}$.

\begin{example}\label{ex skew-product k-graphs}
(\textbf{Skew-product $k$-graph}) Let $\Gamma$ be a discrete abelian group and $\Lambda$ a $k$-graph. Given a functor $c : \Lambda \to \Gamma$, we may form the \emph{skew-product $k$-graph} $\Lambda \times_c \Gamma$, which is the set $\Lambda \times \Gamma$ endowed with structure maps given by $r((\lambda, \gamma)) := (r(\lambda), \gamma)$, $s((\lambda, \gamma)) :=(s(\lambda), c(\lambda) + \gamma)$, product $(\lambda, \gamma)(\mu, c(\lambda)+ \gamma ) := (\lambda\mu, \gamma)$, and $d ( \lambda, \gamma):= d(\lambda)$ (see \cite[Definition 5.1]{Kumjian--Pask}). There is a natural $\Gamma$-action $\alpha$ on $\Lambda \times_c \Gamma$ given by $\alpha_\kappa( \lambda , \gamma) := ( \lambda
, \kappa + \gamma )$. Note that the degree functor $d:\Lambda\longrightarrow \mathbb{N}^k$ can be regarded as a functor from $\Lambda$ to the abelian group $\ZZ^k$. We often denote the corresponding skew-product $k$-graph $\Lambda\times_d \ZZ^k$ by $\overline{\Lambda}$. For example, if $\Lambda$ is the $2$-graph $f^*(R_1)$ of Example \ref{ex pullback k-graphs} $(i)$, then $\overline{\Lambda}$ is isomorphic to the infinite $2$-graph $\Delta_2$ (see \cite[Example 2.2 $(iv)$]{HMPS}), whose definition is similar to that of the $2$-dimensional grid graph $\Omega_2$ of Example \ref{ex a finite grid graph}. 
\end{example}

Let us now recall the definition of \emph{Kumjian--Pask algebras} associated with row-finite $k$-graphs with no sources. 
\begin{dfn}\label{def KP-algebra}
Following \cite{Pino}, let $\Lambda^{\neq 0}:=\{\lambda\in \Lambda~|~d(\lambda)\neq 0\}=\Lambda\setminus \Lambda^0$ and $G(\Lambda^{\neq 0}):=\{\lambda^*~|~\lambda\in \Lambda^{\neq 0}\}$. Let $\mathsf{F}$ be any field. The \emph{Kumjian--Pask} algebra of $\Lambda$ over $\mathsf{F}$ is denoted by $\KP_\mathsf{F}(\Lambda)$ and is defined to be the universal associative $\mathsf{F}$-algebra generated by symbols $p_v,s_\lambda,s_{\lambda^*}$; $v\in \Lambda^0,\lambda\in \Lambda^{\neq 0}$ subject to the relations:

$(KP1)$ $p_up_v=\delta_{u,v}p_u$ for all $u,v\in \Lambda^0$;

$(KP2)$ $s_\lambda s_\mu=s_{\lambda\mu},$ $s_{\mu^*}s_{\lambda^*}=s_{(\lambda\mu)^*}$ for all $\lambda,\mu\in \Lambda^{\neq 0}$ with $s(\lambda)=r(\mu)$, and

$p_{r(\lambda)}s_\lambda=s_\lambda=s_\lambda p_{s(\lambda)},$ $p_{s(\lambda)}s_{\lambda^*}=s_{\lambda^*}=s_{\lambda^*}p_{r(\lambda)}$ for all $\lambda\in \Lambda^{\neq 0}$;

$(KP3)$ $s_{\lambda^*} s_\mu=\delta_{\lambda,\mu}p_{s(\lambda)}$ for all $\lambda,\mu\in \Lambda^{\neq 0}$ with $d(\lambda)=d(\mu)$;

$(KP4)$ $p_v=\displaystyle{\sum_{\lambda\in v\Lambda^\N}} s_\lambda s_{\lambda^*}$ for all $v\in \Lambda^0$ and for all $\N\in \mathbb{N}^k\setminus \{0\}$.
\end{dfn}

The Kumjian--Pask algebra is a higher-rank analogue of the Leavitt path algebra: if $E$ is a directed graph, then $E^*$ is a $1$-graph, where the roles of the range and source maps are reversed; hence $L_\mathsf{F} (E) = \KP_{\mathsf{F}} ( E^* )$.

$\KP_{\mathsf{F}} ( \Lambda )$ is canonically $\mathbb{Z}^k$-graded, where the $n^{\text{th}}$ homogeneous component is defined as: \[\KP_\mathsf{F}(\Lambda)_\N:=\lspan_\mathsf{F} \{s_\alpha s_{\beta^*}~|~\alpha,\beta\in \Lambda, s(\alpha)=s(\beta)~\text{and}~d(\alpha)-d(\beta)=\N\}.\]

\subsection{The infinite path groupoid}\label{ssec path groupoid}
We assume that the readers are well familiar with the notion of a groupoid. Briefly, a \emph{groupoid} $\mathcal{G}$ is a small category consisting of isomorphisms only. The \emph{range} and \emph{source} of a morphism $x\in \mathcal{G}$ are respectively denoted as $r(x)$ and $s(x)$. The set of objects is denoted by $\mathcal{G}^{(0)}$, usually called the \emph{unit space} and the set of \emph{composable} pairs is denoted by $\mathcal{G}^{(2)}$, i.e., \[\mathcal{G}^{(2)}:=\{(x,y)\in \mathcal{G}\times \mathcal{G}~|~s(x)=r(y)\}.\] As we shall be concerned with the graded homology of a certain \'{e}tale groupoid in Section \ref{sec homology}, we recall the necessary terminologies related to topological groupoids.

A groupoid $\mathcal{G}$ is called a \emph{topological groupoid} if it is endowed with a topology such that the partial multiplication and the inversion are continuous. By an \emph{\'{e}tale} groupoid, here we mean a locally compact Hausdorff topological groupoid $\mathcal{G}$ such that the range map $r$ (equivalently, the source map $s$) is a local homeomorphism, i.e., for each $x\in \mathcal{G}$, there is an open neighbourhood $U$ of $x$ such that $r|_U:U\longrightarrow r(U)$ is a homeomorphism. A subset $B$ of a groupoid $\mathcal{G}$ is called a \emph{bisection} if the restrictions $r|_B$ and $s|_B$ are injective. An \emph{ample} groupoid is an \'{e}tale groupoid $\mathcal{G}$ whose topology has a basis consisting of compact open bisections of $\mathcal{G}$. In an ample groupoid $\mathcal{G}$, the set of all compact open bisections is denoted by $\mathcal{G}^a$ and it forms an inverse semigroup with respect to the product \[AB:=\{xy~|~(x,y)\in (A\times B)\cap \mathcal{G}^{(2)}\},\] and the inversion \[A^{-1}:=\{x^{-1}~|~x\in A\}.\] For any subset $X$ of the unit space $\mathcal{G}^{(0)}$, the \emph{restriction} of $\mathcal{G}$ to $X$ is denoted by $\mathcal{G}|_X$ and is defined as $\mathcal{G}|_X:=\{x\in \mathcal{G}~|~s(x),r(x)\in X\}$. The subset $X$ is called \emph{$\mathcal{G}$-full} if for any $u\in \mathcal{G}^{(0)}$, there exists $x\in \mathcal{G}$ such that $r(x)=u$ and $s(x)\in X$. 

Let $\Gamma$ be a group with discrete topology. A topological groupoid $\mathcal{G}$ is said to be $\Gamma$-\emph{graded} if there is a continuous functor $c:\mathcal{G}\longrightarrow \Gamma$, which is usually referred to as a \emph{cocycle} on $\mathcal{G}$. A subset $X$ of $\mathcal{G}$ is called \emph{homogeneous} if $X\subseteq c^{-1}(\gamma)$ for some $\gamma\in \Gamma$. The set of all homogeneous compact open bisections of $\mathcal{G}$ is denoted by $\mathcal{G}_h^a$. 

The \'{e}tale groupoid we are concerned with here is the infinite path groupoid associated with a $k$-graph introduced in \cite{Kumjian--Pask}. We now recall its construction. Let $\Lambda$ be a row-finite $k$-graph with no sources. An infinite path of $\Lambda$ is a degree preserving functor $x:\Omega_k\longrightarrow \Lambda$. The set of infinite paths is denoted by $\Lambda^\infty$. For any $x\in \Lambda^\infty$, the \emph{range} of $x$ is defined as $r(x):=x(0)$. For $\mathbf{p}\in \mathbb{N}^k$ and $x\in \Lambda^\infty$, the $\mathbf{p}$-shifted infinite path $\sigma^\mathbf{p}(x)$ is defined as:
\begin{center}
    $\sigma^\mathbf{p}(x)((\M,\N)):=x((\M+\mathbf{p},\N+\mathbf{p}))$ for all $(\M,\N)\in \Omega_k$.
\end{center} 
Define \begin{center}
    $\mathcal{G}_\Lambda:=\{(x,\N-\M,y)\in \Lambda^\infty \times \mathbb{Z}^k \times \Lambda^\infty~|~\N,\M\in \mathbb{N}^k$ and $\sigma^\N (x)=\sigma^\M(y)\}$,
\end{center}
and $\mathcal{G}_\Lambda^{(0)}:=\Lambda^\infty$; the range and source maps are defined as $r((x,\N,y)):=x$ and $s((x,\N,y)):=y$. Two arrows $(x,\N,y)$ and $(z,\M,w)$ are composable if and only if $y=z$ and in that case
\begin{center}
    $(x,\N,y)(y,\M,w):=(x,\N+\M,w)$.
\end{center}
Define $(x,\N,y)^{-1}:=(y,-\N,x)$. Then, $\mathcal{G}_\Lambda$ is a groupoid which is known as the \emph{infinite path groupoid} of $\Lambda$. 

For $\lambda,\mu\in \Lambda$ with $s(\lambda)=s(\mu)$, define
\begin{center}
    $Z(\lambda,\mu):=\{(\lambda z, d(\lambda)-d(\mu), \mu z)~|~z\in s(\lambda)\Lambda^\infty\}$.
\end{center}
Then $\mathcal{B}:=\{Z(\lambda,\mu)~|~ \lambda,\mu\in \Lambda, s(\lambda)=s(\mu)\}$ forms a basis for an \'{e}tale topology on $\mathcal{G}_\Lambda$ \cite[Proposition 2.8]{Kumjian--Pask}. Moreover, each $Z(\lambda,\mu)$ is a compact open bisection so $\mathcal{G}_\Lambda$ is an ample groupoid. If we identify $x\in \Lambda^\infty$ with the arrow $(x,0,x)$, then the sets $Z(\lambda):=Z(\lambda,\lambda)$ form a basis of compact open sets for the subspace topology on $\mathcal{G}_\Lambda^{(0)}$. The groupoid $\mathcal{G}_\Lambda$ is very useful as it provides groupoid models for the $C^*$-algebra and the Kumjian--Pask algebra of $\Lambda$ (see \cite{Kumjian--Pask, Clark}).

\subsection{Graded $K$-theory of Kumjian--Pask algebras}\label{ssec graded K-theory}

Let $A$ be a $\Gamma$\!-graded ring, where $\Gamma$ is an abelian group.  The category of  finitely generated
$\Gamma$\!-graded  projective right $A$-modules is denoted by $\Pgr[\Gamma] A$. This 
is an exact category with the usual notion of a (split) short exact
sequence. Thus, one can apply Quillen's $Q$-construction~\cite{quillen} to obtain $K$-groups
\[K_i(\Pgr[\Gamma] A),\] for $i\ge 0$, which we denote by $K_i^{\gr}(A)$. 
 The group $\Gamma$ acts on the category
$\Pgr[\Gamma] A$ from the right via $(P,\alpha) \longmapsto P(\alpha)$, where $\alpha \in \Gamma$. By the functoriality of $K$-groups, this gives $K_i^{\gr}(A)$ the structure of a right $\mathbb Z[\Gamma]$-module~\cite{hazgrothen}. For a trivial group $\Gamma$, one retrieves the $K$-theory groups $K_i(A)$.

For $i=0$, $K^{\gr}_0(A)$ is called the graded Grothendieck group of $A$, which coincides with the group completion of the commutative monoid $\mathcal V^{\gr}(A)$, the monoid of isomorphism classes of  finitely generated $\Gamma$-graded projective (right) $A$-modules. Again, for a trivial group, the group completion of $\mathcal V(A)$ gives the Grothendieck group $K_0(A)$.  

In the setting of Leavitt path algebras, for a graph $E$, one can realise $L_\mathsf{F}(E)$ as a Bergman algebra, which in turn allows one to  describe $\mathcal V(L_\mathsf{F}(E))$ and consequently, $K_0(L_\mathsf{F}(E))$ from the underlying graph $E$~\cite{Ara}. 

The class of Kumjian–Pask algebras that can be realised as Bergman algebras remains unknown.
Therefore, there is not yet a systematic way to describe the Grothendieck group $K_0(\KP_\mathsf{F}(\Lambda))$ for a $k$-graph $\Lambda$ (see Problem~\ref{prob For which k-graphs KP-algebras are Bergman algebras}). However, one can describe the commutative $\Gamma$-monoid $\mathcal V^{\gr}(\KP_\mathsf{F}(\Lambda))$ and the graded Grothendieck group  $K^{\gr}_0(\KP_\mathsf{F}(\Lambda))$. This is possible as we have the following equivalence of categories 
\[\Gr\!-\KP_\mathsf{F}(\Lambda) \sim \Mod\!-\KP_\mathsf{F}(\overline \Lambda),\]
where $\overline \Lambda=\Lambda \times_d \mathbb Z^k$ is the skew-product higher-rank graph (see Example~\ref{ex skew-product k-graphs}). One then can show that $\KP_\mathsf{F}(\overline \Lambda)$ is an ultramatricial algebra, which in turn allows one to compute the lower $K$-theory (see~\cite{Ara-Hazrat}). In fact, $\mathcal{V}^{\gr}(\KP_\mathsf{F}(\Lambda))$ can be realised effectively from the geometry of $\Lambda$.  Recall from \cite{HMPS} that the \emph{talented monoid} $T_\Lambda$ is the $k$-graph monoid of the skew-product $k$-graph $\overline{\Lambda}:=\Lambda\times_d \mathbb{Z}^k$. More formally, $T_\Lambda$ can be defined as the commutative monoid freely generated by the symbols $v(\N)$, $v\in \Lambda^0,\N\in \mathbb{Z}^k$ subject to the following relations: 
\begin{center}
    $v(\N)=\displaystyle{\sum_{\alpha\in v\Lambda^{\mathbf{e}_i}}} s(\alpha)(\N+\mathbf{e}_i)$,
\end{center}
for all $i=1,2, \ldots,k$ (see \cite[Definition 3.9]{HMPS}). As a joint consequence of \cite[Corollary 6.7]{Ara-Hazrat} and \cite[Proposition 3.14 $(i)$]{HMPS}, it follows that $\mathcal{V}^{\gr}(\KP_\mathsf{F}(\Lambda))\cong T_\Lambda$ as $\mathbb{Z}^k$-monoids and, therefore, the Grothendieck group completion of $T_\Lambda$ is $K_0^{\gr}(\KP_\mathsf{F}(\Lambda))$, the graded $K$-theory of $\KP_\mathsf{F}(\Lambda)$. In view of this and due to the simpler description of $T_\Lambda$ over $\mathcal{V}^{\gr}(\KP_\mathsf{F}(\Lambda))$, it is quite helpful to deal with $T_\Lambda$ if one wants to study the graded $K$-theory of $\KP_\mathsf{F}(\Lambda)$. This motivates our program of classifying Kumjian--Pask algebras by $K$-theoretic invariants.

\section{Invariance of graded $K$-theory under some $k$-graph moves}\label{sec moves}
In \cite{HMPS}, some Morita invariant properties of Kumjian--Pask algebras, e.g., being simple, semisimple and the graded Morita invariant property of being graded ideal simple are characterised via properties of the talented monoid that are preserved under $\mathbb{Z}^k$-monoid isomorphisms. It suggests that talented monoid may be a possible classifying invariant for certain Kumjian--Pask algebras up to graded Morita equivalence. In this section we gather further evidence in support of this by directly showing that the talented monoid remains isomorphic under certain \emph{moves} on $k$-graphs preserving graded Morita equivalence of Kumjian--Pask algebras. 

Among the four moves introduced in \cite{EFGGGP}, we only focus on \emph{in-splitting} and \emph{sink deletion}. The other two moves, namely \emph{delay} and \emph{reduction} (see \cite{EFGGGP} to know how these work), in general, do not preserve the talented monoid. For example, if we consider $2$-graphs $\Sigma$, $\Lambda$ and $\Gamma$ with the following $1$-skeletons:
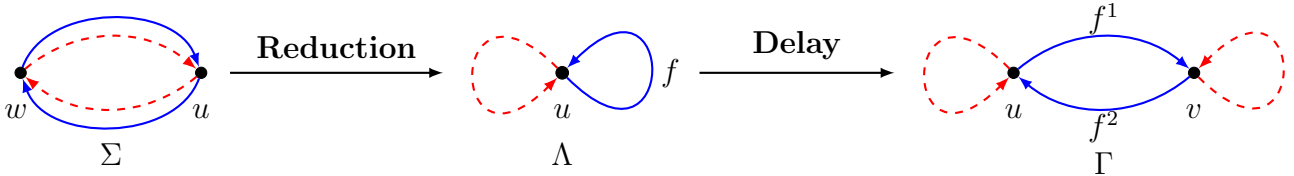
\[
\begin{tikzpicture}[scale=1.2]

\node[circle,draw,fill=black,inner sep=0.5pt] (p11) at (0, 0) {$.$} 
edge[-latex, blue,thick,loop, out=-45, in=45, min distance=50] (p11)
edge[-latex, red,thick, loop, dashed, out=135, in=225, min distance=50] (p11);
                                         
\node at (0, -0.42) {$u$};

\node[inner sep=1.5pt, circle,draw,fill=black] (A) at (5,0) {}
edge[-latex, red,thick, loop, dashed, out=135, in=225, min distance=50] (A);

\node[inner sep=1.5pt, circle,draw,fill=black] (B) at (7,0) {}
edge[-latex, red,thick, loop, dashed, out=-45, in=45, min distance=50] (B);

\node[inner sep=1.5pt, circle,draw,fill=black] (C) at (-4,0) {};	
\node[inner sep=1.5pt, circle,draw,fill=black] (D) at (-6,0) {};

\path[->, red, dashed, >=latex,thick] (C) edge [bend left=40] node[above=0.05cm]{} (D);
\path[->, red, dashed, >=latex,thick] (D) edge [bend left=40] node[above=0.05cm]{} (C);
\path[->,blue, >=latex,thick] (C) edge [bend left=70] node[below=0.05cm]{} (D);
\path[->,blue, >=latex,thick] (D) edge [bend left=70] node[below=0.05cm]{} (C);

\path[->, blue, >=latex,thick] (A) edge [bend left=40] node[above=0.05cm]{} (B);
\path[->, blue, >=latex,thick] (B) edge [bend left=40] node[above=0.05cm]{} (A);

\node[] (P) at (-3.8,0) {};
\node[] (Q) at (-1.2,0) {};

\path[->,black, >=latex,thick] (P) edge [] node[above=0.05cm]{\textbf{Reduction}} (Q);

\node[] (R) at (1.4,0) {};
\node[] (S) at (3.8,0) {};

\path[->,black, >=latex,thick] (R) edge [] node[above=0.03cm]{\textbf{Delay}} (S);

\node at (5,-0.42) {$u$};
\node at (7,-0.42) {$v$};

\node at (-4,-0.42) {$u$};
\node at (-6.05,-0.42) {$w$};

\node at (1.2,0) {$f$};
\node at (6,0.6) {$f^1$};
\node at (6,-0.6) {$f^2$};

\node at (-5,-0.9) {$\Sigma$};
\node at (0,-0.9) {$\Lambda$};
\node at (6,-1) {$\Gamma$};

\end{tikzpicture}
\]
then $\Lambda$ is the reduction of $\Sigma$ at the vertex $w$ and $\Gamma$ is obtained from $\Lambda$ by delaying the edge $f$. One can easily show that $T_\Lambda\cong \mathbb{N}$, whereas both $T_\Sigma$ and $T_\Gamma$ are isomorphic to $\mathbb{N}\oplus \mathbb{N}$. Since graded Morita equivalent rings have order-isomorphic graded $K$-groups and hence isomorphic graded $\mathcal{V}$-monoids, the above example shows that the moves of reduction and delay may fail to preserve graded Morita equivalence of Kumjian--Pask algebras. 

\subsection{In-splitting at a single vertex}\label{ssec in-slitting}
In this subsection, we review the move of \emph{in-splitting} a $k$-graph $\Lambda$ introduced by Eckhardt et al. \cite[Section 3]{EFGGGP} as the analogue of out-splitting of directed graphs. 

Let $v\in \Lambda^0$. The main idea of in-splitting $\Lambda$ at $v$ is to partition $v\Lambda^\mathbf{1}$ into two nonempty subsets $\mathcal{E}_1,\mathcal{E}_2$ by keeping in mind the \emph{pairing condition}: if $\lambda,\mu\in v\Lambda^\mathbf{1}$ are such that $\lambda \alpha=\mu\beta$ for some $\alpha,\beta\in \Lambda^\mathbf{1}$ then $\lambda$ and $\mu$ should be in the same partition set. If $\lambda\neq \mu$ and $d(\lambda)=d(\mu)$ then $\MCE(\lambda,\mu)=\emptyset$ and so the condition puts no restriction on $\lambda,\mu$; they may or may not be kept in the same partition set. 

Suppose $v\Lambda^\mathbf{1}=\mathcal{E}_1\cup \mathcal{E}_2$ is a partition satisfying the pairing condition. The \emph{in-split} of $\Lambda$ at $v$ is a $k$-graph $\Lambda_I$, where the set of vertices is \[\Lambda_I^0:=(\Lambda^0\setminus \{v\})\cup \{v^1,v^2\},\] and the set of edges is \[\Lambda_I^\mathbf{1}:=(\Lambda^\mathbf{1}\setminus \Lambda^\mathbf{1} v)\cup \{f^1,f^2~|~f\in \Lambda^\mathbf{1} v\}.\] The range and source maps are defined on edges as follows:
$$r_I(\alpha):=
	\left\{
	\begin{array}{llll}
		r(\alpha)  & \mbox{if } \alpha\in \Lambda^\mathbf{1}\setminus (\Lambda^\mathbf{1}v \cup v\Lambda^\mathbf{1}), \\
		v^i  & \mbox{if } \alpha\in v\Lambda^\mathbf{1}\setminus \Lambda^\mathbf{1}v ~\mbox{and}~ \alpha\in \mathcal{E}_i;~ i=1,2,\\
            r(f)  & \mbox{if } \alpha=f^i ~\mbox{and}~ r(f)\neq v,\\
        v^j & \mbox{if } \alpha=f^i ~\mbox{and}~ f\in \mathcal{E}_j;~ j=1,2,
	\end{array}
	\right.$$
$$s_I(\alpha):=
	\left\{
	\begin{array}{ll}
		s(\alpha)  & \mbox{if } \alpha\in \Lambda^\mathbf{1}\setminus \Lambda^\mathbf{1}v, \\
		v^i  & \mbox{if } \alpha=f^i;~i=1,2,
	\end{array}
	\right.$$
and then extended to all paths. Similarly, the degree map $d_I$ is defined on edges as 
$$d_I(\alpha):=
	\left\{
	\begin{array}{ll}
		d(\alpha)  & \mbox{if } \alpha\in \Lambda^\mathbf{1}\setminus \Lambda^\mathbf{1}v, \\
		d(f)  & \mbox{if } \alpha=f^i;~i=1,2,
	\end{array}
	\right.$$
and then additively extended to any path. That $(\Lambda_I,d_I)$ is indeed a $k$-graph, is established in \cite[Theorem 3.8]{EFGGGP}. To keep track of the new vertices and edges originated in the in-splitting process, it is helpful to assign a \emph{parent} to each member of $\Lambda_I$. This is done by defining $par(v^i):=v$, $par(f^i):=f$ for $i=1,2$ and $par(x):=x$ if $x$ is neither $v^i$ nor $f^i$. Also, if $\lambda=\lambda_1\lambda_2\cdots \lambda_n$, then $par(\lambda):=par(\lambda_1)par(\lambda_2)\cdots par(\lambda_n)$. The factorization rule for paths in $\Lambda_I$ is then determined by the factorization of their parent in $\Lambda$. 

In \cite[Theorem 3.9]{EFGGGP}, it was shown that there is a gauge-action preserving isomorphism between the $C^*$-algebra of a $k$-graph $\Lambda$ and the $C^*$-algebra of its in-split $\Lambda_I$ at a vertex $v$. The following can be seen as a purely algebraic analogue of this result. 
\begin{prop}\label{pro in-splitting gives graded isomorphic KP algebras}
Let $\Lambda$ be a row-finite $k$-graph with no sources and let $\Lambda_I$ be the in-split of $\Lambda$ at a vertex $v$ with respect to the partition $v\Lambda^\mathbf{1}=\mathcal{E}_1\cup \mathcal{E}_2$. Let $\mathsf{F}$ be any field. Then, there is a graded isomorphism 
\[
\pi:\KP_\mathsf{F}(\Lambda)\longrightarrow \KP_\mathsf{F}(\Lambda_I).
\] 
Consequently, $\KP_\mathsf{F}(\Lambda)$ and $\KP_\mathsf{F}(\Lambda_I)$ are graded Morita equivalent.  
\end{prop}
\begin{proof}
Suppose $(q,t)$ and $(p,s)$ denote the Kumjian--Pask families generating $\KP_\mathsf{F}(\Lambda)$ and $\KP_\mathsf{F}(\Lambda_I)$ respectively. Define \[P_u:=\displaystyle{\sum_{\substack{w\in \Lambda_I^0\\par(w)=u}}} p_w,~S_\lambda:=\displaystyle{\sum_{\substack{\alpha\in \Lambda_I\\par(\alpha)=\lambda}}} s_\alpha,~\text{and}~S_{\lambda^*}:=\displaystyle{\sum_{\substack{\alpha\in \Lambda_I\\par(\alpha)=\lambda}}} s_{\alpha^*}\] for all $u\in \Lambda^0$ and $\lambda\in \Lambda^{\neq 0}$. We now prove that $\{P_u,S_\lambda,S_{\lambda^*}\}$ forms a Kumjian--Pask $\Lambda$-family in the $\mathsf{F}$-algebra $\KP_\mathsf{F}(\Lambda_I)$. Since each $P_u$ is a sum of mutually orthogonal idempotents and distinct vertices in $\Lambda$ cannot have a common offspring, so $P_u$'s are also mutually orthogonal idempotents in $\KP_\mathsf{F}(\Lambda_I)$. This shows that $(P,S)$ satisfies $(KP1)$. Let $\lambda,\mu\in \Lambda^{\neq 0}$ be such that $s(\lambda)=r(\mu)$. If $r(\mu)\neq v$, then for all $\alpha,\beta\in \Lambda_I$ with $par(\alpha)=\lambda$, $par(\beta)=\mu$, we have $r_I(\beta)=r(\mu)=s(\lambda)=s_I(\alpha)$ and so
\begin{align*}
    S_\lambda S_\mu=\left(\displaystyle{\sum_{par(\alpha)=\lambda}} s_\alpha\right)\left(\displaystyle{\sum_{par(\beta)=\mu}} s_\beta\right) &=\displaystyle{\sum_{\substack{par(\alpha)=\lambda\\par(\beta)=\mu}}} s_\alpha s_\beta\\
    &= \displaystyle{\sum_{\substack{par(\alpha)=\lambda\\par(\beta)=\mu}}} s_{\alpha\beta}\\
    &= \displaystyle{\sum_{par(\gamma)=\lambda\mu}} s_\gamma=S_{\lambda\mu}.
\end{align*}
The second last equality follows since there is a bijection between the sets \[\{\alpha\in \Lambda_I~|~par(\alpha)=\lambda\}\times \{\beta\in \Lambda_I~|~par(\beta)=\mu\}\] and \[\{\gamma\in \Lambda_I~|~par(\gamma)=\lambda\mu\}\] defined by $(\alpha,\beta)\longmapsto \alpha\beta$. If $r(\mu)=v$, then $\mu=g\eta$ for some $g\in v\Lambda^\mathbf{1}$. Without loss of generality, assume that $g\in \mathcal{E}_1$. Then, for any $\beta\in \Lambda_I$ with $par(\beta)=\mu$, we have $r_I(\beta)=v^1$. In this case, \[S_\lambda S_\mu=\left(\displaystyle{\sum_{par(\alpha)=\lambda}} s_\alpha\right)\left(\displaystyle{\sum_{par(\beta)=\mu}} s_\beta\right)=\displaystyle{\sum_{\substack{par(\alpha)=\lambda,s_I(\alpha)=v^1\\par(\beta)=\mu}}} s_{\alpha\beta}=\displaystyle{\sum_{par(\gamma)=\lambda\mu}} s_\gamma=S_{\lambda\mu}.\] Similarly, one can show that $S_{\lambda^*}S_{\mu^*}=S_{(\lambda\mu)^*}$ for all $\lambda,\mu\in \Lambda$ with $r(\mu)=s(\lambda)$. The other relations in $(KP2)$ are straightforward to verify. For $(KP3)$, choose $\lambda,\mu\in \Lambda^{\neq 0}$ such that $\lambda\neq \mu$ and $d(\lambda)=d(\mu)$. Then, $\alpha\neq \beta$ for all $\alpha,\beta$ with $par(\alpha)=\lambda$ and $par(\beta)=\mu$. Since $d_I(\alpha)=d(\lambda)=d(\mu)=d_I(\beta)$, we have \[S_{\lambda^*}S_\mu=\left(\displaystyle{\sum_{par(\alpha)=\lambda}} s_{\alpha^*}\right)\left(\displaystyle{\sum_{par(\beta)=\mu}} s_\beta\right)=\displaystyle{\sum_{\substack{par(\alpha)=\lambda\\par(\beta)=\mu}}} s_{\alpha^*}s_\beta=0.\] On the other hand, \[S_{\lambda^*}S_\lambda=\left(\displaystyle{\sum_{par(\alpha)=\lambda}} s_{\alpha^*}\right)\left(\displaystyle{\sum_{par(\beta)=\alpha}} s_\beta\right)=\displaystyle{\sum_{par(\alpha)=\lambda}} s_{\alpha^*}s_\alpha.\] If $s(\lambda)\neq v$, then for all $\alpha\in \Lambda_I$ with $par(\alpha)=\lambda$, $s_I(\alpha)=s(\lambda)$. By \cite[Remark 3.7]{EFGGGP}, there is exactly one such $\alpha$ and hence $S_{\lambda^*}S_\lambda=s_{\alpha^*}s_\alpha=p_{s_I(\alpha)}=P_{s(\lambda)}$. If $s(\lambda)=v$, then we can write $\lambda=\delta f$, where $f\in \Lambda^\mathbf{1} v$. We observe that in this case, there are exactly two offspring of $\lambda$ in $\Lambda_I$. Indeed, if $par(\alpha)=\lambda$, then $\alpha=\xi h$ for some $\xi\in \Lambda_I$, $h\in \Lambda_I^1$ such that $par(\xi)=\delta$ and $par(h)=f$. Clearly, $h$ is either $f^1$ or $f^2$. In both cases 
$$s_I(\xi):=
	\left\{
	\begin{array}{lll}
		r(f) & \mbox{if } r(f)\neq v, \\
            v^1  & \mbox{if } f\in \mathcal{E}_1, \\
            v^2  & \mbox{if } f\in \mathcal{E}_2.
	\end{array}
	\right.$$ 
Again in view of \cite[Remark 3.7]{EFGGGP}, there is a unique $\xi$ with the above properties, and we have exactly two offspring of $\lambda$, namely $\alpha_1:=\xi f^1$ and $\alpha_2:=\xi f^2$. Now, \[S_{\lambda^*}S_\lambda=\displaystyle{\sum_{par(\alpha)=\lambda}} s_{\alpha^*}s_\alpha=s_{\alpha_1^*}s_{\alpha_1}+s_{\alpha_2^*}s_{\alpha_2}=p_{v^1}+p_{v^2}=P_{s(\lambda)}.\] Hence, $(KP3)$ is verified. Finally, for $(KP4)$, take $u\in \Lambda^0$ and $\N\in \mathbb{N}^k\setminus \{0\}$. If $\alpha\neq \beta$ and $par(\alpha)=par(\beta)$, then $s_I(\alpha)\neq s_I(\beta)$ and consequently, $s_\alpha s_{\beta^*}=0$. Therefore, \[\displaystyle{\sum_{\lambda\in u\Lambda^\N}} S_\lambda S_{\lambda^*}=\displaystyle{\sum_{\lambda\in u\Lambda^\N}} \left(\displaystyle{\sum_{par(\alpha)=\lambda}} s_\alpha s_{\alpha^*}\right).\] If $u\neq v$, then $u\Lambda_I^\N=\displaystyle{\bigsqcup_{\lambda\in u\Lambda^\N}} \{\alpha\in \Lambda_I~|~par(\alpha)=\lambda\}$, and so \[\displaystyle{\sum_{\lambda\in u\Lambda^\N}} S_\lambda S_{\lambda^*}=\displaystyle{\sum_{\alpha\in u\Lambda_I^\N}} s_\alpha s_{\alpha^*}=p_u=P_u.\] Again, 
\begin{align*}
    \displaystyle{\sum_{\lambda\in v\Lambda^\N}} S_\lambda S_{\lambda^*} &=\displaystyle{\sum_{\lambda\in v\Lambda^\N}} \left(\displaystyle{\sum_{par(\alpha)=\lambda}} s_\alpha s_{\alpha^*}\right)\\
    &= \displaystyle{\sum_{\substack{\lambda\in v\Lambda^\N\\\lambda(0,\mathbf{e}_i)\in \mathcal{E}_1}}} \left(\displaystyle{\sum_{par(\alpha)=\lambda}} s_\alpha s_{\alpha^*}\right)+\displaystyle{\sum_{\substack{\lambda\in v\Lambda^\N\\\lambda(0,\mathbf{e}_i)\in \mathcal{E}_2}}} \left(\displaystyle{\sum_{par(\beta)=\lambda}} s_\beta s_{\beta^*}\right)~(\text{since}~\N\neq 0, \N\ge \mathbf{e}_i~\text{for some}~i=1,2,\ldots,k)\\
    &= \displaystyle{\sum_{\alpha\in v^1 \Lambda_I^\N}} s_\alpha s_{\alpha^*}+\displaystyle{\sum_{\beta\in v^2\Lambda_I^\N}} s_\beta s_{\beta^*}\\
    &= p_{v^1}+p_{v^2}=P_v.
\end{align*}
Thus, $(P,S)$ satisfies $(KP4)$ and becomes a Kumjian--Pask $\Lambda$-family in $\KP_\mathsf{F}(\Lambda_I)$. By \cite[Theorem 3.4]{Pino}, we have a unique $k$-algebra homomorphism $\pi:\KP_\mathsf{F}(\Lambda)\longrightarrow \KP_\mathsf{F}(\Lambda_I)$ such that \[\pi(q_w)=P_w,~\pi(t_\lambda)=S_\lambda~\text{and}~\pi(t_{\lambda^*})=S_{\lambda^*},\] for all $w\in \lambda^0$ and $\lambda\in \Lambda^{\neq 0}$. Since $d_I(\alpha)=d(par(\alpha))$ for all $\alpha\in \Lambda_I$, $\pi$ preserves the degrees of generators of $\KP_\mathsf{F}(\Lambda)$ and is a graded homomorphism. Note that for all $w\in \Lambda^0$, $\pi(q_w)$ being a sum of vertex idempotents in $\KP_\mathsf{F}(\Lambda_I)$ is never $0$. By invoking the graded uniqueness theorem \cite[Theorem 4.1]{Pino}, we can conclude that $\pi$ is injective. For any $w\in \Lambda^0\setminus \{v\}$, $p_w=P_w=\pi(q_w)$. Choose any $j\in \{1,2,\ldots,k\}$. Then, for $i=1,2$, \[p_{v^i}=\displaystyle{\sum_{\alpha\in v^i\Lambda_I^{\mathbf{e}_j}}} s_\alpha s_{\alpha^*}=\displaystyle{\sum_{\lambda\in v\Lambda^{\mathbf{e}_j}\cap \mathcal{E}_i}} \left(\displaystyle{\sum_{par(\alpha)=\lambda}} s_\alpha s_{\alpha^*}\right)=\displaystyle{\sum_{\lambda\in v\Lambda^{\mathbf{e}_j}\cap \mathcal{E}_i}}S_\lambda S_{\lambda^*}=\pi\left(\displaystyle{\sum_{\lambda\in v\Lambda^{\mathbf{e}_j}\cap \mathcal{E}_i}} t_\lambda t_{\lambda^*}\right).\] Let $\alpha\in \Lambda_I^{\neq 0}$. If $s(par(\alpha))\neq v$, then $s_\alpha=S_{par(\alpha)}=\pi(t_{par(\alpha)})$ and $s_{\alpha^*}=\pi(t_{par(\alpha)^*})$. If $s(par(\alpha))=v$, then $\alpha=\mu f^i$ for some $\mu\in \Lambda_I, f\in \Lambda^\mathbf{1} v$ and $i\in \{1,2\}$. Subsequently, $s_\alpha=S_{par(\alpha)}p_{v^i}=\pi(t_{par(\alpha)})p_{v^i}\in \image(\pi)$ since $p_{v^i}\in \image(\pi)$. Similarly, $s_{\alpha^*}\in \image(\pi)$. Thus, all the generators of $\KP_\mathsf{F}(\Lambda_I)$ are in $\image(\pi)$, which proves that $\pi$ is surjective. Therefore, $\pi$ is a graded isomorphism of $\mathbb{Z}^k$-graded $\mathsf{F}$-algebras. 
\end{proof}

As an immediate consequence of the above proposition, it follows that whenever we in-split a $k$-graph $\Lambda$ at a vertex $v$, then there is an isomorphism between the categories $\KP_\mathsf{F}(\Lambda)-\Gr$ and $\KP_\mathsf{F}(\Lambda_I)-\Gr$ (as any graded left $\KP_\mathsf{F}(\Lambda_I)$-module can be realised canonically as a graded left $\KP_\mathsf{F}(\Lambda)$-module via the graded isomorphism $\pi$ of Proposition \ref{pro in-splitting gives graded isomorphic KP algebras}). This category isomorphism commutes with the respective shift functors $\tau_\N$ for every $\N\in \mathbb{Z}^k$ and hence, induces a $\mathbb{Z}^k$-monoid isomorphism from $\mathcal{V}^{\gr}(\KP_\mathsf{F}(\Lambda))$ to $\mathcal{V}^{\gr}(\KP_\mathsf{F}(\Lambda_I))$; the $\mathbb{Z}^k$-action on $\mathcal{V}^{\gr}(\KP_\mathsf{F}(\Lambda))$ is given as $~^\N [P]:=[P(\N)]$. Therefore, in-splitting results in $\mathbb{Z}^k$-monoid isomorphic talented monoids. However, we can prove this independently without referring to the Kumjian--Pask algebras, which justifies that the talented monoid directly captures geometries of a $k$-graph that are linked with the graded Morita invariant properties of the corresponding Kumjian--Pask algebra. 

The directed graph version of the following theorem was obtained in \cite[Theorem 4.7]{Cordeiro}.
\begin{thm}\label{th in-splitting preserves talented monoid}
Let $\Lambda$ be a row-finite $k$-graph with no sources and $\Lambda_I$ the in-split of $\Lambda$ at a vertex $v$ with respect to the partition $v\Lambda^\mathbf{1}=\mathcal{E}_1\cup \mathcal{E}_2$. The map $\phi_I: T_\Lambda \longrightarrow T_{\Lambda_I}$ defined on generators as
$$\phi_I(w(\N)):=
	\left\{
	\begin{array}{ll}
		w(\N)  & \mbox{if } w\neq v, \\
		v^1(\N)+v^2(\N)  & \mbox{if } w=v
	\end{array}
	\right.$$ 
for all $w\in \Lambda^0$ and $\N\in \mathbb{Z}^k$, is a $\mathbb{Z}^k$-monoid isomorphism. 
\end{thm}
\begin{proof}
We first show that $\phi_I$ is well-defined. For this, we need to show that $\phi_I$ preserves the relations \[u(\N)=\displaystyle{\sum_{\lambda\in u\Lambda^{\mathbf{e}_i}}} s(\lambda)(\N+\mathbf{e}_i),\] for all $u\in \Lambda^0$, $\N\in \mathbb{Z}^k$ and $i=1,2,\ldots,k$. Note that \[\phi_I\left(\displaystyle{\sum_{\lambda\in u\Lambda^{\mathbf{e}_i}}} s(\lambda)(\N+\mathbf{e}_i)\right)=\displaystyle{\sum_{\lambda\in u\Lambda^{\mathbf{e}_i}}}\phi_I(s(\lambda)(\N+\mathbf{e}_i))=\displaystyle{\sum_{\substack{\lambda\in u\Lambda^{\mathbf{e}_i}\\s(\lambda)\neq v}}} s(\lambda)(\N+\mathbf{e}_i)+\displaystyle{\sum_{\lambda\in u\Lambda^{\mathbf{e}_i}v}} (v^1(\N+\mathbf{e}_i)+v^2(\N+\mathbf{e}_i)).\] If $u\neq v$, then it implies 
{\allowdisplaybreaks
\begin{align*}
&\phi_I\left(\displaystyle{\sum_{\lambda\in u\Lambda^{\mathbf{e}_i}}} s(\lambda)(\N+\mathbf{e}_i)\right)\\
=& \displaystyle{\sum_{\substack{\alpha\in u\Lambda_I^{\mathbf{e}_i}\\s(par(\alpha))\neq v}}} s_I(\alpha)(\N+\mathbf{e}_i)+\displaystyle{\sum_{\beta\in u\Lambda_I^{\mathbf{e}_i}v^1}} s_I(\beta)(\N+\mathbf{e}_i)+\displaystyle{\sum_{\gamma\in u\Lambda_I^{\mathbf{e}_i}v^2}} s_I(\gamma)(\N+\mathbf{e}_i)\\
=& \displaystyle{\sum_{h\in u\Lambda_I^{\mathbf{e}_i}}} s_I(h)(\N+\mathbf{e}_i)\\
=&~ u(\N)=\phi_I(u(\N)),
\end{align*}
}
whereas 
\noindent
{\allowdisplaybreaks
\begin{align*}
&\phi_I\left(\displaystyle{\sum_{\lambda\in v\Lambda^{\mathbf{e}_i}}} s(\lambda)(\N+\mathbf{e}_i)\right)\\
=& \displaystyle{\sum_{\substack{\lambda\in v\Lambda^{\mathbf{e}_i}\cap \mathcal{E}_1\\s(\lambda)\neq v}}} s(\lambda)(\N+\mathbf{e}_i)+\displaystyle{\sum_{\substack{\lambda\in v\Lambda^{\mathbf{e}_i}\cap \mathcal{E}_2\\s(\lambda)\neq v}}} s(\lambda)(\N+\mathbf{e}_i)+\displaystyle{\sum_{\lambda\in v\Lambda^{\mathbf{e}_i}v\cap \mathcal{E}_1}}(v^1(\N+\mathbf{e}_i)+v^2(\N+\mathbf{e}_i))\\
&+\displaystyle{\sum_{\lambda\in v\Lambda^{\mathbf{e}_i}v\cap \mathcal{E}_2}} (v^1(\N+\mathbf{e}_i)+v^2(\N+\mathbf{e}_i))\\
=& \left(\displaystyle{\sum_{\substack{\lambda\in v\Lambda^{\mathbf{e}_i}\cap \mathcal{E}_1\\s(\lambda)\neq v}}} s(\lambda)(\N+\mathbf{e}_i)+\displaystyle{\sum_{\lambda\in v\Lambda^{\mathbf{e}_i}v \cap \mathcal{E}_1}}v^1(\N+\mathbf{e}_i)+\displaystyle{\sum_{\lambda\in v\Lambda^{\mathbf{e}_i}v \cap \mathcal{E}_1}}v^2(\N+\mathbf{e}_i)\right)\\
&+\left(\displaystyle{\sum_{\substack{\lambda\in v\Lambda^{\mathbf{e}_i}\cap \mathcal{E}_2\\s(\lambda)\neq v}}} s(\lambda)(\N+\mathbf{e}_i)+\displaystyle{\sum_{\lambda\in v\Lambda^{\mathbf{e}_i}v \cap \mathcal{E}_2}}v^1(\N+\mathbf{e}_i)+\displaystyle{\sum_{\lambda\in v\Lambda^{\mathbf{e}_i}v \cap \mathcal{E}_2}}v^2(\N+\mathbf{e}_i)\right)\\
=& \left(\displaystyle{\sum_{\substack{\alpha\in v^1\Lambda_I^{\mathbf{e}_i}\\s(par(\alpha))\neq v}}} s_I(\alpha)(\N+\mathbf{e}_i)+\displaystyle{\sum_{\beta\in v^1\Lambda_I^{\mathbf{e}_i}v^1}} s_I(\beta)(\N+\mathbf{e}_i)+\displaystyle{\sum_{\gamma\in v^1\Lambda_I^{\mathbf{e}_i}v^2}} s_I(\gamma)(\N+\mathbf{e}_i)\right)\\
&+\left(\displaystyle{\sum_{\substack{\alpha'\in v^2\Lambda_I^{\mathbf{e}_i}\\s(par(\alpha'))\neq v}}} s_I(\alpha')(\N+\mathbf{e}_i)+\displaystyle{\sum_{\beta'\in v^2\Lambda_I^{\mathbf{e}_i}v^1}} s_I(\beta')(\N+\mathbf{e}_i)+\displaystyle{\sum_{\gamma'\in v^2\Lambda_I^{\mathbf{e}_i}v^2}} s_I(\gamma')(\N+\mathbf{e}_i)\right)\\
=& \displaystyle{\sum_{g\in v^1\Lambda_I^{\mathbf{e}_i}}} s_I(g)(\N+\mathbf{e}_i)+\displaystyle{\sum_{h\in v^2\Lambda_I^{\mathbf{e}_i}}} s_I(h)(\N+\mathbf{e}_i)\\
=&~ v^1(\N)+v^2(\N)=\phi_I(v(\N)).
\end{align*}
} 
Thus, $\phi_I$ is a well-defined monoid homomorphism. We now define an inverse for $\phi_I$. For this, we fix $j\in \{1,2,\ldots,k\}$ and define a map $\psi:T_{\Lambda_I}\longrightarrow T_\Lambda$ on generators by 
$$\psi(u(\N)):=
	\left\{
	\begin{array}{ll}
		\hspace{1.25cm} u(\N)  & \mbox{if } u\neq v^1,v^2, \\
		\displaystyle{\sum_{f\in \mathcal{E}_i\cap \Lambda^{\mathbf{e}_j}}} s(f)(\N+\mathbf{e}_j)  & \mbox{if } u=v^i~\mbox{for}~i=1,2.
	\end{array}
	\right.$$ 
Again we first need to check that $\psi$ is well-defined. If $u\neq v^1,v^2$, then for each $i=1,2,\ldots,k$ and $\N\in \mathbb{Z}^k$, we have 
{\allowdisplaybreaks
\begin{align*}
&\psi\left(\displaystyle{\sum_{\alpha\in u\Lambda_I^{\mathbf{e}_i}}} s_I(\alpha)(\N+\mathbf{e}_i)\right)\\
=& \displaystyle{\sum_{\alpha\in u\Lambda_I^{\mathbf{e}_i}}} \psi(s_I(\alpha)(\N+\mathbf{e}_i))\\
=& \displaystyle{\sum_{\substack{\alpha\in u\Lambda_I^{\mathbf{e}_i}\\s_I(\alpha)\neq v^1,v^2}}} s_I(\alpha)(\N+\mathbf{e}_i)+\displaystyle{\sum_{\alpha\in u\Lambda_I^{\mathbf{e}_i}v^1}} \left(\displaystyle{\sum_{f\in \mathcal{E}_1\cap \Lambda^{\mathbf{e}_j}}} s(f)(\N+\mathbf{e}_i+\mathbf{e}_j)\right)+\displaystyle{\sum_{\alpha\in u\Lambda_I^{\mathbf{e}_i}v^2}} \left(\displaystyle{\sum_{f\in \mathcal{E}_2\cap \Lambda^{\mathbf{e}_j}}} s(f)(\N+\mathbf{e}_i+\mathbf{e}_j)\right)\\
=& \displaystyle{\sum_{\substack{\alpha\in u\Lambda_I^{\mathbf{e}_i}\\par(\alpha)\neq v}}}s(par(\alpha))(\N+\mathbf{e}_i)+\displaystyle{\sum_{\lambda\in u\Lambda^{\mathbf{e}_i}v}}\left(\displaystyle{\sum_{f\in \mathcal{E}_1\cap \Lambda^{\mathbf{e}_j}}} s(f)(\N+\mathbf{e}_i+\mathbf{e}_j)+\displaystyle{\sum_{f\in \mathcal{E}_2\cap \Lambda^{\mathbf{e}_j}}} s(f)(\N+\mathbf{e}_i+\mathbf{e}_j)\right)\\
=& \displaystyle{\sum_{\substack{\lambda\in u\Lambda^{\mathbf{e}_i}\\s(\lambda)\neq v}}} s(\lambda)(\N+\mathbf{e}_i)+\displaystyle{\sum_{\lambda\in u\Lambda^{\mathbf{e}_i}v}}\left(\displaystyle{\sum_{f\in v\Lambda^{\mathbf{e}_j}}} s(f)(\N+\mathbf{e}_i+\mathbf{e}_j)\right)\\
=& \displaystyle{\sum_{\substack{\lambda\in u\Lambda^{\mathbf{e}_i}\\s(\lambda)\neq v}}}s(\lambda)(\N+\mathbf{e}_i)+\displaystyle{\sum_{\lambda\in u\Lambda^{\mathbf{e}_i}v}} v(\N+\mathbf{e}_i) ~~(\text{using relations of}~ T_\Lambda)\\
=& \displaystyle{\sum_{\lambda\in u\Lambda^{\mathbf{e}_i}}} s(\lambda)(\N+\mathbf{e}_i)=u(\N)=\psi(u(\N)).
\end{align*}
For $t=1,2$ 
\begin{align*}
&\psi\left(\displaystyle{\sum_{\alpha\in v^t\Lambda_I^{\mathbf{e}_i}}} s_I(\alpha)(\N+\mathbf{e}_i)\right)\\
=&  \displaystyle{\sum_{\substack{\alpha\in v^t\Lambda_I^{\mathbf{e}_i}\\s_I(\alpha)\neq v^1,v^2}}} s_I(\alpha)(\N+\mathbf{e}_i)+\displaystyle{\sum_{\alpha\in v^t\Lambda_I^{\mathbf{e}_i}v^1}} \left(\displaystyle{\sum_{f\in \mathcal{E}_1\cap \Lambda^{\mathbf{e}_j}}} s(f)(\N+\mathbf{e}_i+\mathbf{e}_j)\right)+\displaystyle{\sum_{\alpha\in v^t\Lambda_I^{\mathbf{e}_i}v^2}} \left(\displaystyle{\sum_{f\in \mathcal{E}_2\cap \Lambda^{\mathbf{e}_j}}} s(f)(\N+\mathbf{e}_i+\mathbf{e}_j)\right)\\
=& \displaystyle{\sum_{\substack{\alpha\in v^t\Lambda_I^{\mathbf{e}_i}\\par(\alpha)\neq v}}}s(par(\alpha))(\N+\mathbf{e}_i)+\displaystyle{\sum_{\lambda\in v\Lambda^{\mathbf{e}_i}v\cap \mathcal{E}_t}}\left(\displaystyle{\sum_{f\in \mathcal{E}_1\cap \Lambda^{\mathbf{e}_j}}} s(f)(\N+\mathbf{e}_i+\mathbf{e}_j)+\displaystyle{\sum_{f\in \mathcal{E}_2\cap \Lambda^{\mathbf{e}_j}}} s(f)(\N+\mathbf{e}_i+\mathbf{e}_j)\right)\\
=& \displaystyle{\sum_{\substack{\lambda\in v\Lambda^{\mathbf{e}_i}\cap \mathcal{E}_t\\s(\lambda)\neq v}}} s(\lambda)(\N+\mathbf{e}_i)+\displaystyle{\sum_{\lambda\in v\Lambda^{\mathbf{e}_i}v\cap \mathcal{E}_t}}\left(\displaystyle{\sum_{f\in v\Lambda^{\mathbf{e}_j}}} s(f)(\N+\mathbf{e}_i+\mathbf{e}_j)\right)\\
=& \displaystyle{\sum_{\substack{\lambda\in v\Lambda^{\mathbf{e}_i}\cap \mathcal{E}_t\\s(\lambda)\neq v}}}s(\lambda)(\N+\mathbf{e}_i)+\displaystyle{\sum_{\lambda\in v\Lambda^{\mathbf{e}_i}v\cap \mathcal{E}_t}} v(\N+\mathbf{e}_i) ~~(\text{using relations of}~ T_\Lambda)\\
=& \displaystyle{\sum_{\lambda\in \mathcal{E}_t \cap \Lambda^{\mathbf{e}_i}}} s(\lambda)(\N+\mathbf{e}_i).
\end{align*}
}
The second equality above follows since the parent of any $\alpha\in v^t\Lambda_I^{\mathbf{e}_i}v^1$ (or $v^t\Lambda_I^{\mathbf{e}_i}v^2$) is a loop at $v$ of degree $\mathbf{e}_i$, contained in the partition set $\mathcal{E}_t$; and hence both $v^t\Lambda_I^{\mathbf{e}_i}v^1$ and $v^t\Lambda_I^{\mathbf{e}_i}v^2$ are in bijection with $v\Lambda^{\mathbf{e}_i}v \cap \mathcal{E}_t$. The crucial part which remains to prove in order to complete the verification is that the equality
\begin{equation}\label{welldefined}
\displaystyle{\sum_{\lambda\in \mathcal{E}_t \cap \Lambda^{\mathbf{e}_i}}} s(\lambda)(\N+\mathbf{e}_i)=\displaystyle{\sum_{\mu\in \mathcal{E}_t \cap \Lambda^{\mathbf{e}_j}}} s(\mu)(\N+\mathbf{e}_j)    
\end{equation}
holds in $T_\Lambda$ for each $i=1,2,\ldots,k$. There is nothing to show if $i=j$. So assume that $i\neq j$. Then, in $T_\Lambda$,
\begin{align*}
\displaystyle{\sum_{\lambda\in \mathcal{E}_t\cap \Lambda^{\mathbf{e}_i}}} s(\lambda)(\N+\mathbf{e}_i)&=\displaystyle{\sum_{\lambda\in \mathcal{E}_t\cap \Lambda^{\mathbf{e}_i}}} \left(\displaystyle{\sum_{\alpha\in s(\lambda)\Lambda^{\mathbf{e}_j}}} s(\alpha)(\N+\mathbf{e}_i+\mathbf{e}_j)\right)\\
&= \displaystyle{\sum_{\substack{\gamma\in v\Lambda^{\mathbf{e}_i+\mathbf{e}_j}\\\gamma(0,\mathbf{e}_i)\in \mathcal{E}_t}}} s(\gamma)(\N+\mathbf{e}_i+\mathbf{e}_j)\\
&= \displaystyle{\sum_{\substack{\gamma\in v\Lambda^{\mathbf{e}_i+\mathbf{e}_j}\\\gamma(0,\mathbf{e}_j)\in \mathcal{E}_t}}} s(\gamma)(\N+\mathbf{e}_i+\mathbf{e}_j)\\
&= \displaystyle{\sum_{\mu\in \mathcal{E}_t\cap \Lambda^{\mathbf{e}_j}}} \left(\displaystyle{\sum_{\beta\in s(\mu)\Lambda^{\mathbf{e}_i}}} s(\beta)(\N+\mathbf{e}_j+\mathbf{e}_i)\right)=\displaystyle{\sum_{\mu\in \mathcal{E}_t\cap \Lambda^{\mathbf{e}_j}}} s(\mu)(\N+\mathbf{e}_j).
\end{align*}
The second and fourth equalities are self-explanatory. For the third one, note that whenever $\gamma=\lambda\alpha\in v\Lambda^{\mathbf{e}_i+\mathbf{e}_j}$ and $\lambda\in \mathcal{E}_t\cap \Lambda^{\mathbf{e}_i}$, by the factorization property of $\Lambda$, there exist unique $\mu\in v\Lambda^{\mathbf{e}_j}$ and $\beta\in s(\mu)\Lambda^{\mathbf{e}_i}$ such that $\gamma=\mu\beta$. Then, by the pairing condition, $\lambda$ and $\mu$ are in the same partition set, whence $\mu\in \mathcal{E}_t$. Thus $\gamma(0,\mathbf{e}_i)\in \mathcal{E}_t$ if and only if $\gamma(0,\mathbf{e}_j)\in \mathcal{E}_t$. 

Now using $(\ref{welldefined})$, we have \[\psi\left(\displaystyle{\sum_{\alpha\in v^t\Lambda_I^{\mathbf{e}_i}}} s_I(\alpha)(\N+\mathbf{e}_i)\right)=\displaystyle{\sum_{\lambda\in \mathcal{E}_t\cap \Lambda^{\mathbf{e}_i}}} s(\lambda)(\N+\mathbf{e}_i)=\displaystyle{\sum_{\mu\in \mathcal{E}_t\cap \Lambda^{\mathbf{e}_j}}} s(\mu)(\N+\mathbf{e}_j)=\psi(v^t(\N)),\] showing that $\psi$ is well-defined. To complete the proof, we observe that $\phi_I$ and $\psi$ are mutually inverse homomorphisms. If $w\neq v$ then $\psi(\phi_I(w(\N)))=\psi(w(\N))=w(\N)$. Again, \[\psi(\phi_I(v(\N)))=\psi(v^1(\N)+v^2(\N))=\displaystyle{\sum_{f\in \mathcal{E}_1\cap \Lambda^{\mathbf{e}_j}}} s(f)(\N+\mathbf{e}_j)+\displaystyle{\sum_{g\in \mathcal{E}_2\cap \Lambda^{\mathbf{e}_j}}} s(g)(\N+\mathbf{e}_j)=\displaystyle{\sum_{h\in v\Lambda^{\mathbf{e}_j}}} s(h)(\N+\mathbf{e}_j)=v(\N).\] On the other hand, for $u\neq v^1,v^2$, $\phi_I(\psi(u(\N)))=\phi_I(u(\N))=u(\N)$ and for $t=1,2$,
{\allowdisplaybreaks
\begin{align*}
\phi_I(\psi(v^t(\N)))&=\phi_I\left(\displaystyle{\sum_{f\in \mathcal{E}_t\cap \Lambda^{\mathbf{e}_j}}} s(f)(\N+\mathbf{e}_j)\right)\\
&= \displaystyle{\sum_{\substack{f\in \mathcal{E}_t\cap \Lambda^{\mathbf{e}_j}\\s(f)\neq v}}} s(f)(\N+\mathbf{e}_j)+\displaystyle{\sum_{f\in \mathcal{E}_t\cap \Lambda^{\mathbf{e}_j}v}} (v^1(\N+\mathbf{e}_j)+v^2(\N+\mathbf{e}_j))\\
&= \displaystyle{\sum_{\substack{\alpha\in v^t\Lambda_I^{\mathbf{e}_j}\\s_I(\alpha)\neq v^1,v^2}}} s_I(\alpha)(\N+\mathbf{e}_j)+\displaystyle{\sum_{\alpha\in v^t\Lambda_I^{\mathbf{e}_j}v^1}} s_I(\alpha)(\N+\mathbf{e}_j)+\displaystyle{\sum_{\alpha\in v^t\Lambda_I^{\mathbf{e}_j}v^2}} s_I(\alpha)(\N+\mathbf{e}_j)\\
&= \displaystyle{\sum_{\alpha\in v^t\Lambda_I^{\mathbf{e}_j}}} s_I(\alpha)(\N+\mathbf{e}_j)\\
&= v^t(\N)~~(\text{using relations of}~T_{\Lambda_I}).\qedhere
\end{align*}
}
\end{proof}

\subsection{Sink deletion}\label{ssec sink deletion}
If a directed graph $E$ has a source $v$ which is also a regular vertex, then deleting $v$ and all the edges emanating from $v$, one obtains the \emph{source elimination graph} $E_{\setminus v}$. Similar to in-splitting and out-splitting, the move of source elimination also does not change the graded Morita equivalent class of the corresponding Leavitt path algebra (see \cite[Proposition 13]{Hazrat-main}). The directions of paths are reversed in the transition from directed graphs to path categories and so at the level of $k$-graphs, the  analogue of the source elimination process is \emph{sink deletion}. We briefly review the process as described in \cite[Section 5]{EFGGGP}. 

A vertex $v\in \Lambda^0$ is called an \emph{$\mathbf{e}_i$-sink} (where $1\le i\le k$) if $\Lambda^{\mathbf{e}_i}v=\emptyset$. Notice that an $\mathbf{e}_i$-sink $v$ may emit edges of other degrees, and if that happens, then $v$ is not the only $\mathbf{e}_i$-sink present in the $k$-graph. In this case, it is not hard to observe that any $w\in \Lambda^0$ with $w\Lambda v\neq \emptyset$ is also an $\mathbf{e}_i$-sink (see \cite[Lemma 5.3]{EFGGGP}). Keeping this in mind, the move of sink deletion at $v$ not only removes $v$ but also removes all its out-neighbours. More precisely, the \emph{sink deletion $k$-graph} of $\Lambda$ is defined as \[\Lambda_S:=\{\lambda\in \Lambda~|~r(\lambda)\Lambda v=\emptyset\}.\] Clearly, $\Lambda_S^0=\{w\in \Lambda^0~|~w\Lambda v=\emptyset\}$. The source, range and degree maps are denoted by $s_S$, $r_S$ and $d_S$ respectively, and they are the restrictions of $s$, $r$ and $d$ to $\Lambda_S$. By \cite[Theorem 5.4]{EFGGGP}, $(\Lambda_S,d_S)$ is also a row-finite $k$-graph having no sources. 

The following proposition shows that sink deletion leaves invariant the graded Morita equivalence class of Kumjian--Pask algebras. For the $C^*$-algebra analogue, see \cite[Theorem 5.5]{EFGGGP}.
\begin{prop}\label{pro sink deletion giving graded Morita equivalent KP-algberas}
Let $\Lambda$ be a row-finite $k$-graph without sources, $v$ an $\mathbf{e}_i$-sink in $\Lambda$ and $\Lambda_S$ the sink deletion $k$-graph obtained by deleting $v$. Suppose $|\Lambda^0|< \infty$. Then, for any field $\mathsf{F}$, the Kumjian--Pask algebras $\KP_\mathsf{F}(\Lambda)$ and $\KP_\mathsf{F}(\Lambda_S)$ are graded Morita equivalent.     
\end{prop}
\begin{proof}
Suppose $(p,s)$ and $(q,t)$ are the respective Kumjian--Pask families in $\KP_\mathsf{F}(\Lambda)$ and $\KP_\mathsf{F}(\Lambda_S)$. Let $H:=\Lambda_S^0$. Let $\lambda\in \Lambda$ be such that $r(\lambda)\in H$. So $r(\Lambda)\Lambda v=\emptyset$. This implies $s(\lambda)\Lambda v=\emptyset$, i.e., $s(\lambda)\in H$. Therefore, $H$ is a hereditary subset of $\Lambda^0$. Notice that $\Lambda_S=\Lambda_H$, where $\Lambda_H$ is the $k$-graph defined as: \[\Lambda_H:=\{\lambda\in \Lambda~|~r(\lambda)\in H\}.\] Then the assignment \[q_v\longmapsto p_v,~~~t_\lambda\longmapsto s_\lambda,~~~t_{\lambda^*}\longmapsto s_{\lambda^*},\] extends to a $\mathbb{Z}^k$-graded injective homomorphism $\varphi$ from $\KP_\mathsf{F}(\Lambda_S)$ to $\KP_\mathsf{F}(\Lambda)$ and $\image(\varphi)=p\KP_\mathsf{F}(\Lambda)p$, where $p=\displaystyle{\sum_{w\in \Lambda_S^0}} p_w$. We are not going to prove it here, as it will be exactly the same as the Leavitt path algebra case given in \cite[Proposition 2.2.22 $(i)$ and $(ii)$]{Abrams-Monograph}. Hence, $\KP_\mathsf{F}(\Lambda_S)\cong p\KP_\mathsf{F}(\Lambda)p$ as $\mathbb{Z}^k$-graded rings. Clearly, $p$ is a homogeneous idempotent of $\KP_\mathsf{F}(\Lambda)$. We claim that $p$ is a full idempotent. First, observe that $\KP_\mathsf{F}(\Lambda)p\KP_\mathsf{F}(\Lambda)=I_{H}$, the ideal of $\KP_\mathsf{F}(\Lambda)$ generated by $\{p_u~|~u\in H\}$. This is because, for every $u\in H$, $p_u=p_u p p_u$ and so $I_H\subseteq \KP_\mathsf{F}(\Lambda)p\KP_\mathsf{F}(\Lambda)$; the other inclusion is trivial as $p\in I_H$. Again, notice that the saturated closure of $H$ is $\Lambda^0$. To confirm this, choose any $u\in \Lambda^0\setminus \Lambda_S^0$ and $\alpha\in u\Lambda^{\mathbf{e}_i}$. Then, $s(\alpha)\Lambda v=\emptyset$; otherwise, $s(\alpha)$ turns out to be an $\mathbf{e}_i$-sink, which is a contradiction. Thus, $s(u\Lambda^{\mathbf{e}_i})\subseteq H$ and consequently, $u\in \overline{H}$. Putting these together, we have \[\KP_\mathsf{F}(\Lambda)p\KP_\mathsf{F}(\Lambda)=I_H=I_{\Lambda^0}=\KP_\mathsf{F}(\Lambda).\] So our claim is proved. Therefore, the corner $p\KP_\mathsf{F}(\Lambda)p$ and $\KP_\mathsf{F}(\Lambda)$ are graded Morita equivalent and the result follows by applying \cite[Theorem 3]{Hazrat-main}. 
\end{proof}
We now directly show without referring to Kumjian--Pask algebras that the talented monoid of $\Lambda$ remains invariant under the move of sink deletion. The following theorem should be viewed as the $k$-graph analogue of \cite[Proposition 4.2]{Cordeiro}. 
\begin{thm}\label{th sink deletion preserves talented monoid}
Let $\Lambda$ be a row-finite $k$-graph without sources, $v$ an $\mathbf{e}_i$-sink and $\Lambda_S$ the sink deletion $k$-graph. Then, there is a $\mathbb{Z}^k$-monoid isomorphism $\phi_S:T_{\Lambda_S}\longrightarrow T_\Lambda$ such that $\phi_S(w(\N))=w(\N)$ for all $w\in \Lambda_S^0$ and $\N\in \mathbb{Z}^k$. 
\end{thm}
\begin{proof}
The assignment $w(\N)\longmapsto w(\N)$ when extended linearly to all of $T_{\Lambda_S}$, gives a well-defined $\mathbb{Z}^k$-monoid homomorphism $\phi_S:T_{\Lambda_S}\longrightarrow T_\Lambda$. This is easy to prove since for any $w\in \Lambda_S^0$, $\N\in \mathbb{Z}^k$ and $\M\in \mathbb{N}^k$, we have \[\phi_S(w(\N))=w(\N)=\displaystyle{\sum_{\lambda\in w\Lambda^\M}} s(\lambda)(\N+\M)=\displaystyle{\sum_{\lambda\in w\Lambda_S^\M}} s_S(\lambda)(\N+\M)=\phi_S\left(\displaystyle{\sum_{\lambda\in w\Lambda_S^\M}} s_S(\lambda)(\N+\M)\right).\] Suppose $x,y\in T_{\Lambda_S}$ are such that $\phi_S(x)=\phi_S(y)$. Let $x=\displaystyle{\sum_{i=1}^{\ell}} u_i(\mathbf{p}_i)$ and $y=\displaystyle{\sum_{j=1}^{t}} w_j(\mathbf{q}_j)$. 
Then in $T_\Lambda$ we have \[\displaystyle{\sum_{i=1}^{\ell}} u_i(\mathbf{p}_i)=\displaystyle{\sum_{j=1}^{t}} w_j(\mathbf{q}_j).\]
Passing this equality to $M_{\overline{\Lambda}}$ (where $\overline{\Lambda}$ is the skew product $k$-graph) and using the confluence lemma (see \cite[Lemma 3.5]{HMPS}),
we have $\gamma\in \mathbb{F}_{\overline{\Lambda}}\setminus \{0\}$ such that \[\displaystyle{\sum_{i=1}^{\ell}} (u_i,\mathbf{p}_i)\longrightarrow \gamma~\text{and}~\displaystyle{\sum_{j=1}^{t}} (w_j,\mathbf{q}_j)\longrightarrow \gamma.\] Since $u_i,w_j\in \Lambda_S^0$ for all $i=1,2,\ldots,\ell$, $j=1,2,\ldots,t$ and $\Lambda_S^0$ is hereditary, so every vertex in the support of $\gamma$ is in fact a vertex of $\overline{\Lambda_S}$. Therefore, by applying the confluence lemma again, we have \[\displaystyle{\sum_{i=1}^{\ell}} (u_i,\mathbf{p}_i)=\displaystyle{\sum_{j=1}^{t}} (w_j,\mathbf{q}_j)\] in $M_{\overline{\Lambda_S}}$ and so $x=y$ in $T_{\Lambda_S}$. This proves that $\phi_S$ is injective. For surjectivity, it suffices to show that $u(\N)\in \image(\phi_S)$ for all $u\in \Lambda^0\setminus \Lambda_S^0$ and $n\in \mathbb{Z}^k$. For each such $u$ and $\N$, we can write \[u(\N)=\displaystyle{\sum_{\alpha\in u\Lambda^{\mathbf{e}_i}}}s(\alpha)(\N+\mathbf{e}_i)\] in $T_\Lambda$. The sum is nonempty since $\Lambda$ has no sources. If $\alpha\in u\Lambda^{\mathbf{e}_i}$, then $s(\alpha)\Lambda v=\emptyset$; otherwise, $s(\alpha)$ becomes an $\mathbf{e}_i$-sink by \cite[Lemma 5.3]{EFGGGP}, which is obviously not the case. This implies $s(\alpha)\in \Lambda_S^0$ for all $\alpha\in u\Lambda^{\mathbf{e}_i}$ and so \[u(\N)=\displaystyle{\sum_{\alpha\in u\Lambda^{\mathbf{e}_i}}}s(\alpha)(\N+\mathbf{e}_i)=\phi_S\left(\displaystyle{\sum_{\alpha\in u\Lambda^{\mathbf{e}_i}}} s(\alpha)(\N+\mathbf{e}_i)\right)\in \image(\phi_S).\qedhere\]
\end{proof}
We conclude this section by giving an example which illustrates our results on the invariance of talented monoid under the two $k$-graph moves we have discussed. 
\begin{example}\label{ex on the invariance of talented monoid under moves}
Suppose the following is the $1$-skeleton of a $2$-graph $\Lambda$:

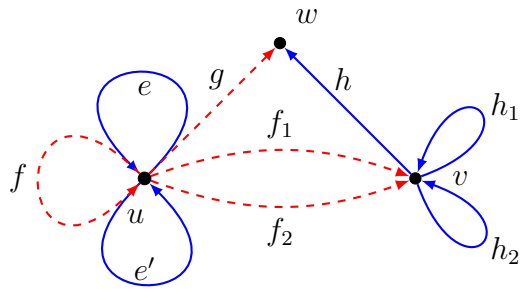
\[
\begin{tikzpicture}[scale=1.2]

\node[circle,draw,fill=black,inner sep=0.5pt] (p11) at (0, 0) {$.$} 
edge[-latex, blue,thick,loop, out=45, in=135, min distance=60] (p11)
edge[-latex, blue,thick, loop, out=225, in=-45, min distance=60] (p11)

edge[-latex, red,thick, loop, dashed, out=135, in=225, min distance=60] (p11);
                                                               
\node at (-0.1, -0.42) {$u$};

\node[inner sep=1.5pt, circle,draw,fill=black] (A) at (3,0) {}
edge[-latex, blue,thick, loop, out=20, in=70, min distance=40] (A)
edge[-latex, blue,thick, loop, out=-70, in=-20, min distance=40] (A);

\node[inner sep=1.5pt, circle,draw,fill=black] (B) at (1.5,1.5) {};

\path[->, red, dashed, >=latex,thick] (p11) edge [bend left=20] node[above=0.05cm]{} (A);
\path[->, red, dashed, >=latex,thick] (p11) edge [bend right=20] node[above=0.05cm]{} (A);
\path[->, red, dashed, >=latex,thick] (p11) edge [] node[above=0.05cm]{} (B);
\path[->, blue, >=latex,thick] (A) edge [] node[above=0.05cm]{} (B);

\node at (3.5,0) {$v$};
\node at (1.8,1.8) {$w$};

\node at (-1.4,0) {$f$};
\node at (1.5,0.6) {$f_1$};
\node at (1.5,-0.6) {$f_2$};
\node at (0.8,1.1) {$g$};
\node at (0,1) {$e$};
\node at (0,-1) {$e'$};
\node at (4,0.8) {$h_1$};
\node at (4,-0.8) {$h_2$};
\node at (2.2,1.1) {$h$};

\end{tikzpicture}
\]
with the factorization rules given as follows:

$(i)$ $u$-$u$ bi-colored paths: $ef=fe$, $e'f=fe'$;

$(ii)$ $u$-$w$ bi-colored paths: $hf_1=ge$, $hf_2=ge'$; 

$(iii)$ $u$-$v$ bi-colored paths: $h_1f_1=f_2 e$, $h_1f_2=f_2 e'$, $h_2 f_1=f_1 e$, $h_2 f_2=f_1 e'$.

We now in-split $\Lambda$ at the vertex $v$. Notice that there is only one possible partition (up to indexing) of $v\Lambda^\mathbf{1}$ which satisfies the pairing condition, namely $\mathcal{E}_1=\{h_1,f_2\}$ and $\mathcal{E}_2=\{h_2,f_1\}$. The in-split $2$-graph $\Lambda_I$ has the following $1$-skeleton:
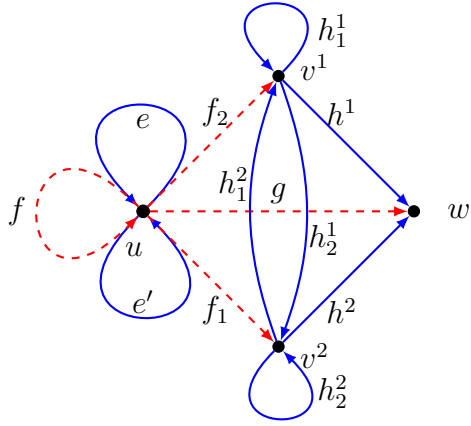
\[
\begin{tikzpicture}[scale=1.2]

\node[circle,draw,fill=black,inner sep=0.5pt] (p11) at (0, 0) {$.$} 
edge[-latex, blue,thick,loop, out=45, in=135, min distance=60] (p11)
edge[-latex, blue,thick, loop, out=225, in=-45, min distance=60] (p11)

edge[-latex, red,thick, loop, dashed, out=135, in=225, min distance=60] (p11);
                                          
\node at (-0.1, -0.42) {$u$};

\node[inner sep=1.5pt, circle,draw,fill=black] (A1) at (1.5,1.5) {}
edge[-latex, blue,thick, loop, out=45, in=135, min distance=40] (A1);

\node[inner sep=1.5pt, circle,draw,fill=black] (A2) at (1.5,-1.5) {}
edge[-latex, blue,thick, loop, out=225, in=-45, min distance=40] (A2);

\node[inner sep=1.5pt, circle,draw,fill=black] (B) at (3,0) {};

\path[->, red, dashed, >=latex,thick] (p11) edge [] node[above=0.05cm]{} (A1);
\path[->, red, dashed, >=latex,thick] (p11) edge [] node[above=0.05cm]{} (A2);
\path[->, red, dashed, >=latex,thick] (p11) edge [] node[above=0.05cm]{} (B);
\path[->, blue, >=latex,thick] (A1) edge [] node[above=0.05cm]{} (B);
\path[->, blue, >=latex,thick] (A2) edge [] node[above=0.05cm]{} (B);
\path[->,blue, >=latex,thick] (A1) edge [bend left=20] node[below=0.05cm]{} (A2);
\path[->,blue, >=latex,thick] (A2) edge [bend left=20] node[below=0.05cm]{} (A1);

\node at (3.5,0) {$w$};
\node at (1.9,1.6) {$v^1$};
\node at (1.9,-1.6) {$v^2$};

\node at (-1.4,0) {$f$};
\node at (1.5,0.2) {$g$};
\node at (1,0.3) {$h_1^2$};
\node at (2,-0.3) {$h_2^1$};
\node at (0.8,1.1) {$f_2$};
\node at (0.8,-1.1) {$f_1$};
\node at (0,1) {$e$};
\node at (0,-1) {$e'$};
\node at (2.1,2) {$h_1^1$};
\node at (2.1,-2) {$h_2^2$};
\node at (2.2,1.1) {$h^1$};
\node at (2.2,-1.1) {$h^2$};

\end{tikzpicture}
\]
Again, $v$ is an $(0,1)$-sink (red edges are of degree $(0,1)$) in $\Lambda$. Carrying out the sink deletion process, we obtain the $2$-graph $\Lambda_S$ with the following $1$-skeleton:
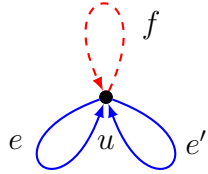
\[
\begin{tikzpicture}[scale=1.2]

\node[circle,draw,fill=black,inner sep=0.5pt] (p11) at (0, 0) {$.$} 
edge[-latex, blue,thick,loop, out=200, in=250, min distance=40] (p11)
edge[-latex, blue,thick, loop, out=-20, in=290, min distance=40] (p11)

edge[-latex, red,thick, loop, dashed, out=65, in=115, min distance=40] (p11);
                                          
\node at (0, -0.5) {$u$};
\node at (0.5, 0.8) {$f$};
\node at (-1, -0.5) {$e$};
\node at (1, -0.5) {$e'$};

\end{tikzpicture}
\]
By Theorem \ref{th in-splitting preserves talented monoid} and Theorem \ref{th sink deletion preserves talented monoid}, $T_\Lambda\cong T_{\Lambda_I}\cong T_{\Lambda_S}$. It is rather easy to compute $T_{\Lambda_S}$. In $T_{\Lambda_S}$, we have $u((i,j))=u((i,0))$ and $u((i,j))=2u((i+1,j))$ for each $i,j\in \mathbb{Z}$. One can easily deduce that $T_{\Lambda_S}$ is $\mathbb{Z}^2$-isomorphic to the commutative monoid 
\[
\mathbb{N}[1/2]:=\left\{\frac{m}{2^t}~:~m\in \mathbb{N}, t\in \mathbb{Z} \right\},
\] where the $\mathbb{Z}^2$-action is given as: $^{(i,j)}\frac{m}{2^t}:=\frac{m}{2^{t+i}}$ for all $(i,j)\in \mathbb{Z}^2$. 
\end{example}

\section{Relationship between graded $K$-theory and graded homology}\label{sec homology}
In this section, we investigate the relationship between the graded $K$-theory of a Kumjian--Pask algebra and the graded homology theory of the infinite path groupoid associated with the underlying $k$-graph. Before going to the details, let us briefly recall homology of ample groupoids, first introduced by Crainic and Moerdijk \cite{Crainic} in the setting of \'{e}tale groupoids. The readers may go through \cite{Matui, MatAdv} for more insights on the homology theory of groupoids. 

\subsection{Graded homology of a graded ample groupoid}\label{ssec groupoid homology}
Let $\mathcal{G}$ be an ample groupoid. For each positive integer $n$, write \[\mathcal{G}^{(n)}:=\{(g_1,g_2,\ldots,g_n)\in \mathcal{G}^n~|~s(g_i)=r(g_{i+1})~\text{for all}~i=1,2,\ldots,n-1\},\] and define $n+1$ maps $d_i^{(n)}:\mathcal{G}^{(n)}\longrightarrow \mathcal{G}^{(n-1)}$, $i=0,1,2,\ldots,n$ by 
$$d_i^{(n)}((g_1,g_2,\ldots,g_n)):=
	\left\{
	\begin{array}{lll}
		(g_2,g_3,\ldots,g_n)  & \mbox{if } i=0, \\
		(g_1,g_2,\ldots,g_ig_{i+1},\ldots,g_n)  & \mbox{if } 1\le i\le n-1,\\
            (g_1,g_2,\ldots,g_{n-1})  & \mbox{if } i=n.
	\end{array}
	\right.$$ 
where $d_0^{(1)},d_1^{(1)}:\mathcal{G}^{(1)}\longrightarrow \mathcal{G}^{(0)}$ are defined to be the \emph{source} and \emph{range} maps, respectively. For any topological abelian group $A$, suppose $C_c(\mathcal{G}^{(n)},A)$ is the set of all continuous functions $f:\mathcal{G}^{(n)}\longrightarrow A$ with compact support. Clearly, each $C_c(\mathcal{G}^{(n)},A)$ is an abelian group. Define the $n^{\text{th}}$ \emph{boundary map} to be the group homomorphism $\partial_n:C_c(\mathcal{G}^{(n)},A)\longrightarrow C_c(\mathcal{G}^{(n-1)},A)$, \[\partial_n:=\displaystyle{\sum_{i=1}^{n}} (-1)^{i}d_{i*}^{(n)},\] where $d_{i*}^{(n)}:C_c(\mathcal{G}^{(n)},A)\longrightarrow C_c(\mathcal{G}^{(n-1)},A)$ is the map given by $f\longmapsto d_{i*}^{(n)}(f)$ with \[(d_{i*}^{(n)}(f))(y):=\displaystyle{\sum_{d_i^{(n)}(x)=y}} f(x),\] for all $y\in \mathcal{G}^{(n-1)}$. Then, \[0\overset{\partial_0}{\longleftarrow} C_c(\mathcal{G}^{(0)},A)\overset{\partial_1}{\longleftarrow} C_c(\mathcal{G}^{(1)},A)\overset{\partial_2}{\longleftarrow} C_c(\mathcal{G}^{(2)},A)\overset{\partial_3}{\longleftarrow}\cdots\] is a chain complex, called the \emph{Moore complex} of the simplicial abelian group $(C_c(\mathcal{G}^{(n)},A))_n$. The \emph{$n^{\text{th}}$ homology group} of $\mathcal{G}$ with constant coefficients $A$ is defined as: \[H_n(\mathcal{G},A):=\Ker \partial_n /\image \partial_{n+1}.\] If we take $A=\mathbb{Z}$, the group of integers with discrete topology, then we simply write the $n^{\text{th}}$ homology $H_n(\mathcal{G},\mathbb{Z})$ as $H_n(\mathcal{G})$ and call it the $n^{\text{th}}$ \emph{integral homology} of $\mathcal{G}$. In particular, $H_0(\mathcal{G})=C_c(\mathcal{G}^{(0)},\mathbb{Z})/\image \partial_1$. Also, define \[H_0(\mathcal{G})^+:=\{[f]\in H_0(\mathcal{G})~|~f(u)\ge 0~\text{for all}~u\in \mathcal{G}^{(0)}\},\] where $[f]=f+\image \partial_1$. Then, $H_0(\mathcal{G})^+$ is a commutative cancellative monoid and $(H_0(\mathcal{G}),H_0(\mathcal{G})^+)$ is a pre-ordered abelian group. 

If $\Gamma$ is a group with discrete topology acting continuously from the left on the groupoid $\mathcal{G}$, i.e., there is a group homomorphism $\phi:\Gamma\longrightarrow \Aut(\mathcal{G})$, then each $C_c(\mathcal{G}^{(n)},A)$ is a $\Gamma$-module with respect to the following $\Gamma$-action: \[\gamma\cdot f:=f\circ \phi_{\gamma^{-1}}.\] Also, $\partial_n$'s are $\Gamma$-module homomorphisms and consequently, $H_n(\mathcal{G},A)$ is a $\Gamma$-module. 

We now present the definition of the graded homology of an ample groupoid.
\begin{dfn}\label{def graded homology}
(\cite[Definition 5.3]{H-Li}) Let $\mathcal{G}$ be a $\Gamma$-graded ample groupoid graded by the continuous cocycle $c:\mathcal{G}\longrightarrow \Gamma$. For each $n\ge 0$, the \emph{$n^{\text{th}}$ graded (integral) homology} of $\mathcal{G}$ is defined as: \[H_n^{\gr}(\mathcal{G}):=H_n(\mathcal{G}\times_c \Gamma).\] Also, define \[H_0^{\gr}(\mathcal{G})^+:=H_0(\mathcal{G}\times_c \Gamma)^+.\] 
\end{dfn}
Here $\mathcal{G}\times_c \Gamma$ is the skew-product of $\mathcal{G}$ with $\Gamma$ (see \cite[Definition 1.6]{Renault}). Since the skew-product groupoid $\mathcal{G}\times_c \Gamma$ is equipped with a natural $\Gamma$-action, namely $~^\gamma (g,\alpha):=(g,\gamma\alpha)$, by the discussion above, $H_n^{\gr}(\mathcal{G})$ is a left $\Gamma$-module. 

\subsection{Graded zeroth homology of the infinite path groupoid}\label{ssec homology of path groupoid}
Given a $k$-graph $\Lambda$ with degree functor $d$, the associated infinite path groupoid $\mathcal{G}_\Lambda$ is canonically $\mathbb{Z}^k$-graded via the continuous cocycle 
\begin{align*}
\Tilde{d}:\mathcal{G}_\Lambda &\longrightarrow \mathbb{Z}^k\\
(x,\N,y) &\longmapsto \N
\end{align*}
Therefore, $H_n^{\gr}(\mathcal{G}_\Lambda)=H_n(\mathcal{G}_\Lambda\times_{\Tilde{d}} \mathbb{Z}^k)$. 

When $\mathcal{G}$ is a strongly $\Gamma$-graded groupoid (see \cite{H-Li}), one can relate the graded homology of $\mathcal{G}$ to the usual (non-graded) homology of $\mathcal{G}_\epsilon$ by \cite[Theorem 5.4]{H-Li}, where $\epsilon$ is the identity of $\Gamma$ and $\mathcal{G}_\epsilon$ is the $\epsilon^{\text{th}}$ homogeneous component of $\mathcal{G}$. In the following proposition, we specialise this to the case of the infinite path groupoid of a $k$-graph $\Lambda$. 

\begin{prop}\label{pro connecting graded homology to non-graded one}
For any row-finite $k$-graph $\Lambda$ without sources, the following are equivalent.

$(i)$ $\Lambda$ satisfies Condition $(Y)$ of \cite[Definition 4.10]{CHR}.

$(ii)$ $\mathcal{G}_\Lambda$ is a strongly $\mathbb{Z}^k$-graded groupoid.

$(iii)$ $Y:=\mathcal{G}_\Lambda^{(0)}\times \{0\}$ is a $(\mathcal{G}_\Lambda\times_{\Tilde{d}} \mathbb{Z}^k)$-full clopen subset of $(\mathcal{G}_\Lambda\times_{\Tilde{d}} \mathbb{Z}^k)^{(0)}$.

If any one of the above holds, then $H_n^{\gr}(\mathcal{G}_\Lambda)\cong H_n((\mathcal{G}_\Lambda)_0)$ for any $n\ge 0$ and $H_0^{\gr}(\mathcal{G}_\Lambda)^+\cong H_0((\mathcal{G}_\Lambda)_0)^+$.  
\end{prop}
\begin{proof}
$(i)\Longrightarrow (ii)$ by \cite[Theorem 4.12]{CHR}. Suppose $(ii)$ holds. $Y$ is clearly a clopen subset of $(\mathcal{G}_\Lambda\times_{\Tilde{d}} \mathbb{Z}^k)^{(0)}$. To show the fullness of $Y$, take any $((x,0,x),\M)\in (\mathcal{G}_\Lambda\times_{\Tilde{d}} \mathbb{Z}^k)^{(0)}$. We can write $(x,0,x)=(x,-\M,y)(y,\M,x)$ for some $y\in \Lambda^\infty$ since $\mathcal{G}_\Lambda$ is strongly $\mathbb{Z}^k$-graded. Then, $((x,-\M,y),\M)\in \mathcal{G}_\Lambda\times_{\Tilde{d}}\mathbb{Z}^k$, $r(((x,-\M,y),\M))=((x,0,x),\M)$ and $s(((x,-\M,y),\M))=((y,0,y),0)\in Y$. Hence, $Y$ is $(\mathcal{G}_\Lambda\times_{\Tilde{d}} \mathbb{Z}^k)$-full and $(iii)$ follows.

Suppose $(iii)$ holds. Take any $\M\in \mathbb{N}^k$ and $x\in \Lambda^\infty$. Then, we have $((x,0,x),\M)\in (\mathcal{G}_\Lambda\times_{\Tilde{d}} \mathbb{Z}^k)^{(0)}$. Since $Y$ is full, there exists $((x,\N-\mathbf{q},y),\M)\in \mathcal{G}_\Lambda\times_{\Tilde{d}}\mathbb{Z}^k$ such that $s(((x,\N-\mathbf{q},y),\M))\in Y$. It follows that $\sigma^\N(x)=\sigma^\mathbf{q}(y)$ and $\N-\mathbf{q}+\M=0$. Let $\beta=y(0,\mathbf{q})$. Then, $s(\beta)=y(\mathbf{q})=\sigma^\mathbf{q}(y)(0)=\sigma^\N(x)(0)=x(\N)$ and $d(\beta)-\N=\mathbf{q}-\N=\M$. So $\Lambda$ satisfies Condition $(Y)$. 

For the remaining part, note that the unit space $(\mathcal{G}_\Lambda\times_{\Tilde{d}} \mathbb{Z}^k)^{(0)}=\mathcal{G}_\Lambda^{(0)}\times \mathbb{Z}^k$ is $\sigma$-compact since $\Lambda^0$ and $\mathbb{Z}^k$ are countable, $Z(v)$'s are compact and \[\mathcal{G}_\Lambda^{(0)}\times \mathbb{Z}^k=\displaystyle{\bigcup_{v\in \Lambda^0,\N\in \mathbb{Z}^k}} Z(v)\times \{\N\}.\] Also, since $\mathcal{G}_\Lambda\times_{\Tilde{d}} \mathbb{Z}^k$ is ample, the unit space is totally disconnected. Thus, by \cite[Theorem 3.6 (2)]{Matui}, $\mathcal{G}_\Lambda\times_{\Tilde{d}} \mathbb{Z}^k$ is homologically similar to $(\mathcal{G}_\Lambda\times_{\Tilde{d}} \mathbb{Z}^k)|_{Y}$. Again, $(\mathcal{G}_\Lambda\times_{\Tilde{d}} \mathbb{Z}^k)|_{Y}\cong (\mathcal{G}_\Lambda)_0$ via the canonical isomorphism which maps $((x,0,y),0)$ to $(x,0,y)$. Therefore, $\mathcal{G}_\Lambda\times_{\Tilde{d}} \mathbb{Z}^k$ is homologically similar to $(\mathcal{G}_\Lambda)_0$ and we are done by \cite[Theorem 3.5 (2)]{Matui}. 
\end{proof}

To relate the graded homology of the graph groupoid $\mathcal{G}_E$ of a directed graph $E$ with the graded $K$-theory of the corresponding Leavitt path algebra $L(E)$, the first named author and Li used $\mathcal{D}_\mathbb{Z}(E)$, the \emph{diagonal subalgebra} of $L_\mathbb{Z}(E)$ (the Leavitt path algebra of $E$ with integral coefficients). They realised $\mathcal{D}_\mathbb{Z}(E)$ as a certain quotient of the free abelian group generated by the elements of the form $\alpha\alpha^*$, $\alpha\in E^*$ (see \cite[Lemma 6.4]{H-Li}). The proof uses the fact that the Leavitt path algebra of any directed graph has a specific basis \cite[Lemma 6.2]{H-Li}. Unfortunately, we do not have a prescribed basis for the Kumjian--Pask algebra of a row-finite $k$-graph. We instead follow a different approach. Below, we define the \emph{path group} associated with a $k$-graph $\Lambda$, which will be used in the sequel as a bridge to set up our desired relationship between the graded homology of $\mathcal{G}_\Lambda$ and $K_0^{\gr}(\KP_\mathsf{F}(\Lambda))$.   
\begin{dfn}\label{def path group}
Let $\Lambda$ be a row-finite $k$-graph without sources. The \emph{path group} of $\Lambda$ is denoted by $\mathcal{P}(\Lambda)$ and is defined as the free abelian group generated by $\Lambda$ subject to the relations
\begin{equation}\label{pathrel}
\lambda=\displaystyle{\sum_{\alpha\in s(\lambda)\Lambda^\N}} \lambda\alpha    
\end{equation}
for all $\lambda\in \Lambda$ and $\N\in \mathbb{N}^k$.
\end{dfn}
We will show that as abelian groups $\mathcal{P}(\Lambda)\cong C_c(\mathcal{G}_\Lambda^{(0)},\mathbb{Z})$, where $\mathcal{G}_\Lambda$ is the path groupoid associated to $\Lambda$. For this we need the following lemma.
\begin{lem}\label{lem key for showing well-definedness}
Let $\lambda\in \Lambda$ be such that $Z(\lambda)=\displaystyle{\bigsqcup_{i=1}^{t}}~Z(\mu_i)$ for some $t\ge 1$ and $\mu_i\in \Lambda$ for each $i=1,2,\ldots,t$. Then, $\lambda=\displaystyle{\sum_{i=1}^{t}} ~\mu_i$ in $\mathcal{P}(\Lambda)$.     
\end{lem}
\begin{proof}
For each $i=1,2,\ldots,t$ we have
\[Z(\mu_i)=Z(\mu_i)\cap Z(\lambda)=\displaystyle{\bigsqcup_{\beta\in \Ext(\mu_i,\{\lambda\})}} Z(\mu_i\beta)=\displaystyle{\bigsqcup_{\alpha\in \Ext(\lambda,\{\mu_i\})}}Z(\lambda\alpha).\] The second equality follows from \cite[Lemma 3.17 $(i)$]{HMPS} and the last equality follows since for each $\alpha\in \Ext(\lambda,\{\mu_i\})$, there is a unique $\beta\in \Ext(\mu_i,\{\lambda\})$ such that $\lambda\alpha=\mu_i\beta$ and conversely. Thus, \[\displaystyle{\bigsqcup_{\nu\in s(\mu_i)\Lambda^{(d(\mu_i)\vee d(\lambda))-d(\mu_i)}}} Z(\mu_i\nu)=\displaystyle{\bigsqcup_{\beta\in \Ext(\mu_i,\{\lambda\})}} Z(\mu_i\beta).\] Clearly, $\Ext(\mu_i,\{\lambda\})\subseteq s(\mu_i)\Lambda^{(d(\mu_i)\vee d(\lambda))-d(\mu_i)}$. From the above equality, it follows that the reverse inclusion also holds. Indeed, if we take $\nu\in s(\mu_i)\Lambda^{(d(\mu_i)\vee d(\lambda))-d(\mu_i)}$ and any $x\in s(\nu)\Lambda^\infty$, then there is a unique $\beta\in \Ext(\mu_i,\{\lambda\})$ such that $\mu_i\nu x\in Z(\mu_i\beta)$, which implies $\mu_i\nu x=\mu_i\beta y$ for some $y\in s(\beta)\Lambda^\infty$. But since $d(\nu)=d(\beta)$, it must happen that $\nu=\beta$. So $\Ext(\mu_i,\{\lambda\})=s(\mu_i)\Lambda^{(d(\mu_i)\vee d(\lambda))-d(\mu_i)}$. Let \[\M:=\displaystyle{\bigvee_{i=1}^{t}} \{(d(\lambda)\vee d(\mu_i))-d(\lambda)\}=\displaystyle{\bigvee_{i=1}^{t}}(d(\lambda)\vee d(\mu_i))-d(\lambda).\] Now in $\mathcal{P}(\Lambda)$, we have
\begin{align*}
\displaystyle{\sum_{i=1}^{t}}\mu_i&= \displaystyle{\sum_{i=1}^{t}}\left(\displaystyle{\sum_{\beta\in s(\mu_i)\Lambda^{(d(\mu_i)\vee d(\lambda))-d(\mu_i)}}}\mu_i\beta\right)~(\text{by defining relations (\ref{pathrel})})\\
&=\displaystyle{\sum_{i=1}^{t}}\left(\displaystyle{\sum_{\beta\in \Ext(\mu_i,\{\lambda\})}} \mu_i\beta\right)\\
&=\displaystyle{\sum_{i=1}^{t}}\left(\displaystyle{\sum_{\alpha\in \Ext(\lambda,\{\mu_i\})}} \lambda\alpha\right)\\
&=\displaystyle{\sum_{i=1}^{t}}\left(\displaystyle{\sum_{\alpha\in \Ext(\lambda,\{\mu_i\})}} \left(\displaystyle{\sum_{\gamma\in s(\alpha)\Lambda^{\M-d(\alpha)}}}\lambda\alpha\gamma\right)\right)~(\text{by defining relations (\ref{pathrel})})\\
&=\displaystyle{\sum_{\delta\in s(\lambda)\Lambda^\M}}\lambda\delta=\lambda~(\text{by defining relations (\ref{pathrel})}).
\end{align*}
It remains to explain how the penultimate equality holds. If $i\neq j$, then it is clear that $\Ext(\lambda,\{\mu_i\})\cap \Ext(\lambda,\{\mu_j\})=\emptyset$; otherwise, $Z(\mu_i)\cap Z(\mu_j)\neq \emptyset$. Again, for the same reason, it cannot happen that $\alpha\gamma=\alpha'\gamma'$ for $\alpha\in\Ext(\lambda,\{\mu_i\})$, $\alpha'\in \Ext(\lambda,\{\mu_j\})$, $\gamma\in s(\alpha)\Lambda^{\M-d(\alpha)}$ and $\gamma'\in s(\alpha')\Lambda^{\M-d(\alpha')}$. Thus, all the summands in line $4$ above are distinct. Also, for each $i=1,2,\ldots,t$, $\alpha\in \Ext(\lambda,\{\mu_i\})$ and $\gamma\in s(\alpha)\Lambda^{\M-d(\alpha)}$, $d(\alpha\gamma)=\M$, which implies $\alpha\gamma\in s(\lambda)\Lambda^\M$. On the other hand, if $\delta\in s(\lambda)\Lambda^\M$, then since $d(\mu_i)\le d(\lambda\delta)$ and $Z(\lambda\delta)\subseteq Z(\lambda)=\displaystyle{\bigsqcup_{i=1}^{t}}~Z(\mu_i)$, it follows that there is a unique $i$ such that $(\lambda\delta)(0,d(\mu_i))=\mu_i$. So $\lambda\delta=\mu_i\beta$ for some $\beta\in \Lambda$. Since $d(\lambda\delta)\ge d(\lambda)\vee d(\mu_i)$, there should exist $\alpha,\beta'\in \Lambda$ such that $(\lambda\delta)(0,d(\lambda)\vee d(\mu_i))=\lambda\alpha$ and $(\mu_i\beta)(0,d(\lambda)\vee d(\mu_i))=\mu_i\beta'$, whence $\lambda\alpha=\mu_i\beta'$ and $\alpha\in \Ext(\lambda,\{\mu_i\})$. Then, $\delta=\alpha\gamma$ for some $\gamma\in s(\alpha)\Lambda^{\M-d(\alpha)}$, which proves that each $\delta\in s(\lambda)\Lambda^\M$ is in the list of summand in line $4$. 
\end{proof}

\begin{prop}\label{pro connecting path group with SA of path groupoid}
Let $\Lambda$ be any row-finite $k$-graph without sources and $\mathbb{Z}$, the group of integers equipped with discrete topology. Then, the map $1_{Z(\lambda)} \mapsto \lambda$ extends to an isomorphism $\phi : C_c(\mathcal{G}_\Lambda^{(0)},\mathbb{Z})\to \mathcal{P}(\Lambda)$. 
\end{prop}
\begin{proof}
First, note that any $f\in C_c(\mathcal{G}_\Lambda^{(0)},\mathbb{Z})$ can be written as $f=\displaystyle{\sum_{U\in \mathcal{F}}} a_U 1_U$, where $\mathcal{F}$ is a finite collection (empty if $f=0$) of mutually disjoint compact open subsets of $\mathcal{G}_\Lambda^{(0)}$ and $a_U\in \mathbb{Z}\setminus \{0\}$ for all $U\in \mathcal{F}$ (precisely, $U=f^{-1}(a_U)$). The set $\mathcal{B}:=\{Z(\lambda):\lambda\in \Lambda\}$ forms a basis of compact sets for the topology of $\mathcal{G}_\Lambda^{(0)}$. Using this fact together with \cite[Lemma 3.17]{HMPS}, we can decompose each $U\in \mathcal{F}$ as a disjoint union $U=\displaystyle{\bigsqcup_{i=1}^{\ell_U}}~Z(\lambda_{\nu_i})$. So $f$ can eventually be expressed as \[f=\displaystyle{\sum_{U\in \mathcal{F}}}\left(\sum_{i=1}^{\ell_U} a_U 1_{Z(\lambda_{\nu_i})}\right).\] This shows that each $f\in C_c(\mathcal{G}_\Lambda^{(0)},\mathbb{Z})$ can be written as a $\mathbb{Z}$-linear combination of characteristic functions of mutually disjoint basic open subsets. Define a map $\phi:C_c(\mathcal{G}_\Lambda^{(0)},\mathbb{Z})\longrightarrow \mathcal{P}(\Lambda)$ by extending the rule \[1_{Z(\lambda)}\longmapsto \lambda,\] linearly to all $\mathbb{Z}$-linear combinations of characteristic functions of mutually disjoint basic open sets. We need to check that this gives a well-defined group homomorphism. We first observe that if $Z(\lambda)=\displaystyle{\bigsqcup_{i=1}^{m}}~U_i$, where $U_i$'s are mutually disjoint compact open subsets of $\mathcal{G}_\Lambda^{(0)}$, then $\phi(1_{Z(\lambda)})=\displaystyle{\sum_{i=1}^{m}}\phi(1_{U_i})$. Write $U_i=\displaystyle{\bigsqcup_{j=1}^{n_i}}~Z(\lambda_{i_j})$ for each $i=1,2,\ldots,m$. Then, by our definition, $\phi(1_{U_i})=\displaystyle{\sum_{j=1}^{n_i}}\phi(1_{Z(\lambda_{i_j})})=\displaystyle{\sum_{j=1}^{n_i}}\lambda_{i_j}$. Again, $Z(\lambda)=\displaystyle{\bigsqcup_{i=1}^{m}}\left(\displaystyle{\bigsqcup_{j=1}^{n_i}}~Z(\lambda_{i_j})\right)$ and so by applying Lemma \ref{lem key for showing well-definedness}, we have \[\phi(1_{Z(\lambda)})=\lambda=\displaystyle{\sum_{i=1}^{m}}\left(\displaystyle{\sum_{j=1}^{n_i}}\lambda_{i_j}\right)=\displaystyle{\sum_{i=1}^{m}}\phi(1_{U_i})\] in $\mathcal{P}(\Lambda)$. Now, assume $\displaystyle{\sum_{i=1}^{m}}a_i 1_{Z(\lambda_i)}=\displaystyle{\sum_{j=1}^{n}}b_j 1_{Z(\mu_j)}$, where $Z(\lambda_i)$'s are mutually disjoint, $Z(\mu_j)$'s are mutually disjoint and $a_i,b_j\in \mathbb{Z}\setminus \{0\}$ for all $i=1,2,\ldots,m$ and $j=1,2,\ldots,n$. For each $i,j$, set $W_{ij}:=Z(\lambda_i)\cap Z(\mu_j)$ and $\mathcal{K}:=\{W_{ij}~|~W_{ij}\neq \emptyset\}$. Note that \[Z(\lambda_i)=\displaystyle{\bigsqcup_{\substack{1\le j\le n\\W_{ij}\neq \emptyset}}}~ W_{ij}=\displaystyle{\bigsqcup_{W\in \mathcal{K},W\subseteq Z(\lambda_i)}}~W\] and similarly, \[Z(\mu_j)=\displaystyle{\bigsqcup_{W\in \mathcal{K},W\subseteq Z(\mu_j)}}~W.\] By our earlier observation, it follows that \[\phi\left(\displaystyle{\sum_{i=1}^{m}}a_i1_{Z(\lambda_i)}\right)=\displaystyle{\sum_{i=1}^{m}}a_i\left(\displaystyle{\sum_{W\in \mathcal{K},W\subseteq Z(\lambda_i)}}\phi(1_W)\right)=\displaystyle{\sum_{W\in \mathcal{K}}}\left(\displaystyle{\sum_{\substack{1\le i\le m\\W\subseteq Z(\lambda_i)}}}a_i\right)\phi(1_W),\] and similarly, \[\phi\left(\displaystyle{\sum_{j=1}^{n}}b_j 1_{Z(\mu_j)}\right)=\displaystyle{\sum_{W\in \mathcal{K}}}\left(\displaystyle{\sum_{\substack{1\le j\le n\\W\subseteq Z(\mu_j)}}}b_j\right)\phi(1_W).\] We remark that there is a single summand in each summation inside the parentheses on the right-hand side of both the equations above. Take any $W\in \mathcal{K}$ and $u\in W$. Clearly, $u\in Z(\lambda_i)$ if and only if $W\subseteq Z(\lambda_i)$, and $u\in Z(\mu_j)$ if and only if $W\subseteq Z(\mu_j)$. Therefore, \[\displaystyle{\sum_{\substack{1\le i\le m\\W\subseteq Z(\lambda_i)}}}a_i=\displaystyle{\sum_{i=1}^{m}}a_i 1_{Z(\lambda_i)}(u)=\displaystyle{\sum_{j=1}^{n}}b_j 1_{Z(\mu_j)}(u)=\displaystyle{\sum_{\substack{1\le j\le n\\W\subseteq Z(\mu_j)}}} b_j,\] and consequently, $\phi\left(\displaystyle{\sum_{i=1}^{m}}a_i 1_{Z(\lambda_i)}\right)=\phi\left(\displaystyle{\sum_{j=1}^{n}}b_j 1_{Z(\mu_j)}\right)$. This shows that $\phi$ is well-defined. It is indeed a group homomorphism. For this, we show that $\phi(a 1_{Z(\lambda)}+b 1_{Z(\mu)})=a\phi(1_{Z(\lambda)})+b\phi(1_{Z(\mu)})$ and then one can extend the idea to show that $\phi(f+g)=\phi(f)+\phi(g)$ for any two $f,g\in C_c(\mathcal{G}_\Lambda^{(0)},\mathbb{Z})$. Note that, $Z(\lambda)=(Z(\lambda)\setminus Z(\mu))\sqcup (Z(\lambda)\cap Z(\mu))$ and $Z(\mu)=(Z(\mu)\setminus Z(\lambda))\sqcup (Z(\mu)\cap Z(\lambda))$. Since $\mathcal{G}_\Lambda^{(0)}$ is Hausdorff, $Z(\lambda)\setminus Z(\mu)$, $Z(\lambda)\cap Z(\mu)$ and $Z(\mu)\setminus Z(\lambda)$ are mutually disjoint compact open subsets. Thus, 
\begin{align*}
&~ a\phi(1_{Z(\lambda)})+b\phi(1_{Z(\mu)})\\
= &~a\left(\phi(1_{Z(\lambda)\setminus Z(\mu)})+\phi(1_{Z(\lambda)\cap Z(\mu)})\right)+b\left(\phi(1_{Z(\mu)\setminus Z(\lambda)})+\phi(1_{Z(\mu)\cap Z(\lambda)})\right)\\
= &~ a\phi(1_{Z(\lambda)\setminus Z(\mu)})+(a+b)\phi(1_{Z(\lambda)\cap Z(\mu)})+b\phi(1_{Z(\mu)\setminus Z(\lambda)})\\
= &~ \phi\left(a 1_{Z(\lambda)\setminus Z(\mu)}+(a+b) 1_{Z(\lambda)\cap Z(\mu)}+b 1_{Z(\mu)\setminus Z(\lambda)}\right)\\
= &~ \phi(a 1_{Z(\lambda)}+b 1_{Z(\mu)}),   
\end{align*}
where the first equality follows by our earlier observation. Now, consider a map $\psi:\mathcal{P}(\Lambda)\longrightarrow C_c(\mathcal{G}_\Lambda^{(0)},\mathbb{Z})$ defined on the generators by \[\lambda\longmapsto 1_{Z(\lambda)},\] and then extended linearly. Note that for any $\N\in \mathbb{N}^k$, \[\psi\left(\displaystyle{\sum_{\alpha\in s(\lambda)\Lambda^\N}} \lambda\alpha\right)=\displaystyle{\sum_{\alpha\in s(\lambda)\Lambda^\N}}1_{Z(\lambda\alpha)}=1_{Z(\lambda)}\] since $Z(\lambda)=\displaystyle{\bigsqcup_{\alpha\in s(\lambda)\Lambda^\N}} Z(\lambda\alpha)$, and hence $\psi$ is compatible with the defining relations of $\mathcal{P}(\Lambda)$. Evidently, $\psi(\phi(1_{Z(\lambda)}))=\psi(\lambda)=1_{Z(\lambda)}$ and $\phi(\psi(\lambda))=\phi(1_{Z(\lambda)})=\lambda$ which assures that $\phi$, $\psi$ are mutually inverse group homomorphisms and that completes the proof.
\end{proof}
Consider the skew-product $k$-graph $\overline{\Lambda}=\Lambda\times_d \mathbb{Z}^k$ (see Example \ref{ex skew-product k-graphs}). Applying Proposition \ref{pro connecting path group with SA of path groupoid} to $\overline{\Lambda}$, we obtain $C_c(\mathcal{G}_{\overline{\Lambda}}^{(0)},\mathbb{Z})\cong \mathcal{P}(\overline{\Lambda})$. Next, we observe that $\mathcal{P}(\overline{\Lambda})$ is a $\mathbb{Z}^k$-module with respect to a certain action of $\mathbb{Z}^k$. 

Note that there is a groupoid isomorphism (see \cite[Theorem 5.2]{Kumjian--Pask})
\begin{align*}
\varphi: \mathcal{G}_\Lambda \times_{\Tilde{d}} \mathbb{Z}^k &\longrightarrow \mathcal{G}_{\overline{\Lambda}}\\
((x,\N,y),\M) &\longmapsto ((x,\M),\N,(y,\M+\N)).
\end{align*}
The readers should note that the set of infinite paths in $\overline{\Lambda}$, i.e., $\overline{\Lambda}^\infty$ is identified with $\Lambda^\infty \times \mathbb{Z}^k$ as follows: for any $\chi\in \overline{\Lambda}^\infty$, $\chi\longmapsto (\pi_1\circ \chi, \pi_2(\chi(0)))\in \Lambda^\infty \times \mathbb{Z}^k$, where $\pi_1:\Lambda\times \mathbb{Z}^k\longrightarrow \Lambda$ and $\pi_2:\Lambda\times \mathbb{Z}^k \longrightarrow \mathbb{Z}^k$ are the usual projections; on the other hand any $(x,\N)\in \Lambda^\infty \times \mathbb{Z}^k$ is realised as the infinite path $(x,\N):\Omega_k\longrightarrow \overline{\Lambda}$ defined by $(x,\N)(\mathbf{p},\mathbf{q}):=(x(\mathbf{p},\mathbf{q}),\N+\mathbf{p})$. 

There is a canonical left $\mathbb{Z}^k$-action on the skew-product groupoid $\mathcal{G}_\Lambda\times_{\tilde{d}} \mathbb{Z}^k$ defined as: \[\mathbf{q}\cdot ((x,\N,y),\M):=((x,\N,y),\mathbf{q}+\M),\] which, in view of the above isomorphism $\varphi$, induces a $\mathbb{Z}^k$-action on $\mathcal{G}_{\overline{\Lambda}}$ as follows: \[\mathbf{q}\cdot ((x,\M),\N,(y,\M+\N)):=((x,\mathbf{q}+\M),\N,(y,\mathbf{q}+\M+\N)).\] Viewing this action as a group homomorphism $\eta:\mathbb{Z}^k\longrightarrow \Aut(\mathcal{G}_{\overline{\Lambda}})$, we have a left $\mathbb{Z}^k$-action on the abelian group $C_c(\mathcal{G}_{\overline{\Lambda}}^{(0)},\mathbb{Z})$ defined by \[\mathbf{q}\cdot f:=f\circ \eta_{-\mathbf{q}}.\] Finally, define a left $\mathbb{Z}^k$-action on the generators of the path group $\mathcal{P}(\overline{\Lambda})$ by \[\mathbf{q}\cdot (\lambda,\N):=\phi(\mathbf{q}\cdot 1_{Z((\lambda,\N))}),\] where $\phi$ is the isomorphism between $C_c(\mathcal{G}_{\overline{\Lambda}}^{(0)},\mathbb{Z})$ and $\mathcal{P}(\overline{\Lambda})$ (see the proof of Proposition \ref{pro connecting path group with SA of path groupoid}). This clearly gives a well-defined action on $\mathcal{P}(\overline{\Lambda})$. For $(x,\M)\in \mathcal{G}_{\overline{\Lambda}}^{(0)}$, \[\left(\mathbf{q}\cdot 1_{Z((\lambda,\N))}\right)((x,\M))=\left(1_{Z((\lambda,\N))}\circ \eta_{-\mathbf{q}}\right)((x,\M))=1_{Z((\lambda,\N))}((x,\M-\mathbf{q}))=1_{Z((\lambda,\mathbf{q}+\N))}((x,\M)).\] The last equality follows since $(x,\M-\mathbf{q})\in Z((\lambda,\N))$ if and only if $(x,\M-\mathbf{q})=(\lambda,\N)(y,\mathbf{p})$ for some $(y,\mathbf{p})\in \overline{\Lambda}^\infty$, i.e, $(x,\M-\mathbf{q})=(\lambda y, \N)$ or, $x=\lambda y$ and $\M=\mathbf{q}+\N$ or, equivalently, $(x,\M)\in Z((\lambda,\mathbf{q}+\N))$. Therefore $\mathbf{q}\cdot (\lambda,\N)=(\lambda,\mathbf{q}+\N)$. 

Now, we state the main theorem of this section.
\begin{thm}\label{th relationship of graded K-theory and graded homology}
Let $\Lambda$ be any row-finite $k$-graph without sources and $\mathsf{F}$ any field. There exists a $\mathbb{Z}[\mathbb{Z}^k]$-module isomorphism \[\Phi:H_0^{\gr}(\mathcal{G}_\Lambda)\longrightarrow K_0^{\gr}(\KP_\mathsf{F}(\Lambda)),\] which when restricted to $H_0^{\gr}(\mathcal{G}_\Lambda)^+$, gives a $\mathbb{Z}^k$-monoid isomorphism between $H_0^{\gr}(\mathcal{G}_\Lambda)^+$ and $T_\Lambda$, the positive cone of $K_0^{\gr}(\KP_\mathsf{F}(\Lambda))$.     
\end{thm}
\begin{proof}
The graded zeroth homology of the $\mathbb{Z}^k$-graded groupoid $\mathcal{G}_\Lambda$ is the usual zeroth homology of the skew-product groupoid $\mathcal{G}_\Lambda\times_{\Tilde{d}} \mathbb{Z}^k$, i.e., \[H_0^{\gr}(\mathcal{G}_\Lambda)=H_0(\mathcal{G}_\Lambda\times_{\Tilde{d}} \mathbb{Z}^k)=C_c((\mathcal{G}_\Lambda\times_{\Tilde{d}}\mathbb{Z}^k)^{(0)},\mathbb{Z})/\image \partial_1,\] where $\partial_1:C_c(\mathcal{G}_\Lambda\times_{\Tilde{d}}\mathbb{Z}^k,\mathbb{Z})\longrightarrow C_c((\mathcal{G}_\Lambda\times_{\Tilde{d}}\mathbb{Z}^k)^{(0)},\mathbb{Z})$ is the map $s_*-r_*$. By \cite[Lemma 5.5]{H-Li}, \[\image \partial_1=\lspan_{\mathbb{Z}} \{1_{s(U)}-1_{r(U)}~|~U\in (\mathcal{G}_\Lambda \times_{\Tilde{d}}\mathbb{Z}^k)_h ^a\}.\] Any homogeneous compact open bisection of degree $\mathbf{p}$ in $\mathcal{G}_\Lambda\times_{\Tilde{d}}\mathbb{Z}^k$ is of the form $U=\displaystyle{\bigsqcup_{i=1}^{\ell}}~Z(\lambda_i,\mu_i)\times \{\N_i\}$, where $\ell\in \mathbb{N}$, $\lambda_i,\mu_i\in \Lambda$, $s(\lambda_i)=s(\mu_i)$, $\N_i\in \mathbb{Z}^k$ and $d(\lambda_i)-d(\mu_i)=\mathbf{p}$ for all $i=1,2,\ldots,\ell$. Using this and doing some simplification, we obtain \[\image \partial_1=\lspan_{\mathbb{Z}} \{1_{Z(\lambda)\times \{\N\}}-1_{Z(s(\lambda))\times \{\N+d(\lambda)\}}~|~\lambda\in \Lambda,\N\in \mathbb{Z}^k\}.\] By Proposition \ref{pro connecting path group with SA of path groupoid}, we have an isomorphism $\phi: C_c(\mathcal{G}_{\overline{\Lambda}}^{(0)},\mathbb{Z})\longrightarrow \mathcal{P}(\overline{\Lambda})$ such that $\phi(1_{Z((\lambda,\N))})=(\lambda,\N)$. Again, we note that $\mathcal{G}_{\overline{\Lambda}}^{(0)}\cong (\mathcal{G}_\Lambda\times_{\Tilde{d}}\mathbb{Z}^k)^{(0)}$. Combining these two facts, we have an isomorphism 
\begin{align*}
\xi:C_c((\mathcal{G}_\Lambda\times_{\Tilde{d}}\mathbb{Z}^k)^{(0)},\mathbb{Z})&\longrightarrow \mathcal{P}(\overline{\Lambda})\\
1_{Z(\lambda)\times \{\N\}}&\longmapsto (\lambda,\N).
\end{align*}
In $\mathcal{P}(\overline{\Lambda})$, consider the subgroup \[J:=\langle \xi(1_{Z(\lambda)\times\{\N\}}-1_{Z(s(\lambda))\times \{\N+d(\lambda)\}})~|~\lambda\in \Lambda,\N\in \mathbb{Z}^k \rangle=\langle (\lambda,\N)-(s(\lambda),\N+d(\lambda))~|~\lambda\in \Lambda,\N\in \mathbb{Z}^k\rangle.\] Then $\xi(\image \partial_1)=J$ and so $\xi$ induces an isomorphism \[\Tilde{\xi}:H_0^{\gr}(\mathcal{G}_\Lambda)=C_c((\mathcal{G}_\Lambda\times_{\Tilde{d}}\mathbb{Z}^k)^{(0)},\mathbb{Z})/\image \partial_1 \longrightarrow \mathcal{P}(\overline{\Lambda})/J.\] Clearly, $\xi$ respects the relevant $\mathbb{Z}^k$-actions, $\image \partial_1$ and $J$ are $\mathbb{Z}^k$-submodules. It follows that $\Tilde{\xi}$ is a $\mathbb{Z}[\mathbb{Z}^k]$-module isomorphism. Now, we show that the quotient group $\mathcal{P}(\overline{\Lambda})/J$ is isomorphic to the group completion of the talented monoid of $\Lambda$, and this isomorphism will be the remaining bit of our desired isomorphism. First, define a map \[f:\mathcal{P}(\overline{\Lambda})\longrightarrow G(T_\Lambda)\] on the generators by \[f((\lambda,\N)):=s(\lambda)(\N+d(\lambda)),\] and then extend linearly to all elements. For each $\M\in \mathbb{N}^k$, since 
\begin{align*}
f((\lambda,\N))=s(\lambda)(\N+d(\lambda))&=\displaystyle{\sum_{\mu\in s(\lambda)\Lambda^\M}}s(\mu)(\N+d(\lambda)+\M) ~(\text{by defining relations of}~ T_\Lambda)\\
&=\displaystyle{\sum_{\mu\in s(\lambda)\Lambda^\M}}s(\lambda\mu)(\N+d(\lambda\mu))\\
&=f\left(\displaystyle{\sum_{\mu\in s(\lambda)\Lambda^\M}} (\lambda\mu,\N)\right)\\
&=f\left(\displaystyle{\sum_{\mu\in s(\lambda)\Lambda^\M}} (\lambda,\N)(\mu,\N+d(\lambda))\right)\\
&=f\left(\displaystyle{\sum_{(\mu,\mathbf{p})\in \overline{s}(\lambda,\N)\overline{\Lambda}^\M}} (\lambda,\N)(\mu,\mathbf{p})\right),
\end{align*}
therefore, $f$ is a well-defined group homomorphism. Moreover, $J\subseteq \Ker(f)$ since $f((\lambda,\N)-(s(\lambda),\N+d(\lambda)))=s(\lambda)(\N+d(\lambda))-s(\lambda)(\N+d(\lambda))=0$. Thus, $f$ induces a group homomorphism \[\Tilde{f}:\mathcal{P}(\overline{\Lambda})/J\longrightarrow G(T_\Lambda).\] Now, consider a monoid homomorphism \[g:T_\Lambda\longrightarrow \mathcal{P}(\overline{\Lambda})/J\] defined on the generators by \[g(v(\N)):=[(v,\N)]_J.\] To show that $g$ is well-defined, it suffices to verify that $g$ is compatible with the defining relations of $T_\Lambda$. For any $\M\in \mathbb{N}^k$, in $T_\Lambda$, we have $v(\N)=\displaystyle{\sum_{\alpha\in v\Lambda^\M}} s(\alpha)(\N+\M)$. Now, \[g\left(\displaystyle{\sum_{\alpha\in v\Lambda^\M}}s(\alpha)(\N+\M)\right)=\displaystyle{\sum_{\alpha\in v\Lambda^\M}}[(s(\alpha),\N+\M)]_J=\displaystyle{\sum_{(\alpha,\N)\in (v,\N)\overline{\Lambda}^\M}} [(\alpha,\N)]_J=[(v,\N)]_J=g(v(\N)).\] Hence, $g$ is well-defined. By the universal property of group completion, there is a unique group homomorphism \[\Tilde{g}:G(T_\Lambda)\longrightarrow \mathcal{P}(\overline{\Lambda})/J\] such that $\Tilde{g}|_{T_\Lambda}=g$. It is easy to observe that $\Tilde{g}\circ \Tilde{f}$ and $\Tilde{f}\circ \Tilde{g}$ are respectively the identity morphisms on $\mathcal{P}(\overline{\Lambda})/J$ and $G(T_\Lambda)$, which shows that $\Tilde{f}$ is a group isomorphism. It also preserves the $\mathbb{Z}^k$-action since \[\Tilde{f}(\mathbf{q}\cdot (\lambda,\N))=\Tilde{f}((\lambda,\mathbf{q}+\N))=s(\lambda)(\mathbf{q}+\N+d(\lambda))=~^\mathbf{q} s(\lambda)(\N+d(\lambda))=~^\mathbf{q} \Tilde{f}((\lambda,\N)).\] Finally, defining $\Phi:=\Tilde{f}\circ \Tilde{\xi}$ and noting that $K_0^{\gr}(\KP_\mathsf{F}(\Lambda))$ is the group completion of $T_\Lambda$ (see \cite[Remark 3.15]{HMPS}), we have the desired $\mathbb{Z}[\mathbb{Z}^k]$-module isomorphism between $H_0^{\gr}(\mathcal{G}_\Lambda)$ and $K_0^{\gr}(\KP_\mathsf{F}(\Lambda))$. 

For the remaining part, observe that any $f\in C_c((\mathcal{G}_\Lambda\times_{\Tilde{d}} \mathbb{Z}^k)^{(0)},\mathbb{Z})^+$ can be written as \[f=\displaystyle{\sum_{i=1}^{t}} a_i 1_{Z(\lambda_i)\times \{\N_i\}},\] where $t\in \mathbb{N}$, $a_i\in \mathbb{Z}^+$, $\N_i\in \mathbb{Z}^k$ for each $t$ and $Z(\lambda_i)$'s are mutually disjoint. Under the isomorphism $\Tilde{\xi}$, the equivalence class of such an $f$ corresponds to $\displaystyle{\sum_{i=1}^{t}}a_i [(\lambda_i,\N_i)]_J$. Therefore, \[\Phi([f])=\Tilde{f}\left(\displaystyle{\sum_{i=1}^{t}}a_i[(\lambda_i,\N_i)]_J\right)=\displaystyle{\sum_{i=1}^{t}}a_i f((\lambda_i,\N_i))=\displaystyle{\sum_{i=1}^{t}}a_i s(\lambda_i)(\N_i+d(\lambda_i))\in T_\Lambda\] since each $a_i >0$. Conversely, if $x=\displaystyle{\sum_{j=1}^{N}}v_j(\N_j)\in T_\Lambda$, then \[\Phi^{-1}(x)=\Tilde{\xi}^{-1}\left(\displaystyle{\sum_{j=1}^{N}}[(v_j,\N_j)]_J\right)=[\displaystyle{\sum_{j=1}^{N}} 1_{Z(v_j)\times \{\N_j\}}]\in H_0^{\gr}(\mathcal{G}_\Lambda)^+,\] since $\displaystyle{\sum_{j=1}^{N}} 1_{Z(v_j)\times \{\N_j\}}\in C_c((\mathcal{G}_\Lambda\times_{\Tilde{d}} \mathbb{Z}^k)^{(0)},\mathbb{Z})^+$. This completes the proof. 
\end{proof}
\begin{rmk}\label{rem alternative proof incorporating type semigroup}
We describe an alternative proof of Theorem \ref{th relationship of graded K-theory and graded homology}. For an ample groupoid $\mathcal{G}$ with locally compact Hausdorff unit space $\mathcal{G}^{(0)}$, one can check that the rule $\psi: C_c(\mathcal{G}^{(0)},\mathbb{Z})^+\longrightarrow \Typ(\mathcal{G})$; $1_{U}\longmapsto \typ(U)$ ($U$ is a compact open subset of $\mathcal{G}^{(0)}$) induces an isomorphism of commutative monoids: \[\Psi:H_0(\mathcal{G})^+\longrightarrow \mathcal{C}(\Typ(\mathcal{G})),\] where $\mathcal{C}(\Typ(\mathcal{G}))$ is the \emph{universal cancellative abelian semigroup} \cite{ABBL} of the \emph{type monoid} $\Typ(\mathcal{G})$, i.e., \[\mathcal{C}(\Typ(\mathcal{G}))=\Typ(\mathcal{G})/\sim\] where $\sim$ is the equivalence relation on $\Typ(\mathcal{G})$ given by $x\sim y$ if and only if $x+z=y+z$ for some $z\in \Typ(\mathcal{G})$. For an ample groupoid $\mathcal{G}$ with compact unit space, an isomorphism between $H_0(\mathcal{G})^+$ and $\mathcal{C}(\Typ(\mathcal{G}))$ was established in \cite[Proposition 1.6]{ABBL} with a different (but canonically equivalent) description of type monoid of $\mathcal{G}$.

If $\mathcal{G}$ is equipped with a left $\Gamma$-action given by a group homomorphism $\eta:\Gamma\longrightarrow \Aut(\mathcal{G})$, then $H_0(\mathcal{G})^+$ becomes a $\Gamma$-monoid with respect to the action \[^\gamma [1_U]:=[1_U\circ \eta_{\gamma^{-1}}]=[1_U\circ (\eta_\gamma)^{-1}].\] On the other hand, $\Typ(\mathcal{G})$ is also equipped with a natural $\Gamma$-action, namely $^\gamma \typ(U):=\typ(\eta_\gamma(U))$. This in turn induces a $\Gamma$-action on $\mathcal{C}(\Typ(\mathcal{G}))$ defined by \[^\gamma [\typ(U)]:=[^\gamma \typ(U)].\] Now, it is a matter of routine verification that the isomorphism $\Psi$ mentioned above is indeed an isomorphism of $\Gamma$-monoids. 

Therefore, if we consider a $\Gamma$-graded groupoid $\mathcal{G}$ which is graded by a cocycle $c:\mathcal{G}\longrightarrow \Gamma$, the above observation, together with \cite[Definition 3.3]{Cordeiro}, implies that we have a $\Gamma$-monoid isomorphism \[H_0^{\gr}(\mathcal{G})^+=H_0(\mathcal{G}\times_c \Gamma)^+\cong \mathcal{C}(\Typ(\mathcal{G}\times_c \Gamma))=\mathcal{C}(\Typ^{\gr}(\mathcal{G})).\] Applying this to the path groupoid $\mathcal{G}_\Lambda$ together with the fact that the graded type monoid of $\mathcal{G}_\Lambda$ is isomorphic to the talented monoid of $\Lambda$ (see \cite[Theorem 3.18]{HMPS}), we have an isomorphism of $\mathbb{Z}^k$-monoids \[H_0^{\gr}(\mathcal{G}_\Lambda)^+\cong \mathcal{C}(\Typ^{\gr}(\mathcal{G}_\Lambda))\cong\mathcal{C}(T_\Lambda)=T_\Lambda,\] where the last equality follows since $T_\Lambda$ is a cancellative monoid (see \cite[Proposition 3.14 $(ii)$]{HMPS}). Finally, taking the group completion of both sides, we obtain an order-preserving isomorphism between $H_0^{\gr}(\mathcal{G}_\Lambda)$ and $K_0^{\gr}(\KP_\mathsf{F}(\Lambda))$. On the level of Leavitt path algebras, the said argument can be employed in view of \cite[Lemma 3.6]{Cordeiro}, to provide another proof of \cite[Theorem 6.6]{H-Li}.\qed
\end{rmk}

Given a $k$-graph $\Lambda$ and a monoid homomorphism $f:\mathbb{N}^\ell\longrightarrow \mathbb{N}^k$, recall that we can form the pullback $f^*(\Lambda)$ along $f$, which is an $\ell$-graph. We finish this section by exhibiting some pleasant connections between the homology theories of the groupoids $\mathcal{G}_{f^*(\Lambda)}$ and $\mathcal{G}_\Lambda$. 

\begin{prop}\label{pro homology of pullback graph groupoid}
Let $\ell,k$ be positive integers, $\Lambda$ a row-finite $k$-graph without sources and $f:\mathbb{N}^\ell\longrightarrow \mathbb{N}^k$ a surjective monoid homomorphism. Then 

$(i)$ $H_0(\mathcal{G}_{f^*(\Lambda)})\cong H_0(\mathcal{G}_\Lambda)$;

$(ii)$ there exists a surjective group homomorphism from $H_0^{\gr}(\mathcal{G}_{f^*(\Lambda)})$ onto $H_0^{\gr}(\mathcal{G}_\Lambda)$. 
\end{prop}
\begin{proof}
$(i)$ By \cite[Proposition 2.10]{Kumjian--Pask}, we have $\mathcal{G}_{f^*(\Lambda)}\cong \mathcal{G}_\Lambda\times \mathbb{Z}^{\ell-k}$. Applying the K\"{u}nneth theorem \cite[Theorem 2.4 \& Corollary 2.5]{MatAdv} for products of \'{e}tale groupoids, we obtain \[H_0(\mathcal{G}_{f^*(\Lambda)})\cong H_0(\mathcal{G}_\Lambda\times \mathbb{Z}^{\ell-k})\cong H_0(\mathcal{G}_\Lambda)\otimes H_0(\mathbb{Z})\otimes \underbrace{\cdots}_{\ell-k} \otimes~ H_0(\mathbb{Z}).\] But the integral zeroth homology of $\mathbb{Z}$ is $\mathbb{Z}$ and for any abelian group (equivalently, $\mathbb{Z}$-module) $\Gamma$, $\Gamma\otimes \mathbb{Z}\cong \Gamma$ via the canonical isomorphism $\gamma\otimes n\longmapsto n\cdot \gamma$. In view of these facts, we finally have $H_0(\mathcal{G}_{f^*(\Lambda)})\cong H_0(\mathcal{G}_\Lambda)$.

$(ii)$ Note that the set of objects $f^*(\Lambda)^0$ can be identified with $\Lambda^0$. The homomorphism $f$ clearly extends to a surjective homomorphism $f:\mathbb{Z}^\ell\longrightarrow \mathbb{Z}^k$. Consider the following correspondence 
\begin{align*}
    \rho:T_{f^*(\Lambda)}&\longrightarrow T_\Lambda\\
    v(\N) &\longmapsto v(f(\N)),
\end{align*}
for all $v\in f^*(\Lambda)^0$ and $\N\in \mathbb{Z}^\ell$. We claim that $\rho$, when extended linearly to all elements of $T_{f^*(\Lambda)}$, gives a well-defined monoid homomorphism. It suffices to check that the rule is compatible with the defining relations of the talented monoid. Let $v\in \Lambda^0,\N\in \mathbb{Z}^\ell$ and $\M\in \mathbb{N}^\ell$. Then, in $T_{f^*(\Lambda)}$, we have \[v(\N)=\displaystyle{\sum_{(\alpha,\M)\in vf^*(\Lambda)^\M}} s((\alpha,\M))(\N+\M).\] Now, 
\begin{align*}
\rho\left(\displaystyle{\sum_{(\alpha,\M)\in vf^*(\Lambda)^\M}}s((\alpha,\M))(\N+\M)\right)&=\displaystyle{\sum_{\alpha\in v\Lambda^{f(\M)}}} \rho\left(s(\alpha)(\N+\M)\right)\\
&=\displaystyle{\sum_{\alpha\in v\Lambda^{f(\M)}}} s(\alpha)(f(\N)+f(\M))\\
&=v(f(\N))\\
&=\rho(v(\N)).    
\end{align*}
Hence, our claim is established. Since $f$ is surjective, it follows from the definition that $\rho$ is also surjective. Since the group completion functor preserves epimorphisms, $\rho$ induces a surjective group homomorphism $\tilde{\rho}:K_0^{\gr}(\KP_\mathsf{F}(f^*(\Lambda)))\longrightarrow K_0^{\gr}(\KP_\mathsf{F}(\Lambda))$. This, together with Theorem \ref{th relationship of graded K-theory and graded homology}, finishes the proof.
\end{proof}
\begin{example}\label{ex computing homology of 2-graph via 1-graph}
Let $n\ge 2$ be an integer. Suppose $\Lambda$ is the $2$-graph with $1$-skeleton:
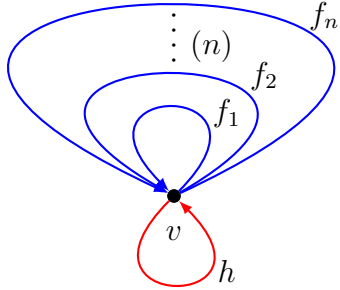
\[
\begin{tikzpicture}[scale=1]

\node[circle,draw,fill=black,inner sep=0.5pt] (p11) at (0, 0) {$.$} 
edge[-latex, blue,thick,loop, out=45, in=135, min distance=60] (p11)
edge[-latex, blue,thick, loop, out=30, in=150, min distance=120] (p11)
edge[-latex, blue,thick, loop, out=25, in=155, min distance=225] (p11)
edge[-latex, red,thick, loop, out=225, in=-45, min distance=60] (p11);
                                                              
\node at (0,1.8) {$\cdot$};
\node at (0,2) {$\cdot$};
\node at (0,2.2) {$\cdot$};
\node at (0,2.4) {$\cdot$};
\node at (0.5,2) {$(n)$};
\node at (0, -0.5) {$v$};

\node at (0.7,1.1) {$f_1$};
\node at (1.2,1.6) {$f_2$};
\node at (2,2.4) {$f_n$};
\node at (0.7,-1) {$h$};

\end{tikzpicture}
\]
and factorization rules: $hf_i=f_i h$ for all $i=1,2,\ldots,n$. We want to compute the zeroth homology of $\mathcal{G}_\Lambda$. For this, we will apply Proposition \ref{pro homology of pullback graph groupoid} $(i)$. Consider the surjective monoid homomorphism $f:\mathbb{N}^2\longrightarrow \mathbb{N}$ defined by $f((a,b)):=a$. We observe that the pullback $2$-graph $f^*(R_n)$ for the rose with $n$ petals
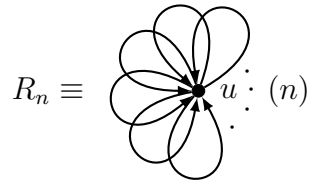
\[
\begin{tikzpicture}[scale=1]

\node[circle,draw,fill=black,inner sep=0.5pt] (p11) at (0, 0) {$.$} 
edge[-latex,thick,loop, out=30, in=100, min distance=50] (p11)
edge[-latex,thick, loop, out=70, in=140, min distance=50] (p11)
edge[-latex,thick, loop, out=110, in=180, min distance=50] (p11)
edge[-latex,thick, loop, out=150, in=220, min distance=50] (p11)
edge[-latex,thick, loop, out=190, in=260, min distance=50] (p11)
edge[-latex,thick, loop, out=230, in=300, min distance=50] (p11);

\node at (0.45,-0.5) {$\cdot$};
\node at (0.65,-0.25) {$\cdot$};
\node at (0.70,0.01) {$\cdot$};
\node at (0.65,0.25) {$\cdot$};
\node at (1.2,0) {$(n)$};
\node at (0.4,0) {$u$};
\node at (-2,0) {$R_n\equiv$};
\end{tikzpicture}
\]
is isomorphic to $\Lambda$. If we name the loops of $R_n$ as $c_1,c_2,\ldots,c_n$ and identify $u$ as $v$, $(u,(0,1))$ as $h$ and $(c_i,(1,0))$ as $f_i$ for each $i=1,2,\ldots,n$, we get a colored graph isomorphism between the $1$-skeletons of $f^*(R_n)$ and $\Lambda$. Moreover, the factorization rules for both these $2$-graphs agree and hence $\Lambda\cong f^*(R_n)$. Since the graph groupoid of $R_n$ is isomorphic to the groupoid arising from the one-sided full shift on a finite alphabet with $n$ symbols, therefore, $H_0(\mathcal{G}_{R_n})\cong \mathbb{Z}_{n-1}$. One can also establish this using the following relationships: $H_0(\mathcal{G}_{R_n})^+\cong \mathcal{C}(\Typ(\mathcal{G}_{R_n}))$ and $\Typ(\mathcal{G}_{R_n})\cong M_{R_n}\cong \mathbb{Z}_{n-1}\cup \{z\}$ (see Remark \ref{rem alternative proof incorporating type semigroup}, \cite[Theorem 7.5]{Bosa} and \cite[Examples 3.2.2 $(i)$]{Abrams-Monograph}). By applying Proposition \ref{pro homology of pullback graph groupoid} $(i)$, we finally obtain $H_0(\mathcal{G}_\Lambda)\cong H_0(\mathcal{G}_{R_n})\cong \mathbb{Z}_{n-1}$.
\end{example}

\section{The dimension group arising from a $k$-graph}\label{sec dimension group}
Throughout this section, we fix $\Lambda$ to be a row-finite $k$-graph without sources such that the set of objects $\Lambda^0$ is finite. For $\N\in \mathbb{N}^k$, we let $A_\N$ be the matrix in $\mathbb{M}_{\Lambda^0}(\mathbb{N})$ such that 
\[A_\N(u,v):=|u\Lambda^\N v|=|\{\lambda\in \Lambda~|~s(\lambda)=v,r(\lambda)=u,d(\lambda)=\N\}|,\]
for all $u,v\in \Lambda^0$ (see \cite{Evans, Kumjian--Pask, PSS}). These matrices are called the \emph{vertex matrices} of $\Lambda$. Using the unique factorization property, it can be shown that $A_\N A_\M=A_{\N+\M}$, and so $A_\N A_\M=A_\M A_\N$ for all $\N,\M\in \mathbb{N}^k$. Moreover, for $\N=(n_1,n_2,\ldots,n_k)$, we can write 
\begin{equation}\label{basemat}
A_\N=A_{\mathbf{e}_1}^{n_1}A_{\mathbf{e}_2}^{n_2}\cdots A_{\mathbf{e}_k}^{n_k}.  
\end{equation}

We now define a directed system of partially ordered groups using the matrices $A_\N$, $\N\in \mathbb{N}^k$. For each $\M\in \mathbb{Z}^k$, write $G_\M:=\mathbb{Z}\Lambda^0$, the free abelian group generated by $\Lambda^0$. Let $\M\le \N\in \mathbb{Z}^k$. Consider the group homomorphisms 
\begin{align*}
    f_{\M,\N}:G_\M &\longrightarrow G_\N\\
    x &\longmapsto xA_{\N-\M}.
\end{align*}
Since $A_0=I_{|\Lambda^0|}$, $(G_\M,f_{\M,\N})$ is a directed system. We define the \emph{dimension group} $H_\Lambda$ of $\Lambda$ to be the direct limit of this directed system, i.e., \[H_\Lambda:=\displaystyle{\varinjlim_{\mathbb{Z}^k}}~ (G_\N,f_{\M,\N}).\]
Another way to realise the dimension group is via an equivalence relation. Define a relation $\sim_\Lambda$ on $\mathbb{Z}\Lambda^0\times \mathbb{Z}^k$ by $(x,\N)\sim_\Lambda (y,\M)$ if $xA_{\mathbf{l}-\N}=yA_{\mathbf{l}-\M}$ for some $\mathbf{l}\in \mathbb{Z}^k$ with $\mathbf{l}\ge \N\vee \M$ or, equivalently, if \[x\left(\displaystyle{\prod_{i=1}^{k}} A_{\mathbf{e}_i}^{l_i-n_i}\right)=y\left(\displaystyle{\prod_{i=1}^{k}} A_{\mathbf{e}_i}^{l_i-m_i}\right).\] It is easy to check that $\sim_\Lambda$ is an equivalence relation and $H_\Lambda=(\mathbb{Z}\Lambda^0\times \mathbb{Z}^k)/\sim_\Lambda$.

The group $H_\Lambda$ is partially ordered with respect to the usual coordinatewise ordering of $\mathbb{Z}\Lambda^0$ and the positive cone is
\[H_\Lambda^+:=\{[x,\N]\in H_\Lambda~|~x\in \mathbb{N}\Lambda^0\}=\displaystyle{\varinjlim_{\mathbb{Z}^k}}~ (\mathbb{N}\Lambda^0,f_{\M,\N}).\]
We call $H_\Lambda^+$, the \emph{dimension monoid} of $\Lambda$. The addition in the group $H_\Lambda$ (and hence, in the monoid $H_\Lambda^+$) is defined as \[[x,\N]+[y,\M]:=[xA_{(\N\vee \M) -\N}+yA_{(\N\vee \M) -\M},\N\vee \M].\] 

For each $i=1,2,\ldots,k$, the map 
\begin{align}\label{shifts}
\delta_i:H_\Lambda &\longrightarrow H_\Lambda\\
\nonumber [x,\N] &\longmapsto [xA_{\mathbf{e}_i},\N]=[x,\N-\mathbf{e}_i],
\end{align}
is an automorphism with inverse $\delta_i^{-1}:H_\Lambda\longrightarrow H_\Lambda$; $[x,\N]\longmapsto [x,\N+\mathbf{e}_i]$. These automorphisms are mutually commuting and preserve the positive cone $H_\Lambda^+$. As a consequence, $H_\Lambda^+$ becomes a $\mathbb{Z}^k$-monoid with respect to the following action:
\begin{equation}\label{acteqn}
^\M[x,\N]:=(\delta_1^{-m_1}\circ \delta_2^{-m_2}\circ \cdots \circ \delta_{k}^{-m_k})([x,\N])=[x,\N+\M], 
\end{equation}
for all $\M=(m_1,m_2,\ldots,m_k)\in \mathbb{N}^k$ and $[x,\N]\in H_\Lambda^+$. 

In \cite[Proposition 3.14]{HMPS}, the talented monoid $T_\Lambda$ was shown to be the direct limit of the directed system $(\mathbb{N}\Lambda^0,\phi_{\M,\N})$, where the monoid homomorphisms $\phi_{\M,\N}$ are defined as $\phi_{\M,\N}(v):=\displaystyle{\sum_{w\in \Lambda^0}} |v\Lambda^{\N-\M}w| w$. Now, for any $\M\le \N\in \mathbb{Z}^k$, we observe that \[f_{\M,\N}(\epsilon_v)(w)=\displaystyle{\sum_{u\in \Lambda^0}} \epsilon_v(u)A_{\N-\M}(u,w)=A_{\N-\M}(v,w)=|v\Lambda^{\N-\M}w|=\phi_{\M,\N}(\epsilon_v)(w),\] for all $v,w\in \Lambda^0$. This shows that $f_{\M,\N}=\phi_{\M,\N}$ in $\End(\mathbb{N}\Lambda^0)$ and consequently, the talented monoid $T_\Lambda$ and the positive cone $H_\Lambda^+$ of the dimension group of $\Lambda$ are isomorphic as monoids. In fact, we exhibit an explicit $\mathbb{Z}^k$-monoid isomorphism between these two in the following theorem.

\begin{thm}\label{th isomorphism of talented monoid and dimension semigroup}
Let $\Lambda$ be a row-finite $k$-graph with no sources such that $|\Lambda^0|< \infty$. Consider an enumeration of the set of objects, $\Lambda^0=\{v_1,v_2,\ldots,v_{|\Lambda^0|}\}$. Suppose $\psi:T_\Lambda\longrightarrow H_\Lambda^+$ is a map defined on generators by
\[v_i(\N)\longmapsto [\epsilon_i,\N],\]
and then extended linearly. Then, $\psi$ is a $\mathbb{Z}^k$-monoid isomorphism. Consequently, there is an order-preserving $\mathbb{Z}[\mathbb{Z}^k]$-module isomorphism $\gamma_\Lambda$ between the graded Grothendieck group $K_0^{\gr}(\KP_{\mathsf{F}}(\Lambda))$ and the dimension group $H_\Lambda$.
\end{thm}

\begin{proof}
First, we show that $\psi$ is well-defined. It suffices to show that $\psi$ is compatible with the defining relations of $T_\Lambda$. Let $i\in \{1,2,\ldots,|\Lambda^0|\}$, $\N\in \mathbb{Z}^k$ and $\M\in \mathbb{N}^k$. Then, in $T_\Lambda$, we have \[v_i(\N)=\displaystyle{\sum_{\alpha\in v_i\Lambda^\M}} s(\alpha)(\N+\M).\]
Now, 
\begin{align*}
\psi\left(\displaystyle{\sum_{\alpha\in v_i\Lambda^\M}} s(\alpha)(\N+\M)\right) &=\psi\left(\displaystyle{\sum_{j=1}^{|\Lambda^0|}} |v_i\Lambda^m v_j|v_j(\N+\M)\right)\\
&=\displaystyle{\sum_{j=1}^{|\Lambda^0|}} |v_i\Lambda^m v_j|[\epsilon_j,\N+\M]\\
&=\displaystyle{\sum_{j=1}^{|\Lambda^0|}} [|v_i\Lambda^m v_j|\epsilon_j,\N+\M]\\
&=[\displaystyle{\sum_{j=1}^{|\Lambda^0|}}|v_i\Lambda^m v_j|\epsilon_j,\N+\M]=[\epsilon_iA_\M,\N+\M]=[\epsilon_i,\N]=\psi(v_i(\N)).
\end{align*}
Hence, $\psi$ is a well-defined monoid homomorphism. To show that $\psi$ is injective, choose any $x,y\in T_\Lambda$ such that $\psi(x)=\psi(y)$. Since $\Lambda$ has no sources, we can write $x=\displaystyle{\sum_{i=1}^{|\Lambda^0|}} x_i v_i(\N)$ and similarly, $y=\displaystyle{\sum_{i=1}^{|\Lambda^0|}} y_i v_i(\M)$, where $\N,\M\in \mathbb{Z}^k$ and $x_i,y_i\in \mathbb{N}$ for all $i=1,2,\ldots,|\Lambda^0|$. Then, $\psi(x)=\psi(y)$ implies that there exists $\mathbf{l}\in \mathbb{Z}^k$ with $\mathbf{l}\ge \N\vee \M$ such that
\begin{equation}\label{mateqn}
(x_1,x_2,\ldots,x_{|\Lambda^0|})A_{\mathbf{l}-\N}=(y_1,y_2,\ldots,y_{|\Lambda^0|})A_{\mathbf{l}-\M}.   
\end{equation}
Now using the relations of $T_\Lambda$ and Equation (\ref{mateqn}), we have
\begin{align*}
    x=\displaystyle{\sum_{i=1}^{|\Lambda^0|}} x_i v_i(\N) 
    &=\displaystyle{\sum_{i=1}^{|\Lambda^0|}} x_i\left(\displaystyle{\sum_{j=1}^{|\Lambda^0|}}|v_i\Lambda^{\mathbf{l}-\N}v_j|v_j(\mathbf{l})\right)\\    &=\displaystyle{\sum_{j=1}^{|\Lambda^0|}}\left(\displaystyle{\sum_{i=1}^{|\Lambda^0|}}|v_i\Lambda^{\mathbf{l}-\N}v_j|x_i\right) v_j(\mathbf{l})\\    &=\displaystyle{\sum_{j=1}^{|\Lambda^0|}}\left(\displaystyle{\sum_{i=1}^{|\Lambda^0|}}|v_i\Lambda^{\mathbf{l}-\M}v_j|y_i\right) v_j(\mathbf{l})=\displaystyle{\sum_{i=1}^{|\Lambda^0|}} y_i\left(\displaystyle{\sum_{j=1}^{|\Lambda^0|}}|v_i\Lambda^{\mathbf{l}-\M}v_j|v_j(\mathbf{l})\right)=\displaystyle{\sum_{i=1}^{|\Lambda^0|}} y_i v_i(\M)=y.
\end{align*}
By definition $\psi$ is surjective. Again, it is easy to see that $\psi$ respects the $\mathbb{Z}^k$-action on the two monoids and this completes the proof of the first part. For the remaining part, let $\gamma_\Lambda:G(T_\Lambda)\longrightarrow G(H_\Lambda^+)$ be the map induced by $\psi$ under the group completion functor. The result now follows immediately in view of \cite[Remark 3.15]{HMPS}. 
\end{proof}

\begin{rmk}\label{rem Connecting K-theory of Kp-algebra and C*-algebra}
The directed system we have considered is the same as that of Evans \cite{Evans}. In \cite[Lemma 3.9]{Evans}, he established that $K_0(C^*(\Lambda\times_d \ZZ^k))\cong \displaystyle{\varinjlim_{\mathbb{Z}^k}}~ (G_\N,f_{\M,\N})$. Relating this with Theorem \ref{th isomorphism of talented monoid and dimension semigroup}, one immediately obtains $K_0^{\gr}(\KP_\mathsf{F}(\Lambda))\cong K_0(C^*(\Lambda\times_d \ZZ^k))$. 
\end{rmk}

Theorem \ref{th isomorphism of talented monoid and dimension semigroup} sets up the intrinsic connection between the graded $K$-theory and dimension group on the level of $k$-graphs and Kumjian--Pask algebras. Via the isomorphism $\gamma_\Lambda$, we can conveniently switch between these two; this is going to be a key tool in our investigation. 

\begin{lem}\label{lem the connecting matrix R}
Let $\Lambda$ and $\Omega$ be row-finite $k$-graphs without sources such that both $\Lambda^0,\Omega^0$ are finite. For each $\N\in \mathbb{N}^k$, suppose $A_\N\in \mathbb{M}_{\Lambda^0}(\mathbb{N})$ and $B_\N\in \mathbb{M}_{\Omega^0}(\mathbb{N})$ are matrices such that $A_\N(u,v):=|u\Lambda^\N v|$ for $u,v\in \Lambda^0$ and $B_\N(x,y):=|x\Omega^\N y|$ for $x,y\in \Omega^0$. 

$(i)$ Suppose there exists a matrix $R\in \mathbb{M}_{\Lambda^0\times \Omega^0}(\mathbb{N})$ such that $A_{\mathbf{e}_i}R=RB_{\mathbf{e}_i}$ for all $i=1,2,\ldots,k$. Then, the map
\begin{align*}
    R:H_\Lambda &\longrightarrow H_\Omega\\
        [x,\N] &\longmapsto [xR,\N]
\end{align*}
is an order-preserving $\mathbb{Z}[x_1^{\pm 1},x_2^{\pm 1},\ldots,x_k^{\pm 1}]$-module homomorphism. Consequently, there exists a unique order-preserving $\mathbb{Z}[x_1^{\pm 1},x_2^{\pm 1},\ldots,x_k^{\pm 1}]$-module homomorphism $\mathfrak{h}_R:K_0^{\gr}(\KP_{\mathsf{F}}(\Lambda))\longrightarrow K_0^{\gr}(\KP_{\mathsf{F}}(\Omega))$ such that the diagram
\begin{equation}\label{comd1}
\begin{tikzcd}[row sep=huge, column sep=huge]
		K_0^{\gr}(\KP_\mathsf{F}(\Lambda)) \arrow [r, "\mathfrak{h}_R"] \arrow [d, "\gamma_\Lambda"'] & K_0^{\gr}(\KP_\mathsf{F}(\Omega)) \arrow [d, "\gamma_\Omega"] \\
		H_\Lambda \arrow [r, "R"] & H_\Omega
\end{tikzcd}    
\end{equation}
commutes.

$(ii)$ If there exists an order-preserving $\mathbb{Z}[x_1^{\pm 1},x_2^{\pm 1},\ldots,x_k^{\pm 1}]$-module homomorphism \[\mathfrak{h}:K_0^{\gr}(\KP_{\mathsf{F}}(\Lambda))\longrightarrow K_0^{\gr}(\KP_{\mathsf{F}}(\Omega)),\] then there is a matrix $R\in \mathbb{M}_{\Lambda^0\times \Omega^0} (\mathbb{N})$ and integers $r_1,r_2,\ldots,r_k\in \mathbb{Z}$ such that $A_{\mathbf{e}_i}R=RB_{\mathbf{e}_i}$ for all $i=1,2,\ldots,k$ and \[\mathfrak{h}(w)=~^{x_1^{r_1}x_2^{r_2}\cdots x_k^{r_k}} \mathfrak{h}_R(w),\] for all $w\in K_0^{\gr}(\KP_\mathsf{F}(\Lambda))$.
\end{lem}
\begin{proof}
$(i)$ It is easy to show that $A_{\mathbf{e}_i}^n R=RB_{\mathbf{e}_i}^n$ for all $N\in \mathbb{N}$. Using (\ref{basemat}), it follows that
\begin{equation}\label{commeqn}
A_\N R=RB_\N
\end{equation}
for any $\N\in \mathbb{N}^k$. Suppose $[x,\N]=[y,\M]$ in $H_\Lambda$. Then, there is some $\mathbf{l}\in \mathbb{Z}^k$ with $\mathbf{l}\ge \N\vee \M$, such that $xA_{\mathbf{l}-\N}=yA_{\mathbf{l}-\M}$. Using this and Equation (\ref{commeqn}), we have \[(xR)B_{\mathbf{l}-\N}=(xA_{\mathbf{l}-\N})R=(yA_{\mathbf{l}-\M})R=(yR)B_{\mathbf{l}-\M}.\]
So $[xR,\N]=[yR,\M]$, and hence $R$ is well-defined. A routine computation using (\ref{commeqn}), yields that it is a group homomorphism. It is also order-preserving since the matrix $R$ is a non-negative integral matrix. Note that the $\mathbb{Z}[x_1^{\pm 1},x_2^{\pm 1},\ldots,x_k^{\pm 1}]$-module structure on the dimension groups is coming from the $\mathbb{Z}^k$-monoid actions on the respective positive cones (see (\ref{acteqn})). So
\[^{x_1^{m_1}x_2^{m_2}\cdots x_k^{m_k}}[xR,\N]=[xR,\N+\displaystyle{\sum_{i=1}^{k}} m_i\mathbf{e}_i]=R([x,\N+\displaystyle{\sum_{i=1}^{k}} m_i\mathbf{e}_i])=R(^{x_1^{m_1}x_2^{m_2}\cdots x_k^{m_k}}[x,\N]).\]
This shows that $R$ is a module homomorphism. Now, define a map
\begin{align*}
    \mathfrak{h}_R:K_0^{\gr}(\KP_{\mathsf{F}}(\Lambda)) &\longrightarrow K_0^{\gr}(\KP_{\mathsf{F}}(\Omega))\\
    w & \longmapsto (\gamma_\Omega^{-1}\circ R \circ \gamma_\Lambda)(w)
\end{align*}
for all $w\in K_0^{\gr}(\KP_{\mathsf{F}}(\Lambda))$. Clearly, $\mathfrak{h}_R$ is the unique $\mathbb{Z}[x_1^{\pm 1},x_2^{\pm 1},\ldots,x_k^{\pm 1}]$-module homomorphism making diagram (\ref{comd1}) commute.

$(ii)$ Suppose we have an order-preserving $\mathbb{Z}[x_1^{\pm 1},x_2^{\pm 1},\ldots,x_k^{\pm 1}]$-module homomorphism \[\mathfrak{h}:K_0^{\gr}(\KP_{\mathsf{F}}(\Lambda))\longrightarrow K_0^{\gr}(\KP_{\mathsf{F}}(\Omega)).\] Then we can obtain an order-preserving $\mathbb{Z}[\mathbb{Z}^k]$-module homomorphism $\theta: H_\Lambda\longrightarrow H_\Omega$ such that the diagram
\begin{center}
\begin{tikzcd}[row sep=huge, column sep=huge]
		K_0^{\gr}(\KP_\mathsf{F}(\Lambda)) \arrow [r, "\mathfrak{h}"] \arrow [d, "\gamma_\Lambda"'] & K_0^{\gr}(\KP_\mathsf{F}(\Omega)) \arrow [d, "\gamma_\Omega"] \\
		H_\Lambda \arrow [r, "\theta"] & H_\Omega
\end{tikzcd} 
\end{center}
commutes, namely $\theta=\gamma_\Omega \circ \mathfrak{h} \circ \gamma_\Lambda^{-1}$.

To show $(ii)$, it suffices to find a matrix $R\in \mathbb{M}_{\Lambda^0\times \Omega^0}(\mathbb{N})$ and integers $r_1,r_2,\ldots,r_k$ such that \[A_{\mathbf{e}_i}R=RB_{\mathbf{e}_i}\] for all $i=1,2,\ldots,k$ and \[\theta([x,\N])=~^{x_1^{r_1}x_2^{r_2}\cdots x_k^{r_k}}[xR,\N]\] for all $[x,\N]\in H_\Lambda$. 

Let $\Lambda^0=\{v_1,v_2,\ldots,v_{|\Lambda^0|}\}$ and $[y^{(i)},\mathbf{l}^{(i)}]:=\theta([\epsilon_i,0])$ for all $i=1,2,\ldots,|\Lambda^0|$. Set $\M:=\displaystyle{\bigvee_{i=1}^{|\Lambda^0|}} \mathbf{l}^{(i)}$. Then, \[[y^{(i)},\mathbf{l}^{(i)}]=[y^{(i)}B_{\M-\mathbf{l}^{(i)}},\M]\] for all $i=1,2,\ldots,|\Lambda^0|$. Since $\theta$ is order-preserving, each $y^{(i)}\in \mathbb{N}^{\Omega^0}$ and consequently, $y^{(i)}B_{\M-\mathbf{l}^{(i)}}\in \mathbb{N}^{\Omega^0}$. Now, we define a matrix $R'\in \mathbb{M}_{\Lambda^0\times \Omega^0}(\mathbb{N})$ such that the $i^{\text{th}}$ row of $R'$ is $y^{(i)}B_{\M-\mathbf{l}^{(i)}}$ for every $1\le i \le |\Lambda^0|$. Let $[x,\N]\in H_\Lambda$. Since $\theta$ is a $\mathbb{Z}[\mathbb{Z}^k]$-module homomorphism, we have
\begin{align*}
\theta([x,\N])=\theta([\displaystyle{\sum_{i=1}^{|\Lambda^0|}} x_i\epsilon_i,\N]) &=\displaystyle{\sum_{i=1}^{|\Lambda^0|}} \theta([x_i\epsilon_i,\N])\\
&=~^\N\left(\displaystyle{\sum_{i=1}^{|\Lambda^0|}} x_i\theta([\epsilon_i,0])\right)\\
&=~^\N(\displaystyle{\sum_{i=1}^{|\Lambda^0|}} x_i[y^{(i)}B_{\M-\mathbf{l}^{(i)}},\M])\\
&=~^\N[\displaystyle{\sum_{i=1}^{|\Lambda^0|}} x_iy^{(i)}B_{\M-\mathbf{l}^{(i)}},\M]=~^\N[xR',\M]=[xR',\M+\N].
\end{align*}

Now. for any $i\in \{1,2,\ldots,k\}$ and $j\in \{1,2,\ldots,|\Lambda^0|\}$, we have 
\[[\epsilon_jA_{\mathbf{e}_i}R',\M]=\theta([\epsilon_jA_{\mathbf{e}_i},0])=\theta([\epsilon_j,-\mathbf{e}_i])=[\epsilon_jR',\M-\mathbf{e}_i]=[\epsilon_jR' B_{\mathbf{e}_i},\M].\]
Therefore, there is some $\mathbf{r}(i,j)\in \mathbb{N}^k$ such that $(\epsilon_jA_{\mathbf{e}_i}R')B_{\mathbf{r}(i,j)}=(\epsilon_jR' B_{\mathbf{e}_i})B_{\mathbf{r}(i,j)}$. Set $\mathbf{r}=\displaystyle{\bigvee_{\substack{1\le i\le k\\1\le j \le |\Lambda^0|}}} \mathbf{r}(i,j)$. Then, $\epsilon_jA_{\mathbf{e}_i}R'B_\mathbf{r}=\epsilon_jR' B_{\mathbf{e}_i}B_\mathbf{r}$ for all $1\le i\le k$ and $1\le j\le |\Lambda^0|$. Finally, we define the matrix $R:=R'B_\mathbf{r}$. So for any $1\le i\le k$, $\epsilon_jA_{\mathbf{e}_i}R=\epsilon_jRB_{\mathbf{e}_i}$ for all $j=1,2,\ldots,|\Lambda^0|$. Thus, $A_{\mathbf{e}_i}R=RB_{\mathbf{e}_i}$ for all $i=1,2,\ldots,k$. Also, \[\theta([x,\N])=[xR',\M+\N]=[xR'B_\mathbf{r},\M+\N+\mathbf{r}]=[xR,\M+\N+\mathbf{r}]=~^{x_1^{m_1+r_1}x_2^{m_2+r_2}\cdots x_k^{m_k+r_k}}[xR,\N],\] for all $[x,\N]\in H_\Lambda$ and we are done.
\end{proof}

In the following result, we connect isomorphism between dimension groups of two $k$-graphs with a certain relation between their vertex matrices and also with isomorphism between the graded homologies of their respective path groupoids.

\begin{thm}\label{th isomorphism of TM in terms of matrices}
Let $\Lambda$ and $\Omega$ be row-finite $k$-graphs without sources such that both $\Lambda^0,\Omega^0$ are finite. For each $\N\in \mathbb{N}^k$, let $A_\N\in \mathbb{M}_{\Lambda^0}(\mathbb{N})$ and $B_\N\in \mathbb{M}_{\Omega^0}(\mathbb{N})$ be the matrices given by $A_\N(u,v):=|u\Lambda^\N v|$ for $u,v\in \Lambda^0$ and $B_\N(x,y):=|x\Omega^\N y|$ for $x,y\in \Omega^0$. Then, the following are equivalent.

$(i)$ There exist $R\in \mathbb{M}_{\Lambda^0\times \Omega^0}(\mathbb{N})$, $S\in \mathbb{M}_{\Omega^0\times \Lambda^0}(\mathbb{N})$ and $\mathbf{p}\in \mathbb{N}^k$ such that \[A_\mathbf{p}=RS,~~~~~~~~~~~~~~~~~~~~B_\mathbf{p}=SR\] and \[A_{\mathbf{e}_i}R=RB_{\mathbf{e}_i},~~~~~~~~~~~~~~~~~~~~~~~~~~~~ B_{\mathbf{e}_i}S=SA_{\mathbf{e}_i}\] for all $i=1,2,\ldots,k$.

$(ii)$ There exists a group isomorphism $\phi: H_\Lambda \longrightarrow H_\Omega$ such that $\phi(H_\Lambda^+)=H_\Omega^+$ and $\phi \circ \delta_i^\Lambda=\delta_i^\Omega \circ \phi$ for all $i=1,2,\ldots,k$, where $\delta_i^\Lambda$, $\delta_i^\Omega$ are the respective automorphisms of $H_\Lambda$ and $H_\Omega$ described in \textup{(\ref{shifts})}.

$(iii)$ The talented monoids $T_\Lambda$ and $T_\Omega$ are isomorphic as $\mathbb{Z}^k$-monoids.

$(iv)$ There is a $\mathbb{Z}[\mathbb{Z}^k]$-module isomorphism $\varphi:H_0^{\gr}(\mathcal{G}_\Lambda)\longrightarrow H_0^{\gr}(\mathcal{G}_\Omega)$ such that $\varphi(H_0^{\gr}(\mathcal{G}_\Lambda)^+)=H_0^{\gr}(\mathcal{G}_\Omega)^+$.

Furthermore, the isomorphism $\phi$ in $(ii)$ is pointed, i.e., preserves the order units if and only if in $(iv)$, $\varphi([1_{\Lambda^\infty\times \{0\}}])=[1_{\Omega^\infty\times \{0\}}]$.
\end{thm}
\begin{proof}
The implications $(ii)\Longleftrightarrow (iii)$ and $(ii)\Longleftrightarrow (iv)$ follow from Theorem \ref{th relationship of graded K-theory and graded homology} and Theorem \ref{th isomorphism of talented monoid and dimension semigroup}. Note that since $\Lambda^0,\Omega^0$ are finite, we have that $\Lambda^\infty\times \{0\}$ and $\Omega^\infty\times \{0\}$ are compact. The isomorphism in $(ii)$ is pointed, i.e., takes $[(1~1~\cdots~1)_{1\times |\Lambda^0|},0]$ to $[(1~1~\cdots~1)_{1\times |\Omega^0|},0]$ if and only if, at the level of talented monoids, $\displaystyle{\sum_{v\in \Lambda^0}}v(0)\longmapsto \displaystyle{\sum_{w\in \Omega^0}} w(0)$, which in turn is equivalent (in view of the isomorphism $\Phi$ of Theorem \ref{th relationship of graded K-theory and graded homology}) to the condition that \[\varphi([1_{\Lambda^\infty\times \{0\}}])=\varphi\left(\displaystyle{\sum_{v\in \Lambda^0}}[1_{Z(v)\times \{0\}}]\right)=\displaystyle{\sum_{w\in \Omega^0}}[1_{Z(w)\times \{0\}}]=[1_{\Omega^\infty\times \{0\}}].\] In the remaining part of the proof, we show that $(i)\Longleftrightarrow (ii)$. 

$(i)\Longrightarrow (ii)$. Consider the following map
\begin{align*}
    \phi:H_\Lambda &\longrightarrow H_\Omega\\
    [x,\N] &\longmapsto [xR,\N]
\end{align*}
for all $[x,\N]\in H_\Lambda$. Using the same argument as in the proof of Lemma \ref{lem the connecting matrix R} $(i)$, it can be shown that $\phi$ is a well-defined group homomorphism. For injectivity, choose $[x,\N],[y,\M]\in H_\Lambda$ such that $\phi([x,\N])=\phi([y,\M])$. Then, $[xR,\N]=[yR,\M]$ in $H_\Omega$. Hence, we have an $\mathbf{l}\in \mathbb{Z}^k$ with $\mathbf{l}\ge \N\vee \M$ such that $(xR)B_{\mathbf{l}-\N}=(yR)B_{\mathbf{l}-\M}$. Since $\mathbf{p}\in \mathbb{N}^k$, $\mathbf{l}+\mathbf{p}\ge \N\vee \M$. Now, we have
\[xA_{(\mathbf{l}+\mathbf{p})-\N}=xA_{\mathbf{l}-\N}A_\mathbf{p}=(xA_{\mathbf{l}-\N}R)S=(xRB_{\mathbf{l}-\N})S=(yRB_{\mathbf{l}-\M})S=(yA_{\mathbf{l}-\M}R)S=yA_{\mathbf{l}-\M}A_\mathbf{p}=yA_{(\mathbf{l}+\mathbf{p})-\M}.\] Thus, $[x,\N]=[y,\M]$. For surjectivity, let $[w,\M]\in H_\Omega$. Then, $[wS,\M+\mathbf{p}]\in H_\Lambda$. Now, \[\phi([wS,\M+\mathbf{p}])=[wSR,\M+\mathbf{p}]=[wB_\mathbf{p},\M+\mathbf{p}]=[w,\M].\] Therefore, $\phi$ is an isomorphism. Since $R$ and $S$ are non-negative integral matrices, it follows that $\phi(H_\Lambda^+)=H_\Omega^+$. That $\phi$ intertwines $\delta_i^\Lambda$ with $\delta_i^\Omega$, is evident from its definition.

$(i)\Longleftarrow (ii)$. Suppose $\phi:H_\Lambda\longrightarrow H_\Omega$ is an order-preserving $\mathbb{Z}[\mathbb{Z}^k]$-module isomorphism with inverse $\psi:H_\Omega\longrightarrow H_\Lambda$. In view of Lemma \ref{lem the connecting matrix R} $(ii)$, we have matrices $R\in \mathbb{M}_{\Lambda^0\times \Omega^0}(\mathbb{N})$, $S\in \mathbb{M}_{\Omega^0\times \Lambda^0}(\mathbb{N})$ and $\mathbf{r},\mathbf{s}\in \mathbb{Z}^k$ such that \[A_{\mathbf{e}_i}R=RB_{\mathbf{e}_i}, ~~~~~~~~~B_{\mathbf{e}_i}S=SA_{\mathbf{e}_i}\] for all $i=1,2,\ldots,k$ and \[\phi([x,\N])=~^\mathbf{r}[xR,\N],~~~~~~~~~\psi([y,\M])=~^\mathbf{s}[yS,\M]\] for all $[x,\N]\in H_\Lambda$ and $[y,\M]\in H_\Omega$. 

Now, for any $j=1,2,\ldots,|\Omega^0|$, $(\phi\psi)([\epsilon_j,0])=([\epsilon_j,0])$. So $[\epsilon_jSR,\mathbf{s}+\mathbf{r}]=[\epsilon_j,0]$. Therefore, \[(\epsilon_jSR)B_{\mathbf{t}^{(j)}-(\mathbf{s}+\mathbf{r})}=\epsilon_jB_{\mathbf{t}^{(j)}}\] for some $\mathbf{t}^{(j)}\ge (\mathbf{s}+\mathbf{r})\vee 0$. Let $\mathbf{t}=\displaystyle{\bigvee_{j=1}^{|\Omega^0|}}\mathbf{t}^{(j)}$. Then, for all $j=1,2,\ldots,|\Omega^0|$, we have $(\epsilon_jSR)B_{\mathbf{t}-(\mathbf{s}+\mathbf{r})}=\epsilon_jB_\mathbf{t}$, which implies $S(RB_{\mathbf{t}-(\mathbf{s}+\mathbf{r})})=B_\mathbf{t}$. Similarly, using the fact that $\psi\phi=id_{H_\Lambda}$, we can find $\mathbf{q}\ge (\mathbf{s}+\mathbf{r})\vee 0$ such that $R(SA_{\mathbf{q}-(\mathbf{s}+\mathbf{r})})=A_\mathbf{q}$. Now, set $\mathcal{R}=RB_{\mathbf{t}-(\mathbf{s}+\mathbf{r})}$ and $\mathcal{S}=SA_{\mathbf{q}-(\mathbf{s}+\mathbf{r})}$. Then, \[A_{\mathbf{e}_i}\mathcal{R}=A_{\mathbf{e}_i}RB_{\mathbf{t}-(\mathbf{s}+\mathbf{r})}=RB_{\mathbf{e}_i}B_{\mathbf{t}-(\mathbf{s}+\mathbf{r})}=\mathcal{R}B_{\mathbf{e}_i}.\] Similarly, $B_{\mathbf{e}_i}\mathcal{S}=\mathcal{S}B_{\mathbf{e}_i}$ for all $i=1,2,\ldots,k$. Finally, \[\mathcal{R}\mathcal{S}=(RB_{\mathbf{t}-(\mathbf{s}+\mathbf{r})})(SA_{\mathbf{q}-(\mathbf{s}+\mathbf{r})})=RSA_{\mathbf{t}-(\mathbf{s}+\mathbf{r})}A_{\mathbf{q}-(\mathbf{s}+\mathbf{r})}=A_\mathbf{q}A_{\mathbf{t}-(\mathbf{s}+\mathbf{r})}=A_{(\mathbf{q}+\mathbf{t})-(\mathbf{s}+\mathbf{r})},\] and \[\mathcal{S}\mathcal{R}=(SA_{\mathbf{q}-(\mathbf{s}+\mathbf{r})})(RRB_{\mathbf{t}-(\mathbf{s}+\mathbf{r})})=SRB_{\mathbf{q}-(\mathbf{s}+\mathbf{r})}B_{\mathbf{t}-(\mathbf{s}+\mathbf{r})}=B_\mathbf{t}B_{\mathbf{q}-(\mathbf{s}+\mathbf{r})}=B_{(\mathbf{q}+\mathbf{t})-(\mathbf{s}+\mathbf{r})}.\] Now, letting $\mathbf{p}=(\mathbf{q}+\mathbf{t})-(\mathbf{s}+\mathbf{r})$, we are done. 
\end{proof}
\begin{rmk}\label{rem matrix connection of in-splitting}
Suppose $v$ is a vertex in $\Lambda$ and $\Lambda_I$ is the in-split of $\Lambda$ at $v$ with respect to a partition $v\Lambda^\mathbf{1}=\mathcal{E}_1\cup \mathcal{E}_2$. Moreover, suppose $A_\N,B_\N$, $\N\in \mathbb{N}^k$ are the vertex matrices for $\Lambda$ and $\Lambda_I$ respectively. We observe that these matrices are related as in Condition $(i)$ of Theorem \ref{th isomorphism of TM in terms of matrices}. Fix any $j\in \{1,2,\ldots,k\}$ and define two matrices $R\in \mathbb{M}_{\Lambda^0\times \Lambda_I^0}(\mathbb{N})$, $S\in \mathbb{M}_{\Lambda_I^0\times \Lambda^0}(\mathbb{N})$ as follows: \[R(u,w):=
	\left\{
	\begin{array}{ll}
		1  & \mbox{if } u=par(w), \\
		0  & \mbox{otherwise,} 
	\end{array}
	\right. S(u,w):=
	\left\{
	\begin{array}{ll}
		|\Lambda^{\mathbf{e}_j}w\cap \mathcal{E}_i|  & \mbox{if } u=v^i~\text{for}~i=1,2, \\
		|u\Lambda^{\mathbf{e}_j}w|  & \mbox{otherwise. }
	\end{array}
	\right.\]
It is not hard to show that $RS=A_{\mathbf{e}_j}$, $SR=B_{\mathbf{e}_j}$ and $A_{\mathbf{e}_i}R=RB_{\mathbf{e}_i}$, $B_{\mathbf{e}_i}S=SA_{\mathbf{e}_i}$ for all $i=1,2,\ldots,k$. So we can provide an alternative proof of Theorem \ref{th in-splitting preserves talented monoid} in view of the equivalence $(i)\Longleftrightarrow (iii)$ of Theorem \ref{th isomorphism of TM in terms of matrices}. 

In symbolic dynamics, state splitting, i.e., in-splitting and out-splitting of directed graphs preserve conjugacy of their edge shifts and consequently, the corresponding adjacency matrices of the original graph and the transformed graph are \emph{strongly shift equivalent} and hence, \emph{shift equivalent} (see \cite[Theorems 7.2.7 and 7.3.3]{Lind-Marcus}). These are among the key moves preserving Morita equivalence of Leavitt path algebras. In the study of $k$-graphs, in-splitting is expected to be a crucial move in connection with the geometric classification of higher-rank graph $C^*$-algebras (as well as of Kumjian--Pask algebras in view of Proposition \ref{pro in-splitting gives graded isomorphic KP algebras}). The vertex matrices of a $k$-graph $\Lambda$ and its in-split $\Lambda_I$ naturally satisfy condition $(i)$ of Theorem \ref{th isomorphism of TM in terms of matrices}, which suggests that this matrix condition is not an abstract one; rather, it may be considered as a suitable analogue of the notion of shift equivalence (eventual conjugacy) in the higher-dimensional setting. However, in two dimensions, though the solution to the conjugacy problem for shifts of finite type is fairly clear, it does not use matrices. In-splitting of a finite $2$-graph may be studied using the textile systems introduced in \cite{Nasu}. In \cite{WoA}, it is shown that any $2$-graph in-splitting may be described by a finite sequence of textile moves as described by Johnson and Madden in \cite{JM}. Each of the moves given in \cite{JM} induces a conjugacy of the underlying two-dimensional shift of finite type. 

\end{rmk}

\section{Lifting morphisms of graded $K$-theories of Kumjian--Pask algebras}\label{sec lifting problem}
Let $\mathsf{F}$ be a field. By $\mathbb{K}\mathbb{P}_\mathsf{F}$, we denote the category whose objects are Kumjian--Pask $\mathsf{F}$-algebras of row-finite $k$-graphs without sources and with finite vertex sets, and morphisms are $\mathbb{Z}^k$-graded algebra homomorphisms modulo conjugation by invertible elements of degree $0$. On the other hand, let $\mathbb{P}_{\mathbb{Z}^k}$ be the category of pointed $\mathbb{Z}^k$-pre-ordered abelian groups with pointed order-preserving group homomorphisms (or, equivalently, pointed order-preserving $\mathbb{Z}[\mathbb{Z}^k]$-module homomorphisms).

In this section, we wish to find an answer to the following question:
\begin{ques}\label{ques lifting question}
Consider the functor $K_0^{\gr}:\mathbb{K}\mathbb{P}_\mathsf{F}\longrightarrow \mathbb{P}_{\mathbb{Z}^k}$ defined as follows:
\begin{align*}
K_0^{\gr}:\Obj(\mathbb{K}\mathbb{P}_\mathsf{F}) &\longrightarrow \Obj(\mathbb{P}_{\mathbb{Z}^k})\\
\KP_{\mathsf{F}}(\Lambda) &\longmapsto K_0^{\gr}(\KP_{\mathsf{F}}(\Lambda)),\\
K_0^{\gr}:\Hom_{\mathbb{K}\mathbb{P}_\mathsf{F}}(\KP_\mathsf{F}(\Lambda),\KP_\mathsf{F}(\Omega))&\longrightarrow \Hom_{\mathbb{P}_{\mathbb{Z}^k}}(K_0^{\gr}(\KP_{\mathsf{F}}(\Lambda)),K_0^{\gr}(\KP_{\mathsf{F}}(\Omega)))\\
\psi&\longmapsto K_0^{\gr}(\psi),
\end{align*}
where $K_0^{\gr}(\psi)([P]):=[P\otimes_{\KP_\mathsf{F}(\Lambda)}\KP_{\mathsf{F}}(\Omega)]$ for all graded finitely generated projective right $\KP_\mathsf{F}(\Lambda)$-modules $P$. Given $\mathfrak{h}\in \Hom_{\mathbb{P}_{\mathbb{Z}^k}}(K_0^{\gr}(\KP_{\mathsf{F}}(\Lambda)),K_0^{\gr}(\KP_{\mathsf{F}}(\Omega)))$, when does there exist $\psi\in \Hom_{\mathbb{K}\mathbb{P}_\mathsf{F}}(\KP_\mathsf{F}(\Lambda),\KP_\mathsf{F}(\Omega))$ such that $K_0^{\gr}(\psi)=\mathfrak{h}$?
\end{ques}

\subsection{Polymorphisms and bimodules associated to $k$-graphs}\label{ssec polymorphism}
For any $k$-graph $\Lambda$ and $\N\in \mathbb{N}^k$, we define the \emph{$\N$-shaped directed graph} $E_\N^\Lambda$, where the vertex set is $\Lambda^0$ and edge set is $\Lambda^\N$; the \emph{range} and \emph{source} maps are the restrictions of $r$ and $s$ to $\Lambda^n$ respectively. Considering $E_\N^\Lambda$ as the polymorphism\footnote{In \cite{ART}, a polymorphism $E=(U,V,E^1,r,s)$ is regarded as a set of arrows originated from $U$ and terminated at $V$; whereas we choose a categorical convention in our polymorphisms, i.e., the arrows are directed from $V$ to $U$; consequently the product of our polymorphisms is considered from right to left.} $(\Lambda^0,\Lambda^0,\Lambda^\N,r,s)$, one can note that $E_\N^\Lambda\cong \displaystyle{\prod_{i=1}^{k}}(E_{\mathbf{e}_i}^\Lambda)^{n_i}$, where $E_{\mathbf{e}_i}^\Lambda$ is the $i^{\text{th}}$-coordinate graph associated with $\Lambda$ and the product is the composition of polymorphisms. The order of the components does not matter in the product since for any two distinct $i,j\in \{1,2,\ldots,k\}$, $E_{\mathbf{e}_i}^\Lambda \circ E_{\mathbf{e}_j}^\Lambda\cong E_{\mathbf{e}_j}^\Lambda \circ E_{\mathbf{e}_i}^\Lambda$ via the isomorphism $\Tilde{\kappa}_{\Lambda,i,j}:=\kappa_{j,i}^{-1}\circ \kappa_{i,j}$, where 
\begin{align*}
    \kappa_{i,j}:E_{\mathbf{e}_i}^\Lambda\circ E_{\mathbf{e}_j}^\Lambda &\longrightarrow E_{\mathbf{e}_i+\mathbf{e}_j}^\Lambda\\
    (\lambda,\mu)&\longmapsto \lambda\mu.
\end{align*}
Now, for any $\N\in \mathbb{N}^k$, suppose $\mathsf{F}\Lambda^\N$ is the $\mathsf{F}$-vector space with basis $\Lambda^\N$. Then, $\mathsf{F}\Lambda^\N$ can be considered as an $\mathsf{F}\Lambda^0-\mathsf{F}\Lambda^0$-bimodule with respect to the actions
\[v\lambda:=\delta_{v,r(\lambda)}\lambda,~~~~~~~~~~~~~ \lambda w:=\delta_{w,s(\lambda)}\lambda\] for all $v,w\in \Lambda^0$ and $\lambda\in \Lambda^\N$. In other words, $\mathsf{F}\Lambda^\N$ is the bimodule of the polymorphism $E_\N^\Lambda$. 
 
\begin{lem}
Let $\Lambda$ be a row-finite $k$-graph without sources such that $|\Lambda^0|<\infty$. Then, for  $\N,\M\in \mathbb{N}^k$, we have an $\mathsf{F}\Lambda^0-\mathsf{F}\Lambda^0$-bimodule isomorphism
\[\phi_{\N,\M}^\Lambda:\mathsf{F}\Lambda^{\N+\M}\longrightarrow \mathsf{F}\Lambda^\N \otimes_{\mathsf{F}\Lambda^0} \mathsf{F}\Lambda^\M,\] such that for any $\lambda\in \Lambda^{\N+\M}$, $\phi_{\N,\M}^\Lambda(\lambda)=\lambda_1\otimes \lambda_2$, where $\lambda=\lambda_1 \lambda_2$ is the unique factorization of $\lambda$ into paths of shape $\N,\M$ respectively. 
\end{lem}
\begin{proof}
Applying \cite[Proposition 3.8]{ART}, we have $\mathsf{F}\Lambda^\N\otimes_{\mathsf{F}\Lambda^0}\mathsf{F}\Lambda^\M=\mathsf{F}(E_\N^\Lambda)^1\otimes_{\mathsf{F}\Lambda^0}\mathsf{F}(E_\M^\Lambda)^1\cong \mathsf{F}(E_\N^\Lambda\circ E_\M^\Lambda)^1$ via the isomorphism
\begin{align*}
    \chi:\mathsf{F}(E_\N^\Lambda)^1\otimes_{\mathsf{F}\Lambda^0}\mathsf{F}(E_\M^\Lambda)^1 &\longrightarrow \mathsf{F}(E_\N^\Lambda\circ E_\M^\Lambda)^1\\
    \left(\displaystyle{\sum_{\alpha\in \Lambda^\N}}a_\alpha \alpha \otimes \displaystyle{\sum_{\beta\in \Lambda^\M}}b_\beta \beta\right) &\longmapsto \displaystyle{\sum_{\substack{(\alpha,\beta)\in \Lambda^\N\times \Lambda^\M\\r(\beta)=s(\alpha)}}}a_\alpha b_\beta (\alpha,\beta).
\end{align*}
Again, the isomorphism 
\begin{align*}
\kappa_{\N,\M}:E_\N^\Lambda \circ E_\M^\Lambda &\longrightarrow E_{\N+\M}^\Lambda\\
(\alpha,\beta) &\longmapsto \alpha\beta,
\end{align*}
of polymorphisms induces an isomorphism of $\mathsf{F}\Lambda^0-\mathsf{F}\Lambda^0$-bimodules $\overline{\kappa_{\N,\M}}:\mathsf{F}(E_\N^\Lambda\circ E_\M^\Lambda)^1\longrightarrow \mathsf{F}(E_{\N+\M}^\Lambda)^1$. Finally, set $\phi_{\N,\M}^\Lambda=\chi^{-1}\circ \overline{\kappa_{\N,\M}}^{-1}$. Then, $\phi_{\N,\M}^\Lambda$ is the required isomorphism. 
\end{proof}

\begin{rmk}\label{rem the general splitting}
One can extend the above lemma for any number of shapes rather than just two. For instance, if $\N,\M,\mathbf{p}\in \mathbb{N}^k$, then clearly, the composition $(\phi_{\N,\M}^\Lambda \otimes id_{\mathsf{F}\Lambda^\mathbf{p}})\circ \phi_{\N+\M,\mathbf{p}}^\Lambda$ is an isomorphism from $\mathsf{F}\Lambda^{\N+\M+\mathbf{p}}$ to $(\mathsf{F}\Lambda^\N\otimes_{\mathsf{F}\Lambda^0} \mathsf{F}\Lambda^\M)\otimes_{\mathsf{F}\Lambda^0}\mathsf{F}\Lambda^\mathbf{p}$. The readers should be convinced that the other possible isomorphism of this type (obtained from different association of the shapes) is, in fact, the same as the said one due to the unique factorization property. Taking into account the associativity of tensor products, we regard this as a canonical isomorphism 
\[\phi_{\N,\M,\mathbf{p}}^\Lambda:\mathsf{F}\Lambda^{\N+\M+\mathbf{p}}\longrightarrow \mathsf{F}\Lambda^\N\otimes_{\mathsf{F}\Lambda^0}\mathsf{F}\Lambda^\M\otimes_{\mathsf{F}\Lambda^0}\mathsf{F}\Lambda^\mathbf{p}.\] 
In general, for $\N=(n_1,n_2,\ldots,n_k)\in \mathbb{N}^k$, we can proceed in this way to obtain an isomorphism 

\[\phi^\Lambda:\mathsf{F}\Lambda^\N \longrightarrow (\mathsf{F}\Lambda^{\mathbf{e}_1})^{\otimes n_1}\otimes_{\mathsf{F}\Lambda^0}(\mathsf{F}\Lambda^{\mathbf{e}_2})^{\otimes n_2}\otimes_{\mathsf{F}\Lambda^0}\dots \otimes_{\mathsf{F}\Lambda^0}(\mathsf{F}\Lambda^{\mathbf{e}_k})^{\otimes n_k}.\]
\end{rmk}

\begin{prop}\label{pro the map rho}
Let $\Lambda$ be a row-finite $k$-graph without sources such that $|\Lambda^0|<\infty$. Then, for any $\N\in \mathbb{N}^k$, there exists an $\mathsf{F}\Lambda^0-\KP_{\mathsf{F}}(\Lambda)$-bimodule isomorphism 
\[\rho_{\Lambda,\N}: \mathsf{F}\Lambda^\N\otimes_{\mathsf{F}\Lambda^0}\KP_\mathsf{F}(\Lambda)\longrightarrow \KP_\mathsf{F}(\Lambda),\] such that for any $x=\displaystyle{\sum_{\lambda\in \Lambda^\N}}a_\lambda \lambda\in \mathsf{F}\Lambda^\N$ and $w\in \KP_\mathsf{F}(\Lambda)$, $\rho_{\Lambda,\N}(x\otimes w)=(\displaystyle{\sum_{\lambda\in \Lambda^\N}}a_\lambda s_\lambda)w$. Furthermore, the diagram

\begin{equation}\label{comd2}
\begin{tikzcd}[row sep=huge, column sep=huge]
		\mathsf{F}\Lambda^{\N+\M}\otimes_{\mathsf{F}\Lambda^0}\KP_\mathsf{F}(\Lambda) \arrow [r, "\rho_{\Lambda,\N+\M}"] \arrow [d, "\phi_{\N,\M}^\Lambda\otimes id_{\KP_\mathsf{F}(\Lambda)}"'] & \KP_\mathsf{F}(\Lambda) & \mathsf{F}\Lambda^\N\otimes_{\mathsf{F}\Lambda^0}\KP_\mathsf{F}(\Lambda) \arrow [l, "\rho_{\Lambda,\N}"'] \\
		(\mathsf{F}\Lambda^\N\otimes_{\mathsf{F}\Lambda^0}\mathsf{F}\Lambda^\M)\otimes_{\mathsf{F}\Lambda^0}\KP_\mathsf{F}(\Lambda) \arrow [rr, "\alpha_{_{\mathsf{F}\Lambda^\N,\mathsf{F}\Lambda^\M,\KP_\mathsf{F}(\Lambda)}}"] && \mathsf{F}\Lambda^\N\otimes_{\mathsf{F}\Lambda^0}(\mathsf{F}\Lambda^\M\otimes_{\mathsf{F}\Lambda^0}\KP_\mathsf{F}(\Lambda)) \arrow [u, "id_{\mathsf{F}\Lambda^\N}\otimes \rho_{\Lambda,\M}"']
\end{tikzcd} 
\end{equation}
is commutative for all $\N,\M\in \mathbb{N}^k$. 
\end{prop}
\begin{proof}
Let $X:=\Lambda^0\cup \Lambda^{\neq 0}\cup G(\Lambda^{\neq 0})$ and $\mathbb{F}_\mathsf{F}(\omega(X))$ be the free $\mathsf{F}$-algebra on $X$. Consider the canonical surjection $\pi: \mathbb{F}_\mathsf{F}(\omega(X)) \longrightarrow \KP_\mathsf{F}(\Lambda)$ defined on generators as
\begin{center}
$v\longmapsto p_v,~~~~\lambda\longmapsto s_\lambda,~~~~~\mu^*\longmapsto s_{\mu^*}$
\end{center}
for all $v\in \Lambda^0$, $\lambda,\mu\in \Lambda^{\neq 0}$. Now, with the help of the map $\pi$, we define
\begin{align*}
    \rho_{\Lambda,\N}:\mathsf{F}\Lambda^\N\otimes_{\mathsf{F}\Lambda^0}\KP_\mathsf{F}(\Lambda) &\longrightarrow \KP_\mathsf{F}(\Lambda)\\
    \displaystyle{\sum_{i=1}^{N}} x_i\otimes w_i &\longmapsto \displaystyle{\sum_{i=1}^{N}} \pi(x_i)w_i.
\end{align*}
A routine verification yields that $\rho_{\Lambda,\N}$ is a well-defined $\mathsf{F}\Lambda^0-\KP_\mathsf{F}(\Lambda)$-bimodule homomorphism. For surjectivity, it suffices to show that the elements of the form $s_\lambda s_{\mu^*}$ of $\KP_\mathsf{F}(\Lambda)$ are in the image of $\rho_{\Lambda,\N}$. So take $\lambda,\mu\in \Lambda$ such that $s(\lambda)=s(\mu)$. By using $(KP4)$, we can write $p_{r(\lambda)}=\displaystyle{\sum_{\alpha\in r(\lambda)\Lambda^n}}s_\alpha s_{\alpha^*}$. Thus, we have
\[s_\lambda s_{\mu^*}=p_{r(\lambda)}s_\lambda s_{\mu^*}=\displaystyle{\sum_{\alpha\in r(\lambda)\Lambda^n}}s_\alpha (s_{\alpha^*}s_\lambda s_{\mu^*})=\displaystyle{\sum_{\alpha\in r(\lambda)\Lambda^n}}\pi(\alpha) (s_{\alpha^*}s_\lambda s_{\mu^*})=\rho_{\Lambda,\N}\left(\displaystyle{\sum_{\alpha\in r(\lambda)\Lambda^\N}} \alpha\otimes s_{\alpha^*}s_\lambda s_{\mu^*}\right).\]
Before proving that $\rho_{\Lambda,\N}$ is injective, let us check the commutativity of Diagram (\ref{comd2}). Since all the maps involved are bimodule homomorphisms, it suffices to show that $\rho_{\Lambda,\N}\circ (id_{\mathsf{F}\Lambda^\N}\otimes \rho_{\Lambda,\M})\circ \alpha_{\mathsf{F}\Lambda^\N,\mathsf{F}\Lambda^\M,\KP_\mathsf{F}(\Lambda)}\circ (\phi_{\N,\M}^\Lambda\otimes id_{\KP_\mathsf{F}(\Lambda)})$ and $\rho_{\Lambda,\N+\M}$ agree on each simple tensor $\lambda\otimes w$, where $\lambda\in \mathsf{F}\Lambda^{\N+\M}$ and $w\in \KP_\mathsf{F}(\Lambda)$. But this is evident since 
\begin{align*}
&\left(\rho_{\Lambda,\N}\circ (id_{\mathsf{F}\Lambda^\N}\otimes \rho_{\Lambda,\M})\circ \alpha_{\mathsf{F}\Lambda^\N,\mathsf{F}\Lambda^\M,\KP_\mathsf{F}(\Lambda)}\circ (\phi_{\N,\M}^\Lambda\otimes id_{\KP_\mathsf{F}(\Lambda)})\right)(\lambda\otimes w)\\
= & \left(\rho_{\Lambda,\N}\circ (id_{\mathsf{F}\Lambda^\N}\otimes \rho_{\Lambda,\M})\circ \alpha_{\mathsf{F}\Lambda^\N,\mathsf{F}\Lambda^\M,\KP_\mathsf{F}(\Lambda)}\right)((\lambda_1\otimes \lambda_2)\otimes w)\\ 
= & \left(\rho_{\Lambda,\N}\circ (id_{\mathsf{F}\Lambda^\N}\otimes \rho_{\Lambda,\M})\right)(\lambda_1\otimes (\lambda_2\otimes w))\\
= & \rho_{\Lambda,\N}(\lambda_1\otimes s_{\lambda_2}w)\\
= & s_{\lambda_1}(s_{\lambda_2}w)\\
= & s_\lambda w\\
= & \rho_{\Lambda,\N+\M}(\lambda\otimes w).
\end{align*}

Now, we prove that $\rho_{\Lambda,\N}$ is injective by induction on $\ell(\N):=n_1+n_2+\ldots+n_k$. The case $\ell(\N)=0$ is trivial since it is possible only if $\N=0\in \mathbb{N}^k$ and in this case $\rho_{\Lambda,0}$ is the canonical isomorphism $\mathsf{F}\Lambda^0\otimes_{\mathsf{F}\Lambda^0}\KP_\mathsf{F}(\Lambda)\longrightarrow \KP_\mathsf{F}(\Lambda)$ given by multiplication, and hence injective. So let us consider the case $\ell(\N)=1$. Then, $\N=\mathbf{e}_i$ for some $i=1,2,\ldots,k$. Suppose $x=\displaystyle{\sum_{j=1}^{t}}x_j\otimes w_j\in \mathsf{F}\Lambda^{\mathbf{e}_i}\otimes_{\mathsf{F}\Lambda^0}\KP_\mathsf{F}(\Lambda)$ is such that $\rho_{\Lambda,\mathbf{e}_i}(x)=0$. For each $j=1,2,\ldots,t$, write $x_j=\displaystyle{\sum_{\lambda\in \Lambda^{\mathbf{e}_i}}} c_{j,\lambda}\lambda$. Then,  
\begin{align*}
0=\displaystyle{\sum_{\alpha\in \Lambda^{\mathbf{e}_i}}} \alpha\otimes s_{\alpha^*}\left(\displaystyle{\sum_{j=1}^{t}}\pi(x_j)w_j\right)
&= \displaystyle{\sum_{\alpha\in \Lambda^{\mathbf{e}_i}}} \alpha \otimes \left(\displaystyle{\sum_{j=1}^{t}} (s_{\alpha^*}\pi(x_j))w_j\right)\\
&= \displaystyle{\sum_{\alpha\in \Lambda^{\mathbf{e}_i}}} \displaystyle{\sum_{j=1}^{t}}\alpha\pi^{-1}(s_{\alpha^*} \pi(x_j))\otimes w_j\\
&= \displaystyle{\sum_{\alpha\in \Lambda^{\mathbf{e}_i}}} \displaystyle{\sum_{j=1}^{t}}\alpha\pi^{-1}\left(s_{\alpha^*} \displaystyle{\sum_{\lambda\in \Lambda^{\mathbf{e}_i}}}c_{j,\lambda}s_\lambda\right)\otimes w_j\\
&= \displaystyle{\sum_{\alpha\in \Lambda^{\mathbf{e}_i}}} \displaystyle{\sum_{j=1}^{t}}\alpha\pi^{-1}\left(\displaystyle{\sum_{\lambda\in \Lambda^{\mathbf{e}_i}}}c_{j,\lambda}s_{\alpha^*}s_\lambda\right)\otimes w_j\\
&= \displaystyle{\sum_{\alpha\in \Lambda^{\mathbf{e}_i}}} \displaystyle{\sum_{j=1}^{t}}\alpha\pi^{-1}(c_{j,\alpha}p_{s(\alpha)})\otimes w_j
= \displaystyle{\sum_{j=1}^{t}}\left(\displaystyle{\sum_{\alpha\in \Lambda^{\mathbf{e}_i}}} c_{j,\alpha}\alpha\right)\otimes w_j=\displaystyle{\sum_{j=1}^{t}}x_j\otimes w_j=x.
\end{align*}

Therefore, $\rho_{\Lambda,\mathbf{e}_i}$ is injective. Now, take $\ell\ge 1$ and assume that $\rho_{\Lambda,\N}$ is injective for all $\N\in \mathbb{N}^k$ with $0\le \ell(\N)\le \ell$. Let $n\in \mathbb{N}^k$ be such that $\ell(\N)=\ell+1$. Writing $\N=\M+\mathbf{p}$ for some $0<\M,\mathbf{p}\in \mathbb{N}^k$, we have \[\rho_{\Lambda,\M}\circ (id_{\mathsf{F}\Lambda^\M}\otimes \rho_{\Lambda,\mathbf{p}})\circ \alpha_{\mathsf{F}\Lambda^\M,\mathsf{F}\Lambda^\mathbf{p},\KP_\mathsf{F}(\Lambda)}\circ (\phi_{\M,\mathbf{p}}^\Lambda\otimes id_{\KP_\mathsf{F}(\Lambda)})=\rho_{\Lambda,\N}\] because (\ref{comd2}) commutes. Now by the induction hypothesis, both $\rho_{\Lambda,\M}$ and $\rho_{\Lambda,\mathbf{p}}$ are injective. Since $\mathsf{F}\Lambda^0$ is semisimple, both the right $\mathsf{F}\Lambda^0$-module $\mathsf{F}\Lambda^{\mathbf{m}}$ and the left $\mathsf{F}\Lambda^0$-module $\KP_\mathsf{F}(\Lambda)$ are flat, and hence the homomorphisms inside parentheses on the left-hand side of the above equality are injective. Being a composition of injective module homomorphisms, it now follows that $\rho_{\Lambda,\N}$ is injective. This completes the proof.
\end{proof}

\subsection{Bridging matrix and the formation of bridging bimodule}\label{ssec bridging matrix}
If $E=(U,V,E^1,r,s)$ is a polymorphism, then the \emph{adjacency matrix} of $E$ is defined to be a matrix $A_E\in \mathbb{M}_{U\times V}(\mathbb{N})$ such that \[A_E(u,v):=|\{e\in E^1~|~s(e)=v~\text{and}~r(e)=u\}|.\] Note that our adjacency matrix $A_E$ is the transpose of the adjacency matrix defined in \cite[Definition 3.2]{ART}. This is simply because the edges in our polymorphism are directed from $V$ to $U$. 

Let $\Lambda,\Omega$ be row-finite $k$-graphs without sources such that $|\Lambda^0|,|\Omega^0|< \infty$. Given a matrix $R\in \mathbb{M}_{\Lambda^0\times \Omega^0}(\mathbb{N})$, we can associate with it a polymorphism $E_R\equiv (\Lambda^0,\Omega^0,R^1,r,s)$, where \[R^1:=\{g_i^{v,w}~|~v\in \Lambda^0,w\in \Omega^0, 1\le i\le R(v,w)\},\] and $r:R^1\longrightarrow \Lambda^0$, $s:R^1\longrightarrow \Omega^0$ are respectively defined as \[r(g_i^{v,w}):=v,~~~~~~~s(g_i^{v,w}):=w.\]
It is then easy to note that the adjacency matrix of the polymorphism $E_R$ is $R$. Also, $E_R$ gives an $\mathsf{F}\Lambda^0-\mathsf{F}\Omega^0$-bimodule $\mathsf{F}R^1$ with respect to the bi-action \[vgw:=\delta_{v,r(g)}\delta_{w,s(g)}g\] for all $v\in \Lambda^0,w\in \Omega^0$ and $g\in R^1$.
\begin{dfn}\label{def the bridging matrix and polymorphism}
A matrix $R\in \mathbb{M}_{\Lambda^0\times \Omega^0}(\mathbb{N})$ is called a \emph{bridging matrix} (respectively, $E_R$, a \emph{bridging polymorphism}) for the $k$-graphs $\Lambda$ and $\Omega$ if there is a family of isomorphisms (called \emph{flips}) $\mathfrak{F}=(\mathfrak{f}_i)_{1\le i\le k}$; \[\mathfrak{f}_i:E_{\mathbf{e}_i}^\Lambda\circ E_R\longrightarrow E_R\circ E_{\mathbf{e}_i}^\Omega,\] such that the diagram 

\begin{equation}\label{comd3}
\begin{tikzcd}[row sep=huge, column sep=huge]
		E_{\mathbf{e}_i}^\Lambda\circ E_{\mathbf{e}_j}^\Lambda \circ E_R \arrow [r, "id_{E_{\mathbf{e}_i}^\Lambda}\times \mathfrak{f}_j"] \arrow [d, "\Tilde{\kappa}_{\Lambda,i,j}\times id_{E_R}"'] & E_{\mathbf{e}_i}^\Lambda \circ E_R \circ E_{\mathbf{e}_j}^\Omega \arrow [r, "\mathfrak{f}_i\times id_{E_{\mathbf{e}_j}^\Omega}"] & E_R\circ E_{\mathbf{e}_i}^\Omega \circ E_{\mathbf{e}_j}^\Omega \arrow [d, "id_{E_R}\times \Tilde{\kappa}_{\Omega,i,j}"] \\	
         E_{\mathbf{e}_j}^\Lambda\circ E_{\mathbf{e}_i}^\Lambda \circ E_R \arrow [r, "id_{E_{\mathbf{e}_j}^\Lambda}\times \mathfrak{f}_i"] & E_{\mathbf{e}_j}^\Lambda \circ E_R\circ E_{\mathbf{e}_i}^\Omega \arrow [r, "\mathfrak{f}_j\times id_{E_{\mathbf{e}_i}^\Omega}"] & E_R\circ E_{\mathbf{e}_j}^\Omega \circ E_{\mathbf{e}_i}^\Omega
\end{tikzcd} 
\end{equation}
commutes for all $1\le i < j \le k$. We call $(R,\mathfrak{F})$, a \emph{specified $\Lambda-\Omega$ bridging pair}. 
\end{dfn}
\begin{rmks}\label{rem explaining the bridging graph} 
$(i)$ Since composition of polymorphisms is associative, we ignore the unnecessary parentheses in (\ref{comd3}). In this way, we avoid the extra layers of associativity isomorphisms. The following figure explains the \emph{bridging} criterion in the case $k=2$. 
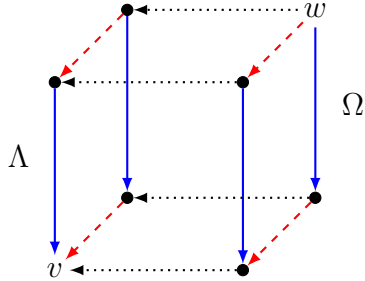
\[
\begin{tikzpicture}[scale=2.5]
\node[inner sep=1.5pt, circle,draw,fill=black] (A1) at (0,0,0) {};
\node[inner sep=1.5pt, circle,draw,fill=black] (A2) at (1,0,0) {};
\node[inner sep=1.5pt, circle,draw,fill=black] (A3) at (1,0,1) {};
\node[inner sep=1pt, circle] (A4) at (0,0,1) {$v$};

\node[inner sep=1.5pt, circle,draw,fill=black] (B1) at (0,1,0) {};
\node[inner sep=1pt, circle] (B2) at (1,1,0) {$w$};
\node[inner sep=1.5pt, circle,draw,fill=black] (B3) at (1,1,1) {};
\node[inner sep=1.5pt, circle,draw,fill=black] (B4) at (0,1,1) {};

\path[->, red, dashed, >=latex,thick] (B1) edge [] node[]{} (B4);
\path[->, red, dashed, >=latex,thick] (A1) edge [] node[]{} (A4);
\path[->, red, dashed, >=latex,thick] (B2) edge [] node[]{} (B3);
\path[->, red, dashed, >=latex,thick] (A2) edge [] node[]{} (A3);

\path[->,blue, >=latex,thick] (B1) edge [] node[]{} (A1);
\path[->,blue, >=latex,thick] (B4) edge [] node[]{} (A4);
\path[->,blue, >=latex,thick] (B2) edge [] node[]{} (A2);
\path[->,blue, >=latex,thick] (B3) edge [] node[]{} (A3);

\path[->,dotted, >=latex,thick] (B2) edge [] node {} (B1);
\path[->,dotted, >=latex,thick] (B3) edge [] node {} (B4);
\path[->,dotted, >=latex,thick] (A2) edge [] node {} (A1);
\path[->,dotted, >=latex,thick] (A3) edge [] node {} (A4);

\node at (0,0.8,1.5) {$\Lambda$};
\node at (1.2,0.5,0) {$\Omega$};
\end{tikzpicture}
\]
The two blue-red commutative squares are parts of the $1$-skeletons of the $k$-graphs; the left one is on the level of $\Lambda$ and the right one is on the level of $\Omega$. The edges of the polymorphism $E_R$, shown as dotted black lines, connect vertices of $\Omega$ to vertices of $\Lambda$ in such a way that the remaining squares are commutative (due to the flips) and all the tri-colored paths from $w\in \Omega^0$ to $v\in \Lambda^0$ are the same. 

$(ii)$ We remark that the commutativity of the above diagram has close similarity with the necessary and sufficient requirement (see \cite[Remark 2.3]{Fowler} and \cite[Theorem 4.4]{Hazlewood}) for a given set of $k$ directed graphs on a common vertex set to be the coordinate graphs of a $k$-graph. 

To be more precise, suppose we have two $k$-graphs $\Lambda,\Omega$ and a matrix $R\in \mathbb{M}_{\Lambda^0\times \Omega^0}(\mathbb{N})$ along with a family of isomorphisms $(\mathfrak{f}_i:E_{\mathbf{e}_i}^\Lambda\circ E_R\longrightarrow E_R\circ E_{\mathbf{e}_i}^\Omega)_{1\le i \le k}$. Form a directed graph $G$, where $G^0:=\Lambda^0\sqcup \Omega^0$ and $G^1:=\Lambda^\mathbf{1}\sqcup \Omega^\mathbf{1} \sqcup R^1$. The range and source maps are defined by extending naturally the individual range and source maps of $\Lambda,\Omega$ and $E_R$. Now, we color the edges of $G$ with $k+1$ colors using the map $d:G^1\longrightarrow \mathbb{N}^{k+1}$ defined by 
$$d(\alpha) :=
	\left\{
	\begin{array}{lll}
		\iota(d_\Lambda(\alpha))  & \mbox{if } \alpha\in \Lambda^\mathbf{1}, \\
            \iota(d_\Omega(\alpha))  & \mbox{if } \alpha\in \Omega^\mathbf{1},\\
		\mathbf{e}_{k+1}  &  \mbox{if } \alpha\in R^1,
	\end{array}
	\right.$$ 
where $\iota:\mathbb{N}^k\longrightarrow \mathbb{N}^{k+1}$ is the inclusion sending $(n_1,n_2,\ldots,n_k)$ to $(n_1,n_2,\ldots,n_k,0)$. One can naturally extend $d$ to color all the paths in $G$ (e.g., if $k=2$ and $\alpha g\beta$ is a path with $\alpha\in \Lambda^\mathbf{1}$, $d_\Lambda(\alpha)=(1,0)$, $g\in R_1$ and $\beta\in \Omega^\mathbf{1}$, $d_\Omega(\beta)=(0,1)$ then $d(\alpha g\beta)=(1,0,0)+(0,0,1)+(0,1,0)=(1,1,1)$). Thus, we have a $(k+1)$-colored directed graph $G$ with $G_i=E_{\mathbf{e}_i}^\Lambda\sqcup E_{\mathbf{e}_i}^\Omega$ for $i=1,2,\ldots,k$ and $G_{k+1}=E_R$. If $i,j\neq k+1$, we define isomorphisms $G_i\circ G_j\longrightarrow G_j\circ G_i$ by using the factorization rules of $\Lambda$ or $\Omega$, whichever is suitable. For each $i\neq k+1$, the isomorphism $G_i\circ G_{k+1}\longrightarrow G_{k+1}\circ G_i$ is given by the flip $\mathfrak{f}_i$. Now, the readers can easily verify (in view of \cite[Remark 2.3]{Fowler}) that the set of paths of $G$, together with the map $d$ and the said factorization rules, forms a well-defined $(k+1)$-graph if and only if the diagram $(\ref{comd3})$ commutes. 

$(iii)$ Given a specified $\Lambda-\Omega$ bridging pair $(R,\mathfrak{F})$, we can naturally form a $\Lambda-\Omega$ \emph{morph} \cite[Definition 3.1]{KPS}. Let $g\in R^1$ and $\omega\in \Omega$. If $\omega\notin \Omega^0$, then we can write $\omega=\omega_1\omega_2\cdots\omega_t$, where $\omega_i\in \Omega^\mathbf{1}$ for all $i=1,2,\ldots,t$ and $d(\omega_1)\le d(\omega_2)\le \cdots \le d(\omega_t)$. Let $d(\omega_i)=e_{d_i}$, $\mathfrak{f}_{d_1}^{-1}((g,\omega_1))=(\lambda_1,g_1)$ and $\mathfrak{f}_{d_{i+1}}^{-1}((g_i,\omega_{i+1})=(\lambda_{i+1},g_{i+1})$ for all $i=1,2,\ldots,t-1$. Let $\lambda=\lambda_1\lambda_2\cdots \lambda_t$. Define a map $\phi:R^1*_{\Omega^0}\Omega\longrightarrow \Lambda*_{\Lambda^0}R^1$ by 
$$\phi((g,\omega)) :=
	\left\{
	\begin{array}{ll}
		 (\lambda,g_t) & \mbox{if } \omega\notin \Omega^0\\
            (r(g),g)  & \mbox{if } \omega\in \Omega^0.
	\end{array}
	\right.$$ 
The commutativity of diagram $(\ref{comd3})$ then guarantees that $\phi$ is a well-defined bijection and subsequently $(R^1,r,s,\phi)$ is a $\Lambda-\Omega$ morph. 
\end{rmks}

A necessary condition for $R$ to be a bridging matrix for $\Lambda$ and $\Omega$ is that $A_{\mathbf{e}_i}R=RB_{\mathbf{e}_i}$ for all $i=1,2,\ldots,k$. This is also sufficient for $k=1$ (diagram (\ref{comd3}) vacuously commutes in this case) since to form a $2$-graph (see Remark \ref{rem explaining the bridging graph} $(ii)$) from a given $2$-colored directed graph, it is enough to specify the factorization rules for bi-colored paths of degree $(1,1)$ and that can be done by any isomorphism $\mathfrak{f}:E_1^\Lambda\circ E_R\longrightarrow E_R\circ E_1^\Omega$ (see \cite[\S 6]{Kumjian--Pask}). However, for $k\ge 2$, only commutativity on the level of matrices is not sufficient. We give some examples.
\begin{example}\label{ex matrix commutativity is not enough}
Suppose $\Lambda$ is a $2$-graph with the following $1$-skeleton:
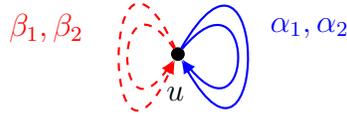
\[
\begin{tikzpicture}[scale=0.35]

\node[circle,draw,fill=black,inner sep=0.5pt] (p11) at (1, 0) {$.$} 

edge[-latex, blue,thick,loop, out=50, in=-50, min distance=123, looseness=2.2] (p11)
edge[-latex, blue,thick, loop, out=60, in=-60, min distance=190, looseness=2.3] (p11)
edge[-latex, red,thick, loop, dashed, out=125, in=235, min distance=115, looseness=1.5] (p11)
edge[-latex, red, thick, loop, dashed, out=115, in=245, min distance=190, looseness=2.2] (p11);
                                                               
\node at (0.9, -1.5) {$u$};
\node at (-4, 1) {\textcolor{red}{$\beta_1,\beta_2$}};
\node at (6,1) {\textcolor{blue}{$\alpha_1,\alpha_2$}};

\end{tikzpicture}
\]
and $\Omega$ is a $2$-graph with the following $1$-skeleton:
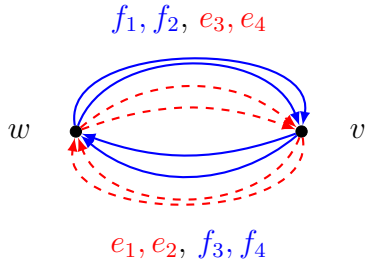
\[
\begin{tikzpicture}[scale=1.5]
\node[inner sep=1.5pt, circle,draw,fill=black] (A) at (0,0) {};	
\node[inner sep=1.5pt, circle,draw,fill=black] (B) at (2,0) {};

\path[->, red, dashed, >=latex,thick] (A) edge [bend left=20] node[above=0.05cm]{} (B);
\path[->, red, dashed, >=latex,thick] (A) edge [bend left=40] node[above=0.05cm]{} (B);
\path[->, red, dashed, >=latex,thick] (B) edge [bend left=70] node[above=0.05cm]{} (A);
\path[->, red, dashed, >=latex,thick] (B) edge [bend left=100] node[above=0.05cm]{} (A);
\path[->,blue, >=latex,thick] (B) edge [bend left=20] node[below=0.05cm]{} (A);
\path[->,blue, >=latex,thick] (B) edge [bend left=40] node[below=0.05cm]{} (A);
\path[->,blue, >=latex,thick] (A) edge [bend left=70] node[below=0.05cm]{} (B);
\path[->,blue, >=latex,thick] (A) edge [bend left=100] node[below=0.05cm]{} (B);

\node at (-0.5,0) {$w$};
\node at (2.5,0) {$v$};
\node at (1,1) {\textcolor{blue}{$f_1,f_2$},\textcolor{red}{$~~e_3,e_4$}};
\node at (1,-1) {\textcolor{red}{$e_1,e_2$},\textcolor{blue}{$~~f_3,f_4$}};
\end{tikzpicture}
\]

The factorization rules for $\Lambda$ are given by $\alpha_i\beta_j=\beta_j\alpha_i$ for all $1\le i,j\le 2$. We fix the factorization rules for $\Omega$ as follows:
\begin{align*}
v\text{-}v~&\text{bi-colored paths}\hspace{2cm} w\text{-}w~\text{bi-colored paths}\\
&f_1e_1=e_3f_3\hspace{4.2cm} f_3e_3=e_2f_2\\
&f_1e_2=e_4f_3\hspace{4.2cm} f_3e_4=e_2f_1\\
&f_2e_1=e_3f_4\hspace{4.2cm} f_4e_3=e_1f_2\\
&f_2e_2=e_4f_4\hspace{4.2cm} f_4e_4=e_1f_1
\end{align*}

The matrices for $\Lambda$ are $A_{(1,0)}=A_{(0,1)}=(2)$ and the matrices for $\Omega$ are \[B_{(1,0)}=B_{(0,1)}=
\begin{pmatrix}
0 & 2 \\
2 & 0
\end{pmatrix}.\]
Take $R=\begin{pmatrix}
1 & 1
\end{pmatrix}$. Then, $A_{(1,0)}R=RB_{(1,0)}$ and $A_{(0,1)}R=RB_{(0,1)}$. The following figure shows the edges of the polymorphism $E_R$ connecting the vertices of $\Lambda$ and $\Omega$. 

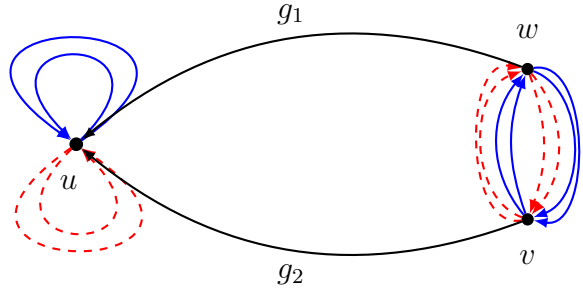
\[
\begin{tikzpicture}[scale=1]

\node[circle,draw,fill=black,inner sep=0.5pt] (p11) at (0, 0) {$.$} 
edge[-latex, blue,thick,loop, out=45, in=135, min distance=60] (p11)
edge[-latex, blue,thick, loop, out=35, in=145, min distance=90] (p11)

edge[-latex, red,thick, loop, dashed, out=225, in=-45, min distance=60] (p11)
edge[-latex, red, thick, loop, dashed, out=215, in=-35, min distance=90] (p11);
                                                               
\node at (-0.1, -0.5) {$u$};

\node[inner sep=1.5pt, circle,draw,fill=black] (A) at (6,1) {};	
\node[inner sep=1.5pt, circle,draw,fill=black] (B) at (6,-1) {};

\path[->, red, dashed, >=latex,thick] (A) edge [bend left=20] node[above=0.05cm]{} (B);
\path[->, red, dashed, >=latex,thick] (A) edge [bend left=40] node[above=0.05cm]{} (B);
\path[->, red, dashed, >=latex,thick] (B) edge [bend left=70] node[above=0.05cm]{} (A);
\path[->, red, dashed, >=latex,thick] (B) edge [bend left=100] node[above=0.05cm]{} (A);
\path[->,blue, >=latex,thick] (B) edge [bend left=20] node[below=0.05cm]{} (A);
\path[->,blue, >=latex,thick] (B) edge [bend left=40] node[below=0.05cm]{} (A);
\path[->,blue, >=latex,thick] (A) edge [bend left=70] node[below=0.05cm]{} (B);
\path[->,blue, >=latex,thick] (A) edge [bend left=100] node[below=0.05cm]{} (B);

\node at (6,1.5) {$w$};
\node at (6,-1.5) {$v$};

\path[->, >=latex,thick] (A) edge [bend right=30] node[above=0.05cm]{$g_1$} (p11);
\path[->, >=latex,thick] (B) edge [bend left=30] node[below=0.05cm]{$g_2$} (p11);
\end{tikzpicture}
\]
For each $i=1,2$, there are four possible choices for an isomorphism $\mathfrak{f}_i:E_{\mathbf{e}_i}^\Lambda\circ E_R\longrightarrow E_R\circ E_{\mathbf{e}_i}^\Omega$. Thus, in total, we have $16$ different families of flips. A tedious verification confirms that no such family makes the diagram $(\ref{comd3})$ commute, for instance, if we choose 
\begin{align*}
    \mathfrak{f}_1:E_{(1,0)}^\Lambda\circ E_R &\longrightarrow E_R\circ E_{(1,0)}^\Omega\\
    (\alpha_1,g_1)&\longmapsto (g_2,f_2)\\
    (\alpha_1,g_2)&\longmapsto (g_1,f_3)\\
    (\alpha_2,g_1)&\longmapsto (g_2,f_1)\\
    (\alpha_2,g_2)&\longmapsto (g_1,f_4)
\end{align*}
and
\begin{align*}
    \mathfrak{f}_2:E_{(0,1)}^\Lambda\circ E_R &\longrightarrow E_R\circ E_{(0,1)}^\Omega\\
    (\beta_1,g_1)&\longmapsto (g_2,e_3)\\
    (\beta_1,g_2)&\longmapsto (g_1,e_2)\\
    (\beta_2,g_1)&\longmapsto (g_2,e_4)\\
    (\beta_2,g_2)&\longmapsto (g_1,e_1),
\end{align*}
then $(\alpha_1,\beta_2,g_1)\longmapsto (\alpha_1,g_2,e_4)\longmapsto (g_1,f_3,e_4)\longmapsto (g_1,e_2,f_1)$ by taking the top-right path in $(\ref{comd3})$; whereas $(\alpha_1,\beta_2,g_1)\longmapsto (\beta_2,\alpha_1,g_1)\longmapsto (\beta_2,g_2,f_2)\longmapsto (g_1,e_1,f_2)$ by taking the left-bottom path. 
\end{example}
\begin{example}\label{ex positive example of a bridging matrix}
The following are the $1$-skeletons of two $2$-graphs $\Lambda$ and $\Omega$.
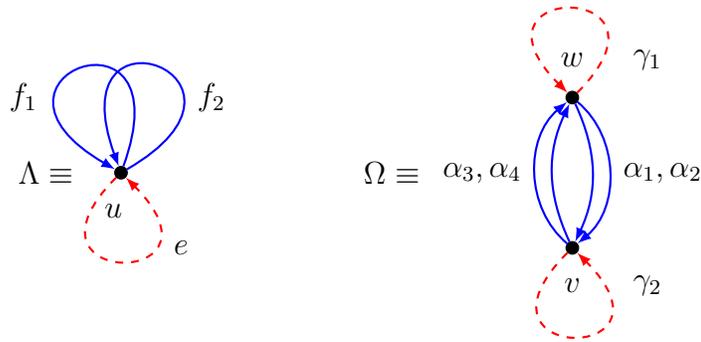
\[
\begin{tikzpicture}[scale=1]

\node[circle,draw,fill=black,inner sep=0.5pt] (p11) at (0, 0) {$.$} 

edge[-latex, blue,thick,loop, out=30, in=110, min distance=70] (p11)
edge[-latex, blue,thick, loop, out=70, in=150, min distance=70] (p11)

edge[-latex, red,thick, loop, dashed, out=225, in=-45, min distance=60] (p11);

\node at (-0.1, -0.5) {$u$};
\node at (-1,0) {$\Lambda\equiv$};

\node[circle,draw,fill=black,inner sep=0.5pt] (A) at (6, 1) {$.$}
edge[-latex, red,thick,loop, dashed, out=45, in=135, min distance=60] (A);

\node[circle,draw,fill=black,inner sep=0.5pt] (B) at (6, -1) {$.$}
edge[-latex, red,thick, loop, dashed, out=225, in=-45, min distance=60] (B);

\path[->,blue, >=latex,thick] (B) edge [bend left=25] node[below=0.05cm]{} (A);
\path[->,blue, >=latex,thick] (B) edge [bend left=50] node[below=0.05cm]{} (A);
\path[->,blue, >=latex,thick] (A) edge [bend left=25] node[below=0.05cm]{} (B);
\path[->,blue, >=latex,thick] (A) edge [bend left=50] node[below=0.05cm]{} (B);

\node at (6,1.5) {$w$};
\node at (6,-1.5) {$v$};
\node at (3.6,0) {$\Omega\equiv$};
\node at (-1.3,1) {$f_1$};
\node at (1.2,1) {$f_2$};
\node at (0.8,-1) {$e$};
\node at (4.8,0) {$\alpha_3,\alpha_4$};
\node at (7.2,0) {$\alpha_1,\alpha_2$};
\node at (7,1.5) {$\gamma_1$};
\node at (7,-1.5) {$\gamma_2$};

\end{tikzpicture}
\]
In $\Lambda$, suppose $ef_i=f_i e$ for $i=1,2$. For $\Omega$, suppose the factorization rules are $\alpha_i\gamma_1=\gamma_2\alpha_i$ for $i=1,2$ and $\alpha_j\gamma_2=\gamma_1\alpha_j$ for $j=3,4$. It is easy to see that the matrix $R=\begin{pmatrix}
1 & 1
\end{pmatrix}$ satisfies $A_{\mathbf{e}_i}R=RB_{\mathbf{e}_i}$ for $i=1,2$. Consider the following isomorphisms:
\begin{align*}
    \mathfrak{f}_1:E_{(1,0)}^\Lambda\circ E_R &\longrightarrow E_R\circ E_{(1,0)}^\Omega\\
    (f_1,g^{u,w})&\longmapsto (g^{u,v},\alpha_1)\\
    (f_1,g^{u,v})&\longmapsto (g^{u,w},\alpha_3)\\
    (f_2,g^{u,w})&\longmapsto (g^{u,v},\alpha_2)\\
    (f_2,g^{u,v})&\longmapsto (g^{u,w},\alpha_4)
\end{align*}
and
\begin{align*}
    \mathfrak{f}_2:E_{(0,1)}^\Lambda\circ E_R &\longrightarrow E_R\circ E_{(0,1)}^\Omega\\
    (e,g^{u,w})&\longmapsto (g^{u,w},\gamma_1)\\
    (e,g^{u,v})&\longmapsto (g^{u,v},\gamma_2).\\
\end{align*}
One can check that the flips $\mathfrak{f}_1,\mathfrak{f}_2$ make the diagram (\ref{comd3}) commute. Therefore, $R$ is a bridging matrix for $\Lambda$ and $\Omega$. 
\end{example}

The existence of a bridging matrix has the following consequence in terms of bimodules of polymorphisms: 
\begin{lem}\label{lem towards specified conjugacy}
Let $\Lambda,\Omega$ be row-finite $k$-graphs without sources such that $|\Lambda^0|,|\Omega^0|< \infty$. Suppose $(R,\mathfrak{F}=(\mathfrak{f}_i)_{1\le i\le k})$ is a specified $\Lambda-\Omega$ bridging pair. Then, we have $\mathsf{F}\Lambda^0-\mathsf{F}\Omega^0$-bimodule isomorphisms
\[\sigma_\N:\mathsf{F}\Lambda^\N\otimes_{\mathsf{F}\Lambda^0}\mathsf{F}R^1\longrightarrow \mathsf{F}R^1\otimes_{\mathsf{F}\Omega^0}\mathsf{F}\Omega^\N,\]
for all $\N\in \mathbb{N}^k$ such that the diagram 

\begin{equation}\label{comd4}
\begin{tikzcd}[row sep=huge, column sep=huge]
        \mathsf{F}\Lambda^{\N+\M}\otimes_{\mathsf{F}\Lambda^0}\mathsf{F}R^1 \arrow [r, "\sigma_{\N+\M}"] \arrow [d, "\phi_{\N,\M}^\Lambda\otimes id_{\mathsf{F}R^1}"] & \mathsf{F}R^1\otimes_{\mathsf{F}\Omega^0}\mathsf{F}\Omega^{\N+\M} \arrow [d, "id_{\mathsf{F}R^1}\otimes \phi_{\N,\M}^\Omega"] \\	(\mathsf{F}\Lambda^\N\otimes_{\mathsf{F}\Lambda^0}\mathsf{F}\Lambda^\M)\otimes_{\mathsf{F}\Lambda^0}\mathsf{F}R^1 \arrow [d, "\alpha_{_{\mathsf{F}\Lambda^\N,\mathsf{F}\Lambda^\M,\mathsf{F}R^1}}"] & \mathsf{F}R^1\otimes_{\mathsf{F}\Omega^0}(\mathsf{F}\Omega^\N\otimes_{\mathsf{F}\Omega^0}\mathsf{F}\Omega^\M) \arrow [d, "\alpha_{_{\mathsf{F}R^1,\mathsf{F}\Omega^\N,\mathsf{F}\Omega^\M}}^{-1}"]\\
        \mathsf{F}\Lambda^\N\otimes_{\mathsf{F}\Lambda^0}(\mathsf{F}\Lambda^\M\otimes_{\mathsf{F}\Lambda^0}\mathsf{F}R^1) \arrow [d, "id_{\mathsf{F}\Lambda^\N}\otimes \sigma_\M"] & (\mathsf{F}R^1\otimes_{\mathsf{F}\Omega^0}\mathsf{F}\Omega^\N)\otimes_{\mathsf{F}\Omega^0}\mathsf{F}\Omega^\M \arrow [d, "\sigma_\N^{-1}\otimes id_{\mathsf{F}\Omega^\M}"]\\
        \mathsf{F}\Lambda^\N\otimes_{\mathsf{F}\Lambda^0}(\mathsf{F}R^1\otimes_{\mathsf{F}\Omega^0}\mathsf{F}\Omega^\M) \arrow [r,"\alpha_{_{\mathsf{F}\Lambda^\N,\mathsf{F}R^1,\mathsf{F}\Omega^\M}}^{-1}"] & (\mathsf{F}\Lambda^\N\otimes_{\mathsf{F}\Lambda^0}\mathsf{F}R^1)\otimes_{\mathsf{F}\Omega^0}\mathsf{F}\Omega^\M       
\end{tikzcd}
\end{equation}
commutes for all $\N,\M\in \mathbb{N}^k$. 
\end{lem}
\begin{proof}
For any $i=1,2,\ldots,k$, since $\mathsf{F}\Lambda^{\mathbf{e}_i}\otimes_{\mathsf{F}\Lambda^0}\mathsf{F}R^1\cong \mathsf{F}(E_{\mathbf{e}_i}^\Lambda \circ E_R)^1$ and $\mathsf{F}R^1\otimes_{\mathsf{F}\Omega^0}\mathsf{F}\Omega^{\mathbf{e}_i}\cong \mathsf{F}(E_R\circ E_{\mathbf{e}_i}^\Omega)^1$, the flip 
\[\mathfrak{f}_i:E_{\mathbf{e}_i}^\Lambda\circ E_R\longrightarrow E_R\circ E_{\mathbf{e}_i}^\Omega\] induces the isomorphism
\[\Tilde{\mathfrak{f}}_i:\mathsf{F}\Lambda^{\mathbf{e}_i}\otimes_{\mathsf{F}\Lambda^0}\mathsf{F}R^1\longrightarrow \mathsf{F}R^1\otimes_{\mathsf{F}\Omega^0}\mathsf{F}\Omega^{\mathbf{e}_i}.\] Let $\sigma_{\mathbf{e}_i}:=\Tilde{\mathfrak{f}}_i$ for all $i=1,2,\ldots,k$. Suppose we have defined $\sigma_\M$ for some $m\in \mathbb{N}^k$. Then, we define \[\sigma_{\M+\mathbf{e}_i}:\mathsf{F}\Lambda^{\M+\mathbf{e}_i}\otimes_{\mathsf{F}\Lambda^0}\mathsf{F}R^1\longrightarrow \mathsf{F}R^1\otimes_{\mathsf{F}\Omega^0}\mathsf{F}\Omega^{\M+\mathbf{e}_i}\] as \[\sigma_{\M+\mathbf{e}_i}:=\left(id_{\mathsf{F}R^1}\otimes (\phi_{\M,\mathbf{e}_i}^\Omega)^{-1}\right) \circ \alpha_{_{\mathsf{F}R^1,\mathsf{F}\Omega^\M,\mathsf{F}\Omega^{\mathbf{e}_i}}}\circ \left(\sigma_\M\otimes id_{\mathsf{F}\Omega^{\mathbf{e}_i}}\right)\circ \alpha_{_{\mathsf{F}\Lambda^\M,\mathsf{F}R^1,\mathsf{F}\Omega^{\mathbf{e}_i}}}^{-1}\circ \left(id_{\mathsf{F}\Lambda^\M}\otimes \Tilde{\mathfrak{f}}_i\right)\circ \alpha_{_{\mathsf{F}\Lambda^\M,\mathsf{F}\Lambda^{\mathbf{e}_i},\mathsf{F}R^1}}\circ \left(\phi_{\M,\mathbf{e}_i}^\Lambda\otimes id_{\mathsf{F}R^1}\right).\] The commutativity of (\ref{comd3}) guarantees that if $\sigma_\M,\sigma_{\M'}$ are already defined and $\M+\mathbf{e}_i=\M'+\mathbf{e}_j$ for $i\neq j$, then $\sigma_{\M+\mathbf{e}_i}=\sigma_{\M'+\mathbf{e}_j}$. Therefore, we can define $\sigma_\N$ for each $\N\in \mathbb{N}^k$ without bothering about which basic $k$-tuple to choose first in the factorization $\N=n_1 \mathbf{e}_1+n_2 \mathbf{e}_2+\dots +n_k \mathbf{e}_k$. 

Now, we prove the commutativity of (\ref{comd4}). Fix any $\N\in \mathbb{N}^k$. We claim that (\ref{comd4}) commutes for every $\M\in \mathbb{N}^k$. For this, we proceed by induction on $\ell(\M)=m_1+m_2+\dots+m_k$. The case $\ell(\M)=1$ is already covered in view of our inductive definition of the maps $\sigma_p$'s. Assume that (\ref{comd4}) commutes for all $\M\in \mathbb{N}^k$ with $\ell(\M)=N(\ge 1)$ and choose any $\mathbf{p}\in \mathbb{N}^k$ with $\ell(\mathbf{p})=N+1$. Then, $\mathbf{p}=\M+\mathbf{e}_i$ for some $i\in \{1,2,\ldots,k\}$. Now, 
{\allowdisplaybreaks
\begin{align*}
    &\left(id_{\mathsf{F}R^1}\otimes (\phi_{\N,\mathbf{p}}^\Omega)^{-1}\right) \circ \alpha_{_{\mathsf{F}R^1,\mathsf{F}\Omega^\N,\mathsf{F}\Omega^\mathbf{p}}}\circ \left(\sigma_\N\otimes id_{\mathsf{F}\Omega^\mathbf{p}}\right)\circ \alpha_{_{\mathsf{F}\Lambda^\N,\mathsf{F}R^1,\mathsf{F}\Omega^\mathbf{p}}}^{-1}\circ \left(id_{\mathsf{F}\Lambda^\N}\otimes \sigma_\mathbf{p}\right)\circ \alpha_{_{\mathsf{F}\Lambda^\N,\mathsf{F}\Lambda^\mathbf{p},\mathsf{F}R^1}}\circ \left(\phi_{\N,\mathbf{p}}^\Lambda\otimes id_{\mathsf{F}R^1}\right)\\    
    =&\left(id_{\mathsf{F}R^1}\otimes (\phi_{\N+\M,\mathbf{e}_i}^\Omega)^{-1}\right)\circ \left(id_{\mathsf{F}R^1}\otimes (\phi_{\N,\M}^\Omega)^{-1}\otimes id_{\mathsf{F}\Omega^{\mathbf{e}_i}}\right)\circ \left(id_{\mathsf{F}R^1}\otimes \alpha_{_{\mathsf{F}\Omega^\N,\mathsf{F}\Omega^\M,\mathsf{F}\Omega^{\mathbf{e}_i}}}^{-1}\right)\circ \left(id_{\mathsf{F}R^1}\otimes id_{\mathsf{F}\Omega^\N}\otimes \phi_{\M,\mathbf{e}_i}^\Omega\right)\circ \\
    &\alpha_{_{\mathsf{F}R^1,\mathsf{F}\Omega^\N,\mathsf{F}\Omega^\mathbf{p}}}\circ \left(\sigma_\N\otimes id_{\mathsf{F}\Omega^\mathbf{p}}\right)\circ \alpha_{_{\mathsf{F}\Lambda^\N,\mathsf{F}R^1,\mathsf{F}\Omega^\mathbf{p}}}^{-1}\circ \left(id_{\mathsf{F}\Lambda^\N}\otimes \sigma_\mathbf{p}\right)\circ \alpha_{_{\mathsf{F}\Lambda^\N,\mathsf{F}\Lambda^\mathbf{p},\mathsf{F}R^1}}\circ \left(\phi_{\N,\mathbf{p}}^\Lambda\otimes id_{\mathsf{F}R^1}\right)\\
    =&\left(id_{\mathsf{F}R^1}\otimes (\phi_{\N+\M,\mathbf{e}_i}^\Omega)^{-1}\right)\circ \left(id_{\mathsf{F}R^1}\otimes (\phi_{\N,\M}^\Omega)^{-1}\otimes id_{\mathsf{F}\Omega^{\mathbf{e}_i}}\right)\circ \left(id_{\mathsf{F}R^1}\otimes \alpha_{_{\mathsf{F}\Omega^\N,\mathsf{F}\Omega^\M,\mathsf{F}\Omega^{\mathbf{e}_i}}}^{-1}\right)\circ \left(id_{\mathsf{F}R^1}\otimes id_{\mathsf{F}\Omega^\N}\otimes \phi_{\M,\mathbf{e}_i}^\Omega\right)\circ \\
    &\alpha_{_{\mathsf{F}R^1,\mathsf{F}\Omega^\N,\mathsf{F}\Omega^\mathbf{p}}}\circ \left(\sigma_\N\otimes id_{\mathsf{F}\Omega^\mathbf{p}}\right)\circ \alpha_{_{\mathsf{F}\Lambda^\N,\mathsf{F}R^1,\mathsf{F}\Omega^\mathbf{p}}}^{-1}\circ \left(id_{\mathsf{F}\Lambda^\N}\otimes \sigma_\mathbf{p}\right)\circ \alpha_{_{\mathsf{F}\Lambda^\N,\mathsf{F}\Lambda^\mathbf{p},\mathsf{F}R^1}}\circ \left(id_{\mathsf{F}\Lambda^\N}\otimes (\phi_{\M,\mathbf{e}_i}^\Lambda)^{-1}\otimes id_{\mathsf{F}R^1}\right)\circ\\   &\left(\alpha_{_{\mathsf{F}\Lambda^\N,\mathsf{F}\Lambda^\M,\mathsf{F}\Lambda^{\mathbf{e}_i}}}\otimes id_{\mathsf{F}R^1}\right)\circ \left(\phi_{\N,\M}^\Lambda\otimes id_{\mathsf{F}\Lambda^{\mathbf{e}_i}}\otimes id_{\mathsf{F}R^1}\right)\circ \left(\phi_{\N+\M,\mathbf{e}_i}^\Lambda\otimes id_{\mathsf{F}R^1}\right)\\
    =&\left(id_{\mathsf{F}R^1}\otimes (\phi_{\N+\M,\mathbf{e}_i}^\Omega)^{-1}\right)\circ \alpha_{_{\mathsf{F}R^1,\mathsf{F}\Omega^{\N+\M},\mathsf{F}\Omega^{\mathbf{e}_i}}}\circ \overline{\left(id_{\mathsf{F}R^1}\otimes (\phi_{\N,\M}^\Omega)^{-1}\otimes id_{\mathsf{F}\Omega^{\mathbf{e}_i}}\right)\circ \left(\alpha_{_{\mathsf{F}R^1,\mathsf{F}\Omega^\N,\mathsf{F}\Omega^\M}}\otimes id_{\mathsf{F}\Omega^{\mathbf{e}_i}}\right)\circ}\\
    &\overline{\left(\sigma_\N\otimes id_{\mathsf{F}\Omega^\M}\otimes id_{\mathsf{F}\Omega^{\mathbf{e}_i}}\right)\circ \left(\alpha_{_{\mathsf{F}\Lambda^\N,\mathsf{F}R^1,\mathsf{F}\Omega^\M}}^{-1}\otimes id_{\mathsf{F}\Omega^{\mathbf{e}_i}}\right)\circ \left(id_{\mathsf{F}\Lambda^\N}\otimes \sigma_\M \otimes id_{\mathsf{F}\Omega^{\mathbf{e}_i}}\right)\circ \left(\alpha_{_{\mathsf{F}\Lambda^\N,\mathsf{F}\Lambda^\M,\mathsf{F}R^1}}\otimes id_{\mathsf{F}\Omega^{\mathbf{e}_i}}\right)\circ}\\
    &\overline{\left(\phi_{\N,\M}^\Lambda \otimes id_{\mathsf{F}R^1} \otimes id_{\mathsf{F}\Omega^{\mathbf{e}_i}}\right)}\circ \alpha_{_{\mathsf{F}\Lambda^{\N+\M},\mathsf{F}R^1,\mathsf{F}\Omega^{\mathbf{e}_i}}}^{-1}\circ \left(id_{\mathsf{F}\Lambda^{\N+\M}}\otimes \Tilde{\mathfrak{f}}_i\right)\circ \alpha_{_{\mathsf{F}\Lambda^{\N+\M},\mathsf{F}\Lambda^{\mathbf{e}_i},\mathsf{F}R^1}}\circ \left(\phi_{\N+\M,\mathbf{e}_i}^\Lambda\otimes id_{\mathsf{F}R^1}\right)\\
    =&\left(id_{\mathsf{F}R^1}\otimes (\phi_{\N+\M,\mathbf{e}_i}^\Omega)^{-1}\right)\circ \alpha_{_{\mathsf{F}R^1,\mathsf{F}\Omega^{\N+\M},\mathsf{F}\Omega^{\mathbf{e}_i}}}\circ \textcolor{blue}{\left(\sigma_{\N+\M}\otimes id_{\mathsf{F}\Omega^{\mathbf{e}_i}}\right)}\circ \alpha_{_{\mathsf{F}\Lambda^{\N+\M},\mathsf{F}R^1,\mathsf{F}\Omega^{\mathbf{e}_i}}}^{-1}\circ \left(id_{\mathsf{F}\Lambda^{\N+\M}}\otimes \Tilde{\mathfrak{f}}_i\right)\circ \alpha_{_{\mathsf{F}\Lambda^{\N+\M},\mathsf{F}\Lambda^{\mathbf{e}_i},\mathsf{F}R^1}}\circ\\ &\left(\phi_{\N+\M,\mathbf{e}_i}^\Lambda\otimes id_{\mathsf{F}R^1}\right)\\
    =&\sigma_{\N+\M+\mathbf{e}_i}=\sigma_{\N+\mathbf{p}}.
\end{align*}
}
We use the induction hypothesis on the overlined composition and replace it in the next step by the map marked in blue. All other replacements follow from associativity of tensor products. Hence, by induction, (\ref{comd4}) commutes for all $\M\in \mathbb{N}^k$. Since $\N$ was chosen arbitrarily, this completes the proof.
\end{proof}

Now, we are in a position to proceed towards the construction of the bridging bimodule in the higher-rank setting. 

We continue with the assumption that $\Lambda$ and $\Omega$ are row-finite $k$-graphs without sources and with finite object sets. Suppose we have a matrix $R\in \mathbb{M}_{\Lambda^0\times \Omega^0}(\mathbb{N})$ such that $A_{\mathbf{e}_i}R=RB_{\mathbf{e}_i}$ for all $i=1,2,\ldots,k$. Note that the matrix of the product polymorphism $E_{\mathbf{e}_i}^\Lambda\circ E_R$ is the product of the individual matrices, i.e., $A_{\mathbf{e}_i}R$ and similarly, the matrix of $E_R\circ E_{\mathbf{e}_i}^\Omega$ is $RB_{\mathbf{e}_i}$. Therefore, the polymorphisms $E_{\mathbf{e}_i}^\Lambda\circ E_R$ and $E_R\circ E_{\mathbf{e}_i}^\Omega$ are isomorphic for each $i=1,2,\ldots,k$. 

Let us choose isomorphisms 
\[\mathfrak{f}_i:E_{\mathbf{e}_i}^\Lambda\circ E_R\longrightarrow E_R\circ E_{\mathbf{e}_i}^\Omega,\] such that (\ref{comd3}) commutes, i.e., we choose a specific family of isomorphisms $\mathfrak{F}=(\mathfrak{f}_i)_{i\le i\le k}$ such that $(R,\mathfrak{F})$ is a $\Lambda-\Omega$ bridging pair. 

We set \[M(R):=\mathsf{F}R^1\otimes_{\mathsf{F}\Omega^0}\KP_\mathsf{F}(\Omega).\] Then $M(R)$ is automatically an $\mathsf{F}\Lambda^0-\KP_\mathsf{F}(\Omega)$-bimodule. We now impose a certain left $\KP_\mathsf{F}(\Lambda)$ action on $M(R)$ which will make it a $\KP_\mathsf{F}(\Lambda)-\KP_\mathsf{F}(\Omega)$-bimodule. We begin by fixing some notation. 
\begin{itemize}
    \item For any $A-B$-bimodule $M$ and $a\in A$, $\mu_a:M\longrightarrow M$ is the $B$-module endomorphism defined as left multiplication by $a$. 

    \item For any $\lambda\in \Lambda$ and any $\mathsf{F}\Lambda^0-\mathsf{F}\Omega^0$-bimodule $M$, define $\theta_\lambda:M\longrightarrow \mathsf{F}\Lambda^{d(\lambda)}\otimes_{\mathsf{F}\Lambda^0}M$ to be the $\mathsf{F}\Lambda^0-\mathsf{F}\Omega^0$-bimodule homomorphisms such that \[\theta_\lambda(m):=\lambda\otimes m,\] for all $m\in M$.

    \item For any $\mu\in \Lambda$ and any $\mathsf{F}\Lambda^0-\mathsf{F}\Omega^0$-bimodule $M$, define $\theta_{\mu^*}:\mathsf{F}\Lambda^{d(\mu)}\otimes_{\mathsf{F}\Lambda^0}M\longrightarrow M$ to be the $\mathsf{F}\Lambda^0-\mathsf{F}\Omega^0$-bimodule homomorphism such that \[\theta_{\mu^*}(x\otimes m):=\left(\pi^{-1}(s_{\mu^*}\pi(x))\right)m,\] for all $x\in \mathsf{F}\Lambda^{d(\mu)}$ and $m\in M$. 
\end{itemize}
Next, we construct a certain Kumjian--Pask $\Lambda$-family inside the $\mathsf{F}$-algebra $\End_{\KP_\mathsf{F}(\Omega)}(M(R))$. Define
\begin{center}
    $P_v:=\mu_v\otimes id_{\KP_\mathsf{F}(\Omega)}$;

    $S_\lambda:=\left(id_{\mathsf{F}R^1}\otimes \rho_{\Omega,d(\lambda)}\right)\circ \alpha_{_{\mathsf{F}R^1,\mathsf{F}\Omega^{d(\lambda)},\KP_\mathsf{F}(\Omega)}}\circ \left(\sigma_{d(\lambda)}\otimes id_{\KP_\mathsf{F}(\Omega)}\right)\circ \left(\theta_\lambda\otimes id_{\KP_\mathsf{F}(\Omega)}\right)$;

    $S_{\mu^*}:=\left(\theta_{\mu^*}\otimes id_{\KP_\mathsf{F}(\Omega)}\right)\circ \left(\sigma_{d(\mu)}^{-1}\otimes id_{\KP_\mathsf{F}(\Omega)}\right)\circ \alpha_{_{\mathsf{F}R^1,\mathsf{F}\Omega^{d(\mu)},\KP_\mathsf{F}(\Omega)}}^{-1}\circ \left(id_{\mathsf{F}R^1}\otimes\rho_{\Omega,d(\mu)}^{-1}\right)$;
\end{center}
for all $v\in \Lambda^0$ and $\lambda,\mu\in \Lambda$. 

The relations $(KP1)$, $(KP3)$ and $(KP4)$ can be verified exactly in the same way as it was done for Relations $1,4$ and $5$ in the setting of Leavitt path algebras (see the proof of \cite[Theorem 4.11]{ART}). So we do not repeat these here. Also, one can verify that 
\begin{align*}
    P_{r(\lambda)}S_\lambda=S_\lambda=S_\lambda P_{s(\lambda)} ~\text{and}~
    P_{s(\lambda)}S_{\lambda^*}=S_{\lambda^*}=S_{\lambda^*} P_{r(\lambda)}
\end{align*}
for all $\lambda\in \Lambda^{\neq 0}$. Below, we only verify the remaining part of $(KP2)$, i.e., $S_\lambda S_\mu=S_{\lambda\mu}$ and $S_{\mu^*}S_{\lambda^*}=S_{(\lambda\mu)^*}$ for all $\lambda,\mu\in \Lambda^{\neq 0}$ such that $s(\lambda)=r(\mu)$. 

Let $\lambda,\mu\in \Lambda^{\neq 0}$ be such that $s(\lambda)=r(\mu)$, $d(\lambda)=\N$ and $d(\mu)=\M$. Then, we have

{\allowdisplaybreaks
\begin{align*}
    S_\lambda S_\mu=& \left(\left(id_{\mathsf{F}R^1}\otimes \rho_{\Omega,\N}\right)\circ \alpha_{_{\mathsf{F}R^1,\mathsf{F}\Omega^\N,\KP_\mathsf{F}(\Omega)}}\circ \left(\sigma_\N\otimes id_{\KP_\mathsf{F}(\Omega)}\right)\circ \left(\theta_\lambda\otimes id_{\KP_\mathsf{F}(\Omega)}\right)\right)\circ\\ &\left(\left(id_{\mathsf{F}R^1}\otimes \rho_{\Omega,\M}\right)\circ \alpha_{_{\mathsf{F}R^1,\mathsf{F}\Omega^\M,\KP_\mathsf{F}(\Omega)}}\circ \left(\sigma_\M\otimes id_{\KP_\mathsf{F}(\Omega)}\right)\circ \left(\theta_\mu\otimes id_{\KP_\mathsf{F}(\Omega)}\right)\right)\\
    =& \left(id_{\mathsf{F}R^1}\otimes \rho_{\Omega,\N}\right)\circ \alpha_{_{\mathsf{F}R^1,\mathsf{F}\Omega^\N,\KP_\mathsf{F}(\Omega)}}\circ \left(\sigma_\N\theta_\lambda\otimes \rho_{\Omega,\M}\right) \circ \alpha_{_{\mathsf{F}R^1,\mathsf{F}\Omega^\M,\KP_\mathsf{F}(\Omega)}}\circ \left(\sigma_\M\otimes id_{\KP_\mathsf{F}(\Omega)}\right)\circ \left(\theta_\mu\otimes id_{\KP_\mathsf{F}(\Omega)}\right)\\
    =& \left(id_{\mathsf{F}R^1}\otimes \rho_{\Omega,\N}\right)\circ \alpha_{_{\mathsf{F}R^1,\mathsf{F}\Omega^\N,\KP_\mathsf{F}(\Omega)}}\circ \left((id_{\mathsf{F}R^1}\otimes id_{\mathsf{F}\Omega^\N})\otimes \rho_{\Omega,\M}\right)\circ \left(\sigma_\N\theta_\lambda\otimes (id_{\mathsf{F}\Omega^\M}\otimes id_{\KP_\mathsf{F}(\Omega)})\right)\circ\\  &\alpha_{_{\mathsf{F}R^1,\mathsf{F}\Omega^\M,\KP_\mathsf{F}(\Omega)}}\circ \left(\sigma_\M\otimes id_{\KP_\mathsf{F}(\Omega)}\right)\circ \left(\theta_\mu\otimes id_{\KP_\mathsf{F}(\Omega)}\right)\\
    =& \left(id_{\mathsf{F}R^1}\otimes \rho_{\Omega,\N}\right)\circ \left(id_{\mathsf{F}R^1}\otimes (id_{\mathsf{F}\Omega^\N}\otimes \rho_{\Omega,\M})\right)\circ \alpha_{_{\mathsf{F}R^1,\mathsf{F}\Omega^\N,\mathsf{F}\Omega^\M\otimes \KP_\mathsf{F}(\Omega)}}\circ \alpha_{_{\mathsf{F}R^1\otimes \mathsf{F}\Omega^\N,\mathsf{F}\Omega^\M,\KP_\mathsf{F}(\Omega)}}\circ\\
    &\left((\sigma_\N\theta_\lambda\otimes (id_{\mathsf{F}\Omega^\M})\otimes id_{\KP_\mathsf{F}(\Omega)}\right)\circ \left(\sigma_\M\otimes id_{\KP_\mathsf{F}(\Omega)}\right)\circ \left(\theta_\mu\otimes id_{\KP_\mathsf{F}(\Omega)}\right)\\
    =& \left(id_{\mathsf{F}R^1}\otimes \rho_{\Omega,\N}\right)\circ \left(id_{\mathsf{F}R^1}\otimes (id_{\mathsf{F}\Omega^\N}\otimes \rho_{\Omega,\M})\right)\circ \left(id_{\mathsf{F}R^1}\otimes \alpha_{_{\mathsf{F}\Omega^\N,\mathsf{F}\Omega^\M,\KP_\mathsf{F}(\Omega)}}\right)\circ \left(id_{\mathsf{F}R^1}\otimes (\phi_{\N,\M}^\Omega\otimes id_{\KP_\mathsf{F}(\Omega)})\right)\circ\\
    &\alpha_{_{\mathsf{F}R^1,\mathsf{F}\Omega^{\N+\M},\KP_\mathsf{F}(\Omega)}}\circ \left((id_{\mathsf{F}R^1}\otimes (\phi_{\N,\M}^\Omega)^{-1})\otimes id_{\KP_\mathsf{F}(\Omega)}\right)\circ \left(\alpha_{_{\mathsf{F}R^1,\mathsf{F}\Omega^\N,\mathsf{F}\Omega^\M}}\otimes id_{\KP_\mathsf{F}(\Omega)}\right)\circ\\
    &\left((\sigma_\N\theta_\lambda\otimes (id_{\mathsf{F}\Omega^\M})\otimes id_{\KP_\mathsf{F}(\Omega)}\right)\circ \left(\sigma_\M\otimes id_{\KP_\mathsf{F}(\Omega)}\right)\circ \left(\theta_\mu\otimes id_{\KP_\mathsf{F}(\Omega)}\right)\\
    =& \left(id_{\mathsf{F}R^1}\otimes \left(\rho_{\Omega,\N}\circ (id_{\mathsf{F}\Omega^\N}\otimes \rho_{\Omega,\M})\circ \alpha_{_{\mathsf{F}\Omega^\N,\mathsf{F}\Omega^\M,\KP_\mathsf{F}(\Omega)}}\circ (\phi_{\N,\M}^\Omega\otimes id_{\KP_\mathsf{F}(\Omega)})\right)\right)\circ \alpha_{_{\mathsf{F}R^1,\mathsf{F}\Omega^{\N+\M},\KP_\mathsf{F}(\Omega)}}\circ\\
    &\left((id_{\mathsf{F}R^1}\otimes (\phi_{\N,\M}^\Omega)^{-1})\otimes id_{\KP_\mathsf{F}(\Omega)}\right)\circ \left(\alpha_{_{\mathsf{F}R^1,\mathsf{F}\Omega^\N,\mathsf{F}\Omega^\M}}\otimes id_{\KP_\mathsf{F}(\Omega)}\right)\circ \left((\sigma_\N\otimes id_{\mathsf{F}\Omega^\M})\otimes id_{\KP_\mathsf{F}(\Omega)}\right)\circ\\
    &\overline{\left((\theta_\lambda\otimes id_{\mathsf{F}\Omega^\M})\otimes id_{\KP_\mathsf{F}(\Omega)}\right)\circ \left(\sigma_\M\otimes id_{\KP_\mathsf{F}(\Omega)}\right)\circ \left(\theta_\mu\otimes id_{\KP_\mathsf{F}(\Omega)}\right)}\\
    =& \left(id_{\mathsf{F}R^1}\otimes \rho_{\Omega,\N+\M}\right)\circ \alpha_{_{\mathsf{F}R^1,\mathsf{F}\Omega^{\N+\M},\KP_\mathsf{F}(\Omega)}}\circ \left((id_{\mathsf{F}R^1}\otimes (\phi_{\N,\M}^\Omega)^{-1})\otimes id_{\KP_\mathsf{F}(\Omega)}\right)\circ \left(\alpha_{_{\mathsf{F}R^1,\mathsf{F}\Omega^\N,\mathsf{F}\Omega^\M}}\otimes id_{\KP_\mathsf{F}(\Omega)}\right)\circ\\
    &\left((\sigma_\N\otimes id_{\mathsf{F}\Omega^\M})\otimes id_{\KP_\mathsf{F}(\Omega)}\right)\circ \textcolor{blue}{\left(\alpha_{_{\mathsf{F}\Lambda^\N,\mathsf{F}R^1,\mathsf{F}\Omega^\M}}^{-1}\otimes id_{\KP_\mathsf{F}(\Omega)}\right)\circ \left((id_{\mathsf{F}\Lambda^\N}\otimes \sigma_\M)\otimes id_{\KP_\mathsf{F}(\Omega)}\right)\circ}\\
    &\textcolor{blue}{\left(\alpha_{_{\mathsf{F}\Lambda^\N,\mathsf{F}\Lambda^\M,\mathsf{F}R^1}}\otimes id_{\KP_\mathsf{F}(\Omega)}\right)\circ \left((\phi_{\N,\M}^\Lambda\otimes id_{\mathsf{F}R^1})\otimes id_{\KP_\mathsf{F}(\Omega)}\right)\circ \left(\theta_{\lambda\mu}\otimes id_{\KP_\mathsf{F}(\Omega)}\right)}~\text{(Using (\ref{comd2}) for $\Omega$)}\\
    =&\left(id_{\mathsf{F}R^1\otimes \rho_{\Omega,\N+\M}}\right)\circ \alpha_{_{\mathsf{F}R^1,\mathsf{F}\Omega^{\N+\M},\KP_\mathsf{F}(\Omega)}}\circ \left(\sigma_{\N+\M}\otimes id_{\KP_\mathsf{F}(\Omega)}\right)\circ \left(\theta_{\lambda\mu}\otimes id_{\KP_\mathsf{F}(\Omega)}\right)~\text{(Using (\ref{comd4}))}\\
    =& S_{\lambda\mu}
\end{align*}
}
and 
{\allowdisplaybreaks
\begin{align*}
    S_{\mu^*}S_{\lambda^*}=& \left((\theta_{\mu^*}\otimes id_{\KP_\mathsf{F}(\Omega)})\circ (\sigma_\M^{-1}\otimes id_{\KP_\mathsf{F}(\Omega)})\circ \alpha_{_{\mathsf{F}R^1,\mathsf{F}\Omega^\M,\KP_\mathsf{F}(\Omega)}}\circ (id_{\mathsf{F}R^1}\otimes \rho_{\Omega,\M}^{-1})\right)\circ\\
    &\left((\theta_{\lambda^*}\otimes id_{\KP_\mathsf{F}(\Omega)})\circ (\sigma_\N^{-1}\otimes id_{\KP_\mathsf{F}(\Omega)})\circ \alpha_{_{\mathsf{F}R^1,\mathsf{F}\Omega^\N,\KP_\mathsf{F}(\Omega)}}\circ (id_{\mathsf{F}R^1}\otimes \rho_{\Omega,\N}^{-1})\right)\\
    =& \left(\theta_{\mu^*}\otimes id_{\KP_\mathsf{F}(\Omega)}\right)\circ (\sigma_\M^{-1}\otimes id_{\KP_\mathsf{F}(\Omega)})\circ \alpha_{_{\mathsf{F}R^1,\mathsf{F}\Omega^\M,\KP_\mathsf{F}(\Omega)}}^{-1}\circ \left(\theta_{\lambda^*}\otimes \rho_{\Omega,\M}^{-1}\right)\circ \left(\sigma_\N^{-1}\otimes id_{\KP_\mathsf{F}(\Omega)}\right)\circ\\
    &\alpha_{_{\mathsf{F}R^1,\mathsf{F}\Omega^\N,\KP_\mathsf{F}(\Omega)}}^{-1}\circ \left(id_{\mathsf{F}R^1}\otimes \rho_{\Omega,\N}^{-1}\right)\\
    =& \overline{\left(\theta_{\mu^*}\otimes id_{\KP_\mathsf{F}(\Omega)}\right)\circ \left(\sigma_\M^{-1}\otimes id_{\KP_\mathsf{F}(\Omega)}\right)\circ \left((\theta_{\lambda^*}\otimes id_{\mathsf{F}\Omega^\M})\otimes id_{\KP_\mathsf{F}(\Omega)}\right)}\circ \alpha_{_{\mathsf{F}\Lambda^\N\otimes \mathsf{F}R^1,\mathsf{F}\Omega^\M,\KP_\mathsf{F}(\Omega)}}^{-1}\circ\\
    &\left((id_{\mathsf{F}\Lambda^\N}\otimes id_{\mathsf{F}R^1})\otimes \rho_{\Omega,\M}^{-1}\right)\circ \left(\sigma_\N^{-1}\otimes id_{\KP_\mathsf{F}(\Omega)}\right)\circ \alpha_{_{\mathsf{F}R^1,\mathsf{F}\Omega^\N,\KP_\mathsf{F}(\Omega)}}^{-1}\circ \left(id_{\mathsf{F}R^1}\otimes \rho_{\Omega,\N}^{-1}\right)\\
    =& \textcolor{blue}{\left(\theta_{(\lambda\mu)^*}\otimes id_{\KP_\mathsf{F}(\Omega)}\right)\circ \left(((\phi_{\N,\M}^\Lambda)^{-1}\otimes id_{\mathsf{F}R^1})\otimes id_{\KP_\mathsf{F}(\Omega)}\right)\circ \left(\alpha_{_{\mathsf{F}\Lambda^\N,\mathsf{F}\Lambda^\M,\mathsf{F}R^1}}^{-1}\otimes id_{\KP_\mathsf{F}(\Omega)}\right)\circ}\\
    &\textcolor{blue}{\left((id_{\mathsf{F}\Lambda^\N}\otimes \sigma_\M^{-1})\otimes id_{\KP_\mathsf{F}(\Omega)}\right)\circ \left(\alpha_{_{\mathsf{F}\Lambda^\N,\mathsf{F}R^1,\mathsf{F}\Omega^\M}}\otimes id_{\KP_\mathsf{F}(\Omega)}\right)}\circ \alpha_{_{\mathsf{F}\Lambda^\N\otimes \mathsf{F}R^1,\mathsf{F}\Omega^\M,\KP_\mathsf{F}(\Omega)}}^{-1}\circ \left((id_{\mathsf{F}\Lambda^\N}\otimes id_{\mathsf{F}R^1})\otimes \rho_{\Omega,\M}^{-1}\right)\circ\\
    &\left(\sigma_\N^{-1}\otimes id_{\KP_\mathsf{F}(\Omega)}\right)\circ \alpha_{_{\mathsf{F}R^1,\mathsf{F}\Omega^\N,\KP_\mathsf{F}(\Omega)}}^{-1}\circ \left(id_{\mathsf{F}R^1}\otimes \rho_{\Omega,\N}^{-1}\right)\\
    =& \left(\theta_{(\lambda\mu)^*}\otimes id_{\KP_\mathsf{F}(\Omega)}\right)\circ \left(((\phi_{\N,\M}^\Lambda)^{-1}\otimes id_{\mathsf{F}R^1})\otimes id_{\KP_\mathsf{F}(\Omega)}\right)\circ \left(\alpha_{_{\mathsf{F}\Lambda^\N,\mathsf{F}\Lambda^\M,\mathsf{F}R^1}}^{-1}\otimes id_{\KP_\mathsf{F}(\Omega)}\right)\circ\\
    &\left((id_{\mathsf{F}\Lambda^\N}\otimes \sigma_\M^{-1})\otimes id_{\KP_\mathsf{F}(\Omega)}\right)\circ \left(\alpha_{_{\mathsf{F}\Lambda^\N,\mathsf{F}R^1,\mathsf{F}\Omega^\M}}\otimes id_{\KP_\mathsf{F}(\Omega)}\right)\circ \alpha_{_{\mathsf{F}\Lambda^\N\otimes \mathsf{F}R^1,\mathsf{F}\Omega^\M,\KP_\mathsf{F}(\Omega)}}^{-1}\circ \left(\sigma_\N^{-1}\otimes(id_{\mathsf{F}\Omega^\M}\otimes id_{\KP_\mathsf{F}(\Omega)})\right)\circ\\
    &\left((id_{\mathsf{F}R^1}\otimes id_{\mathsf{F}\Omega^\N})\otimes \rho_{\Omega,\M}^{-1}\right)\circ \alpha_{_{\mathsf{F}R^1,\mathsf{F}\Omega^\N,\KP_\mathsf{F}(\Omega)}}^{-1}\circ \left(id_{\mathsf{F}R^1}\otimes \rho_{\Omega,\N}^{-1}\right)\\
    =& \left(\theta_{(\lambda\mu)^*}\otimes id_{\KP_\mathsf{F}(\Omega)}\right)\circ \left(((\phi_{\N,\M}^\Lambda)^{-1}\otimes id_{\mathsf{F}R^1})\otimes id_{\KP_\mathsf{F}(\Omega)}\right)\circ \left(\alpha_{_{\mathsf{F}\Lambda^\N,\mathsf{F}\Lambda^\M,\mathsf{F}R^1}}^{-1}\otimes id_{\KP_\mathsf{F}(\Omega)}\right)\circ\\
    &\left((id_{\mathsf{F}\Lambda^\N}\otimes \sigma_\M^{-1})\otimes id_{\KP_\mathsf{F}(\Omega)}\right)\circ \left(\alpha_{_{\mathsf{F}\Lambda^\N,\mathsf{F}R^1,\mathsf{F}\Omega^\M}}\otimes id_{\KP_\mathsf{F}(\Omega)}\right)\circ \left((\sigma_\N^{-1}\otimes id_{\mathsf{F}\Omega^\M})\otimes id_{\KP_\mathsf{F}(\Omega)}\right)\circ\\ 
    &\alpha_{_{\mathsf{F}R^1\otimes \mathsf{F}\Omega^\N,\mathsf{F}\Omega^\M,\KP_\mathsf{F}(\Omega)}}^{-1}\circ \alpha_{_{\mathsf{F}R^1, \mathsf{F}\Omega^\N,\mathsf{F}\Omega^\M\otimes\KP_\mathsf{F}(\Omega)}}^{-1}\circ \left(id_{\mathsf{F}R^1}\otimes (id_{\mathsf{F}\Omega^\N}\otimes \rho_{\Omega,\M}^{-1})\right)\circ \left(id_{\mathsf{F}R^1}\otimes \rho_{\Omega,\N}^{-1}\right)\\
    =& \left(\theta_{(\lambda\mu)^*}\otimes id_{\KP_\mathsf{F}(\Omega)}\right)\circ \left(((\phi_{\N,\M}^\Lambda)^{-1}\otimes id_{\mathsf{F}R^1})\otimes id_{\KP_\mathsf{F}(\Omega)}\right)\circ \left(\alpha_{_{\mathsf{F}\Lambda^\N,\mathsf{F}\Lambda^\M,\mathsf{F}R^1}}^{-1}\otimes id_{\KP_\mathsf{F}(\Omega)}\right)\circ\\
    &\left((id_{\mathsf{F}\Lambda^\N}\otimes \sigma_\M^{-1})\otimes id_{\KP_\mathsf{F}(\Omega)}\right)\circ \left(\alpha_{_{\mathsf{F}\Lambda^\N,\mathsf{F}R^1,\mathsf{F}\Omega^\M}}\otimes id_{\KP_\mathsf{F}(\Omega)}\right)\circ \left((\sigma_\N^{-1}\otimes id_{\mathsf{F}\Omega^\M})\otimes id_{\KP_\mathsf{F}(\Omega)}\right)\circ\\
    &\left(\alpha_{_{\mathsf{F}R^1,\mathsf{F}\Omega^\N,\mathsf{F}\Omega^\M}}^{-1}\otimes id_{\KP_\mathsf{F}(\Omega)}\right)\circ \left((id_{\mathsf{F}R^1}\otimes \phi_{\N,\M}^\Omega)\otimes id_{\KP_\mathsf{F}(\Omega)}\right)\circ \alpha_{_{\mathsf{F}R^1,\mathsf{F}\Omega^{\N+\M},\KP_\mathsf{F}(\Omega)}}^{-1}\circ\\
    &\left(id_{\mathsf{F}R^1}\otimes ((\phi_{\N,\M}^\Omega)^{-1}\otimes id_{\KP_\mathsf{F}(\Omega)})\right)\circ \left(id_{\mathsf{F}R^1}\otimes \alpha_{_{\mathsf{F}\Omega^\N,\mathsf{F}\Omega^\M,\KP_\mathsf{F}(\Omega)}}^{-1}\right)\circ \left(id_{\mathsf{F}R^1}\otimes (id_{\mathsf{F}\Omega^\N}\otimes \rho_{\Omega,\M}^{-1})\right)\circ \left(id_{\mathsf{F}R^1}\otimes \rho_{\Omega,\N}^{-1}\right)\\
    =& \left(\theta_{(\lambda\mu)^*}\otimes id_{\KP_\mathsf{F}(\Omega)}\right)\circ \left(\sigma_{\N+\M}^{-1}\otimes id_{\KP_\mathsf{F}(\Omega)}\right)\circ \alpha_{_{\mathsf{F}R^1,\mathsf{F}\Omega^{\N+\M},\KP_\mathsf{F}(\Omega)}}^{-1}\circ \left(id_{\mathsf{F}R^1}\otimes \rho_{\Omega,\N+\M}^{-1}\right) ~\text{(Using (\ref{comd3}) and (\ref{comd4}))}\\
    =& S_{(\lambda\mu)^*}.
\end{align*}
}

With this, we have shown that $\{P_v,S_\lambda,S_{\mu^*}~|~v\in \Lambda^0,\lambda,\mu\in \Lambda^{\neq 0}\}$ is a Kumjian--Pask $\Lambda$-family in $\End_{\KP_\mathsf{F}(\Omega)}(M(R))$. Hence, by the universal property of $\KP_\mathsf{F}(\Lambda)$ (see \cite[Theorem 3.4]{Pino}), there exists a unique $\mathsf{F}$-algebra homomorphism \[\varphi:\KP_\mathsf{F}(\Lambda)\longrightarrow \End_{\KP_\mathsf{F}(\Omega)}(M(R))\] such that 
\begin{center}
    $\varphi(p_v)=P_v$, $\varphi(s_\lambda)=S_\lambda$ and $\varphi(s_{\mu^*})=S_{\mu^*}$
\end{center}
for all $v\in \Lambda^0$ and $\lambda,\mu\in \Lambda^{\neq 0}$. Now, we can define a left $\KP_\mathsf{F}(\Lambda)$-action on $M(R)$ as \[w\cdot x:=\varphi(w)(x)\] for all $w\in \KP_\mathsf{F}(\Lambda)$ and $x\in M(R)$. As each $\varphi(w)$ is a $\KP_\mathsf{F}(\Omega)$-module endomorphism, the left $\KP_\mathsf{F}(\Lambda)$-action and the right $\KP_\mathsf{F}(\Omega)$-action commute, which makes $M(R)$ a $\KP_\mathsf{F}(\Lambda)-\KP_\mathsf{F}(\Omega)$-bimodule as desired. 

For each $\N\in \mathbb{N}^k$, set \[M(R)_\N:=\mathsf{F}R^1\otimes_{\mathsf{F}\Omega^0}(\KP_\mathsf{F}(\Omega))_\N.\] Clearly, the above grading makes $M(R)$, a graded right $\KP_\mathsf{F}(\Omega)$-module.

Any simple tensor in $M(R)_\N$ can be written as $x\otimes w$, where $x\in \mathsf{F}R^1$ and $w\in (\KP_\mathsf{F}(\Omega))_\N$. Then, \[P_v(x\otimes w)=\left(\mu_v\otimes id_{\KP_\mathsf{F}(\Omega)}\right)(x\otimes w)=vx\otimes w\in M(R)_\N,\] for all $v\in \Lambda^0$. Again,
\begin{align*}
    S_\lambda(x\otimes w)=&\left(\left(id_{\mathsf{F}R^1}\otimes \rho_{\Omega,d(\lambda)}\right)\circ \alpha_{_{\mathsf{F}R^1,\mathsf{F}\Omega^{d(\lambda)},\KP_\mathsf{F}(\Omega)}}\circ \left((\sigma_{d(\lambda)}\circ\theta_\lambda)\otimes id_{\KP_\mathsf{F}(\Omega)}\right)\right)(x\otimes w)\\
    =& \left(\left(id_{\mathsf{F}R^1}\otimes \rho_{\Omega,d(\lambda)}\right)\circ \alpha_{_{\mathsf{F}R^1,\mathsf{F}\Omega^{d(\lambda)},\KP_\mathsf{F}(\Omega)}}\right)\left(\displaystyle{\sum_{j=1}^{t}}(y_j\otimes \mu_j)\otimes w\right)\\
    &~~~~~~~~~~~~~~~~~~~~~~~~~~~~~~~~~~~~~~~~~~~~~~~~~~~~~~~~~~~~~~~~~~~~~~~~~(\text{Assuming}~\sigma_{d(\lambda)}(\lambda\otimes x)=\displaystyle{\sum_{j=1}^{t}}(y_j\otimes \mu_j)\in \mathsf{F}R^1\otimes_{\mathsf{F}\Omega^0}\mathsf{F}\Omega^{d(\lambda)})\\
    =& \left(id_{\mathsf{F}R^1}\otimes \rho_{\Omega,d(\lambda)}\right)\left(\displaystyle{\sum_{j=1}^{t}}y_j\otimes (\mu_j\otimes w)\right)\\
    =& \displaystyle{\sum_{j=1}^{t}} y_j\otimes \pi(\mu_j)w\in M(R)_{d(\lambda)+\N}
\end{align*}
for all $\lambda\in \Lambda^{\neq 0}$. 

Similarly, $S_{\mu^*}(x\otimes w)\in M(R)_{\N-d(\mu)}$ for all $\mu\in \Lambda^{\neq 0}$. This shows that $P_v,S_\lambda$ and $S_{\mu^*}$ are graded $\KP_\mathsf{F}(\Omega)$-module endomorphisms of degree $0,d(\lambda)$ and $-d(\mu)$ respectively. Thus the homomorphism \[\varphi:\KP_\mathsf{F}(\Lambda)\longrightarrow \displaystyle{\bigoplus_{\N\in \mathbb{Z}^k}}\left(\End_{\KP_\mathsf{F}(\Omega)}(M(R))\right)_\N\subseteq \End_{\KP_\mathsf{F}(\Omega)}(M(R))\] preserves degrees of generators and hence a graded $\mathsf{F}$-algebra homomorphism. Consequently, \[M(R)=\displaystyle{\bigoplus_{\N\in \mathbb{Z}^k}}M(R)_\N\] becomes a graded left $\KP_\mathsf{F}(\Lambda)$-module.

\subsection{A sufficient condition for lifting}\label{ssec the condition}
The graded $\KP_\mathsf{F}(\Lambda)-\KP_\mathsf{F}(\Omega)$-bimodule $M(R)$, obtained in the previous subsection, gives rise to the tensor product functor
\begin{align*}
    \mathcal{F}_{M(R)}:=-\otimes_{\KP_\mathsf{F}(\Lambda)}M(R):\Gr\!-\KP_\mathsf{F}(\Lambda)&\longrightarrow \Gr\!-\KP_\mathsf{F}(\Omega).
\end{align*}
\begin{lem}\label{lem tensor product functor inducing group hom}
Suppose $R$ is a bridging matrix for the $k$-graphs $\Lambda$ and $\Omega$. The functor $\mathcal{F}_{M(R)}$ induces an order preserving $\mathbb{Z}[\mathbb{Z}^k]$-module homomorphism 
\begin{align*}
K_0^{\gr}(\mathcal{F}_{M(R)}):K_0^{\gr}(\KP_\mathsf{F}(\Lambda))\longrightarrow K_0^{\gr}(\KP_\mathsf{F}(\Omega))
\end{align*}
such that $K_0^{\gr}(\mathcal{F}_{M(R)})([P])= [P\otimes_{\KP_\mathsf{F}(\Lambda)}M(R)]$ for all graded projective right $\KP_\mathsf{F}(\Lambda)$-modules $P$ and $\mathfrak{h}_R=K_0^{\gr}(\mathcal{F}_{M(R)})$, where $\mathfrak{h}_R$ is the unique module homomorphism mentioned in Lemma \textup{\ref{lem the connecting matrix R}} $(i)$.     
\end{lem}
\begin{proof}
We first define a $\mathbb{Z}^k$-monoid homomorphism $\mathcal{V}^{\gr}(\KP_\mathsf{F}(\Lambda))\longrightarrow \mathcal{V}^{\gr}(\KP_\mathsf{F}(\Omega))$. Since $T_\Lambda\cong \mathcal{V}^{\gr}(\KP_\mathsf{F}(\Lambda))$ via the $\mathbb{Z}^k$-monoid isomorphism $v(\N)\longmapsto [(p_v\KP_\mathsf{F}(\Lambda))(-\N)]$ and $v(\N)$'s are generators of $T(\Lambda)$, it suffices to define the images of the graded projective right $\KP_\mathsf{F}(\Lambda)$-modules of the form $(p_v\KP_\mathsf{F}(\Lambda))(\N)$ where $v\in \Lambda^0$ and $\N\in \mathbb{Z}^k$. Define $\mathcal{V}^{\gr}(\mathcal{F}_{M(R)}):\mathcal{V}^{\gr}(\KP_\mathsf{F}(\Lambda))\longrightarrow \mathcal{V}^{\gr}(\KP_\mathsf{F}(\Omega))$ by \[\mathcal{V}^{\gr}(\mathcal{F}_{M(R)})\left([(p_v\KP_\mathsf{F}(\Lambda))(\N)]\right):=[(p_v\KP_\mathsf{F}(\Lambda))(\N)\otimes_{\KP_\mathsf{F}(\Lambda)}M(R)].\] 
Note that 
\begin{align*}
(p_v\KP_\mathsf{F}(\Lambda))(\N)\otimes_{\KP_\mathsf{F}(\Lambda)}M(R)\cong & \left(\mathsf{F}v\otimes_{\mathsf{F}\Lambda^0}\KP_\mathsf{F}(\Lambda)(\N)\right)\otimes_{\KP_\mathsf{F}(\Lambda)}M(R)\\
\cong & ~\mathsf{F}v\otimes_{\mathsf{F}\Lambda^0}\left(\KP_\mathsf{F}(\Lambda)(\N)\otimes_{\KP_\mathsf{F}(\Lambda)}M(R)\right)\\
\cong & ~\mathsf{F}v\otimes_{\mathsf{F}\Lambda^0}M(R)(\N),
\end{align*}
where the grading of $\mathsf{F}v$ is concentrated in degree zero and each isomorphism is a graded module isomorphism of graded right $\KP_\mathsf{F}(\Omega)$-modules. For any $g\in R^1$, the assignment $(g\otimes x)\longmapsto p_{s(g)} x$ extends to a graded right $\KP_\mathsf{F}(\Omega)$-module isomorphism from $\mathsf{F}g\otimes_{\mathsf{F}\Omega^0}\KP_\mathsf{F}(\Omega)$ to $p_{s(g)}\KP_\mathsf{F}(\Omega)$. Therefore,
\begin{align*}
    \mathsf{F}v\otimes_{\mathsf{F}\Lambda^0}M(R)(\N) \cong & ~\mathsf{F}v\otimes_{\mathsf{F}\Lambda^0}\left(\mathsf{F}R^1\otimes_{\mathsf{F}\Omega^0}\KP_{\mathsf{F}}(\Omega)(\N)\right)\\
    \cong &~\left(\mathsf{F}v\otimes_{\mathsf{F}\Lambda^0}\mathsf{F}R^1\right)\otimes_{\mathsf{F}\Omega^0}\KP_\mathsf{F}(\Omega)(\N)\\
    \cong &~\left(\displaystyle{\bigoplus_{\substack{w\in \Omega^0\\R(v,w)\neq 0}}}\left(\displaystyle{\bigoplus_{i=1}^{R(v,w)}} \mathsf{F}g_i^{v,w}\right)\right)\otimes_{\mathsf{F}\Omega^0}\KP_\mathsf{F}(\Omega)(\N)\\
    \cong &~\displaystyle{\bigoplus_{\substack{w\in \Omega^0\\R(v,w)\neq 0}}}\left(\displaystyle{\bigoplus_{i=1}^{R(v,w)}}(\mathsf{F}g_i^{v,w}\otimes_{\mathsf{F}\Omega^0}\KP_\mathsf{F}(\Omega)(\N))\right)\\
    \cong &~\displaystyle{\bigoplus_{\substack{w\in \Omega^0\\R(v,w)\neq 0}}} (p_w\KP_\mathsf{F}(\Omega))(\N)^{R(v,w)}
\end{align*}
and consequently, $\mathcal{V}^{\gr}(\mathcal{F}_{M(R)})\left([(p_v\KP_\mathsf{F}(\Lambda))(\N)]\right)=\displaystyle{\sum_{\substack{w\in \Omega^0\\R(v,w)\neq 0}}}R(v,w)[(p_w\KP_\mathsf{F}(\Omega))(\N)]$. This description shows that $\mathcal{V}^{\gr}(\mathcal{F}_{M(R)})$ is a well-defined monoid homomorphism between the graded $\mathcal{V}$-monoids. Evidently, it respects the $\mathbb{Z}^k$-action and hence induces an order-preserving $\mathbb{Z}[\mathbb{Z}^k]$-module homomorphism \[K_0^{\gr}(\mathcal{F}_{M(R)}):K_0^{\gr}(\KP_\mathsf{F}(\Lambda))\longrightarrow K_0^{\gr}(\KP_\mathsf{F}(\Omega)),\] via the group completion functor. Clearly, $K_0^{\gr}(\mathcal{F}_{M(R)})([P])= [P\otimes_{\KP_\mathsf{F}(\Lambda)}M(R)]$ for all graded projective right $\KP_\mathsf{F}(\Lambda)$-modules $P$. Now, 
\begin{align*}
    \gamma_\Omega\left(K_0^{\gr}(\mathcal{F}_{M(R)})([(p_v\KP_\mathsf{F}(\Lambda))(\N)])\right)=& \gamma_\Omega\left(\displaystyle{\sum_{\substack{w\in \Omega^0\\R(v,w)\neq 0}}}R(v,w)[(p_w\KP_\mathsf{F}(\Omega))(\N)]\right)\\
    =& \displaystyle{\sum_{\substack{w\in \Omega^0\\R(v,w)\neq 0}}} [R(v,w)\epsilon_w,-\N]\\
    =& [\displaystyle{\sum_{\substack{w\in \Omega^0\\R(v,w)\neq 0}}}R(v,w)\epsilon_w,-\N]\\
    =& [\epsilon_vR,-\N]=R\left([\epsilon_v,-\N]\right)=R\left(\gamma_\Lambda([(p_v\KP_\mathsf{F}(\Lambda))(\N)])\right).
\end{align*}
Thus, by uniqueness of the map $\mathfrak{h}_R$ (see Lemma \ref{lem the connecting matrix R} $(i)$), it follows that $\mathfrak{h}_R=K_0^{\gr}(\mathcal{F}_{M(R)})$.
\end{proof}
\begin{dfn}\label{def bridging homomorphism between graded K-groups}
Suppose $\Lambda$ and $\Omega$ are row-finite $k$-graphs without sources such that $|\Lambda^0|,|\Omega^0|< \infty$. A morphism $\mathfrak{h}\in \mathbb{P}_{\mathbb{Z}^k}\left(K_0^{\gr}(\KP_\mathsf{F}(\Lambda)),K_0^{\gr}(\KP_\mathsf{F}(\Omega))\right)$ is called a \emph{bridging homomorphism} if the matrix $R$ obtained in Lemma \ref{lem the connecting matrix R} $(ii)$ is a bridging matrix for $\Lambda$ and $\Omega$ (see Definition \ref{def the bridging matrix and polymorphism}). 
\end{dfn}
In the following theorem, we provide an answer to Question \ref{ques lifting question} by showing that any bridging homomorphism between graded $K$-theories of Kumjian--Pask algebras can be lifted to a graded homomorphism between the algebras. 
\begin{thm}\label{the lifting theorem}
Suppose $\Lambda$ and $\Omega$ are row-finite $k$-graphs without sources such that $|\Lambda^0|,|\Omega^0|< \infty$. Let $\mathfrak{h}:K_0^{\gr}(\KP_\mathsf{F}(\Lambda))\longrightarrow K_0^{\gr}(\KP_\mathsf{F}(\Omega))$ be any bridging homomorphism. Then, there exists a unital graded homomorphism $\psi:\KP_\mathsf{F}(\Lambda)\longrightarrow \KP_\mathsf{F}(\Omega)$ such that $K_0^{\gr}(\psi)=\mathfrak{h}$.    
\end{thm}
\begin{proof}
By Lemma \ref{lem the connecting matrix R} $(ii)$, there exist a matrix $R\in \mathbb{M}_{\Lambda^0\times \Omega^0}(\mathbb{N})$ and a $k$-tuple $\mathbf{r}\in \mathbb{Z}^k$ with $\mathfrak{h}(w)=~^{\mathbf{r}}\mathfrak{h}_R(w)$ for all $w\in K_0^{\gr}(\KP_\mathsf{F}(\Lambda))$. Since $\mathfrak{h}$ is a bridging homomorphism, $R$ is a bridging matrix for $\Lambda$ and $\Omega$, whence $\mathfrak{h}_R=K_0^{\gr}(\mathcal{F}_{M(R)})$ by Lemma \ref{lem tensor product functor inducing group hom}. In view of the description of the map $\mathcal{V}^{\gr}(\mathcal{F}_{M(R)})$, one can easily verify that for each $w\in K_0^{\gr}(\KP_\mathsf{F}(\Lambda))$, \[~^{\mathbf{r}}K_0^{\gr}(\mathcal{F}_{M(R)})(w)=K_0^{\gr}(\mathcal{F}_{M(R)(-\mathbf{r})})(w).\] Therefore, $\mathfrak{h}=K_0^{\gr}(\mathcal{F}_{M(R)(-\mathbf{r})})$. Since $\mathfrak{h}$ is pointed, \[[M(R)(-\mathbf{r})]=[\KP_\mathsf{F}(\Lambda)\otimes_{\KP_\mathsf{F}(\Lambda)}M(R)(-\mathbf{r})]=K_0^{\gr}(\mathcal{F}_{M(R)(-\mathbf{r})})([\KP_\mathsf{F}(\Lambda)])=[\KP_\mathsf{F}(\Omega)]\] in $K_0^{\gr}(\KP_\mathsf{F}(\Omega))$. Hence, $M(R)(-\mathbf{r})$ and $\KP_\mathsf{F}(\Omega)$ are isomorphic as graded right $\KP_\mathsf{F}(\Omega)$-modules. Finally, by \cite[Proposition 3.1]{Brix}, there exists a unital graded homomorphism $\psi:\KP_\mathsf{F}(\Lambda)\longrightarrow \KP_\mathsf{F}(\Omega)$ such that $K_0^{\gr}(\psi)=K_0^{\gr}(\mathcal{F}_{M(R)(-\mathbf{r})})=\mathfrak{h}$.
\end{proof}

\begin{rmk}\label{rem the generality of bridging criterion}
The lifting result for Leavitt path algebras is perfectly generalised in Theorem \ref{the lifting theorem}. The matrix $R$ of any pointed order-preserving $\mathbb{Z}[x^{-1},x]$-module homomorphism between graded Grothendieck groups of Leavitt path algebras is trivially a bridging matrix. As a result, the lifting can always be done for $1$-graphs and consequently, $K_0^{\gr}$ is a full functor.    
\end{rmk}

\section{Further directions}\label{sec future}
Given two finite directed graphs $E$ and $F$, the Graded Classification Conjecture (\cite{willie, mathann}) states that the Leavitt path algebras $L_\mathsf{F}(E)$ and $L_\mathsf{F}(F)$ with coefficient field $\mathsf{F}$, are graded isomorphic if and only if there is an order preserving $\mathbb{Z}[x^{-1},x]$-module isomorphism \[\phi:K_0^{\gr}(L_\mathsf{F}(E))\longrightarrow K_0^{\gr}(L_\mathsf{F}(F))\] such that $\phi([L_\mathsf{F}(E)])=[L_\mathsf{F}(F)]$ (equivalently, if there exists a $\mathbb{Z}$-monoid isomorphism between $T_E$ and $T_F$ which preserves the order units). So far, no counterexample has been found contradicting the conjecture. However, in the case of higher-rank graphs, we found the following example showing that the graded $K$-theory (or, equivalently, the talented monoid), in general, may fail to classify Kumjian--Pask algebras up to graded isomorphism.

\begin{example}\label{ex non-isomorphic KP-algebras with isomorphic talented monoids}
Consider the following colored directed graph:
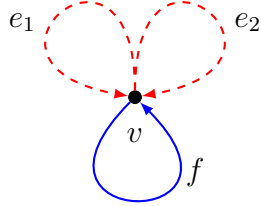
\[
\begin{tikzpicture}[scale=1]
\node[circle,draw,fill=black,inner sep=0.5pt] (p11) at (0, 0) {$.$} 

edge[-latex, red,thick,loop, dashed, in=10, out=90, min distance=70] (p11)
edge[-latex, red,thick, loop, dashed, out=90, in=170, min distance=70] (p11)

edge[-latex, blue,thick, loop, out=225, in=-45, min distance=70] (p11);
                                        \node at (0, -0.5) {$v$};

\node at (-1.5,1) {$e_1$};
\node at (1.5,1) {$e_2$};
\node at (0.8,-1) {$f$};

\end{tikzpicture}
\]
We can obtain two non-isomorphic $2$-graphs by setting two different factorization rules for the bi-colored paths in the above directed graph.

$(I)$ The $2$-graph $\Lambda_1$: We define the factorization rules as follows: \[e_1 f=fe_1,~~ e_2 f=fe_2.\]

$(II)$ The $2$-graph $\Lambda_2$: Here the factorization rules are defined as: \[e_1 f=fe_2,~~ e_2f=fe_1.\]

The defining relations for both the talented monoids $T_{\Lambda_1}$ and $T_{\Lambda_2}$ are the same: \[v((i,j))=v((i+1,j))=2 v((i,j+1));~ (i,j)\in \mathbb{Z}^2.\] Consequently, the map $v((i,j))\longmapsto v((i,j))$ is a pointed $\mathbb{Z}^2$-isomorphism between $T_{\Lambda_1}$ and $T_{\Lambda_2}$. In fact, $T_{\Lambda_1}\cong T_{\Lambda_2}\cong \mathbb{N}[1/2]$ (see Example \ref{ex on the invariance of talented monoid under moves}).

Let $\mathsf{F}$ be any field. We denote the Kumjian--Pask families in $\KP_\mathsf{F}(\Lambda_1)$ and $\KP_\mathsf{F}(\Lambda_2)$ by $(p,s)$ and $(P,S)$ respectively.

Note that if $\mathcal{A}$ is the subalgebra of $\KP_\mathsf{F}(\Lambda_2)$ generated by $S_{e_1}, S_{e_2}, S_{e_1^*}, S_{e_2^*}$, then $\mathcal{A}\cong L_\mathsf{F}(1,2)$ as $\mathsf{F}$-algebras. Setting $e_1$ as a special edge emanating from $v$, we have a basis for the $\mathsf{F}$-vector space $\mathcal{A}$ consisting of the following elements:
\begin{itemize}
\item $P_v$;

\item $S_\lambda,S_{\lambda^*}$ where $\lambda\in \Lambda_2$ and $d(\lambda)=(0,n)$ for some $n\in \mathbb{N}\setminus \{0\}$;

\item $S_\lambda S_{\mu^*}$ where $s(\lambda)=s(\mu)$, $\lambda,\mu\in \Lambda_2$, $d(\lambda)=(0,n)$, $d(\mu)=(0,m)$ for some $n,m\in \mathbb{N}\setminus \{0\}$ and if $\lambda=\alpha_1\alpha_2\cdots \alpha_n$, $\mu=\beta_1\beta_2\cdots \beta_m$, $\alpha_i,\beta_j\in \{e_1,e_2\}$ then either $\alpha_n\neq \beta_m$ or $\alpha_n=\beta_m$ and this common edge is not $e_1$. 
\end{itemize}

We observe that there cannot exist any $\mathbb{Z}^2$-graded ring isomorphism between $\KP_\mathsf{F}(\Lambda_1)$ and $\KP_\mathsf{F}(\Lambda_2)$. In $\KP_\mathsf{F}(\Lambda_1)$, using relations $(KP2)$ (see Definition \ref{def KP-algebra}), we have $s_{f^*}s_{e_1^*}=s_{e_1^*}s_{f^*}$, which upon multiplying from both sides by $s_f$, yields $s_{e_1^*}s_f=s_f s_{e_1^*}$. A similar argument shows that $s_{e_2^*}s_f=s_f s_{e_2^*}$ and in $\KP_\mathsf{F}(\Lambda_2)$, $S_{e_1^*}S_f=S_f S_{e_2^*}$, $S_{e_2^*} S_f=S_f S_{e_1^*}$. It follows that $s_f\in Z(\KP_\mathsf{F}(\Lambda_1))$, whereas $S_f\notin Z(\KP_\mathsf{F}(\Lambda_2))$. If we have a graded isomorphism $\phi:\KP_\mathsf{F}(\Lambda_1)\longrightarrow \KP_\mathsf{F}(\Lambda_2)$, then it is not hard to see that $\phi(s_f)$ should be of the form $S_f x$, where \[x=\displaystyle{\sum_{i=1}^{t}}a_iS_{\lambda_i}S_{\mu_i^*},\] $a_i\in \mathsf{F}\setminus \{0\}$, $d(\lambda_i)=d(\mu_i)=(0,n_i)$, $n_i\in \mathbb{N}$ for all $i=1,2,\ldots,t$. Clearly, $x\in \mathcal{A}\setminus \{0\}$. Since $s_f\in Z(\KP_\mathsf{F}(\Lambda_1))$, we should have $S_f x\in Z(\KP_\mathsf{F}(\Lambda_2))$. Assume that $x$ is in normal form (if not, then we can easily transform those monomials which are not in the basis of $\mathcal{A}$, to express $x$ as a linear combination of basic vectors and all the monomials of such a basic representation satisfy the same condition which is satisfied by the monomials of the starting representation). Now, two cases may appear.

\emph{Case-I:} $\mu_i(0,(0,1))=e_2$ for some $i$. Then, $\mu_i=e_2\gamma$ for some $\gamma\in \Lambda_2$ with $d(\gamma)=(0,n_i-1)$. Observe that if a monomial $S_\lambda S_{\mu^*}(\neq P_v)$ is in normal form (i.e., the first edges of $\lambda$ and $\mu$ either do not coincide or, if they do, then the common first edge is not $e_1$) and we multiply it by $S_{e_1}$ (or, $S_{e_2}$) from left or right, then the resulting element also remains in normal form. Now, \[xS_{e_1}=a_iS_{\lambda_i}S_{\mu_i^*}S_{e_1}+\displaystyle{\sum_{j\neq i}}a_jS_{\lambda_j}S_{\mu_j^*}S_{e_1}=a_iS_{\lambda_i}S_{\gamma^*}S_{e_2^*}S_{e_1}+\displaystyle{\sum_{j\neq i}}a_j S_{\lambda_j} S_{\mu_j^*}S_{e_1}=\displaystyle{\sum_{j\neq i}}a_j S_{\lambda_j} S_{\mu_j^*}S_{e_1}.\] Thus, the number of basic vectors in $xS_{e_1}$ is less than $t$, whereas the number of basic vectors in $S_{e_2}x$ is still $t$. Thus, $S_{e_2}x\neq xS_{e_1}$ which implies $S_{e_1}S_f x\neq S_f x S_{e_1}$, a contradiction to the fact that $S_f x\in Z(\KP_\mathsf{F}(\Lambda_2))$. 

\emph{Case-II:} $\mu_i(0,(0,1))=e_1$ for all $i$ with $d(\mu_i)\neq 0$. Clearly, $x\notin \mathsf{F}P_v$; otherwise, $\phi(s_f)=cS_f$ for some $c\in \mathsf{F}$ and then $\phi(s_f)\notin Z(\KP_\mathsf{F}(\Lambda_2))$ as $S_{e_1}S_f=S_fS_{e_2}\neq S_f S_{e_1}$. Now, $x S_{e_2}=0$ if $\mu_i(0,(0,1))=e_1$ for all $i$ and $x S_{e_2}\in \mathsf{F}S_{e_2}$ if $S_{\lambda_i}=S_{\mu_i}=P_v$ for some $i$, in which case, obviously, $t\ge 2$. So if $x S_{e_2}\neq 0$, then the number of summand in the basic representation of $x S_{e_2}$ reduces to $1$ which is less than $t$. On the other hand, $S_{e_1}x\neq 0$ and the number of basic vectors in $S_{e_1}x$ is $t$. Therefore, $S_{e_1}x\neq x S_{e_2}$ and so $S_{e_2}\phi(s_f)\neq \phi(s_f)S_{e_2}$, which shows $\phi(s_f)=S_f x\notin Z(\KP_\mathsf{F}(\Lambda_2))$, the same contradiction. 

In fact, it can be shown that $\KP_\mathsf{F}(\Lambda_1)$ is graded isomorphic to $L_\mathsf{F}(1,2)[t^{-1},t]$, the ring of Laurent polynomials with coefficients from $L_\mathsf{F}(1,2)$, whereas $\KP_\mathsf{F}(\Lambda_2)$ is graded isomorphic to the skew Laurent polynomial ring $L_\mathsf{F}(1,2)[t^{-1},t;\varphi]$, where $\varphi:L_\mathsf{F}(1,2)\longrightarrow L_\mathsf{F}(1,2)$ is the ring automorphism switching $e_1$ and $e_2$. The readers should be convinced that the $\mathbb{Z}^2$-grading on the Laurent polynomial (and the skew Laurent polynomial) ring is defined by setting \[L_\mathsf{F}(1,2)[t^{-1},t]_{\mathbf{n}}:=\lspan_\mathsf{F} \{xt^i~|~x\in L_\mathsf{F}(1,2)^h,i\in \mathbb{Z};~ \iota_2(\deg(x))+(i,0)=\mathbf{n}\},\] where $\deg(x)$ is the degree of $x$ in $L_\mathsf{F}(1,2)$ when we view it as the Leavitt path algebra $L_\mathsf{F}(R_2)$ with the usual $\mathbb{Z}$-grading and $\iota_2:\mathbb{Z}\longrightarrow \mathbb{Z}^2$ is the standard injection sending $j$ to $(0,j)$.     
\end{example} 

The above example shows that the factorization of paths, which may affect the algebraic structure of a Kumjian--Pask algebra, is not well captured by the graded $K$-theory. In the $1$-graph case, factorization of paths into edges is always unique and hence the above problem does not appear. However, this significant difference between $1$-graphs and general $k$-graphs makes the investigation of graded $K$-theory (or, the talented monoid) in line with the classification of Kumjian--Pask algebras more interesting. To be precise, we state the following problem.

\begin{prob}\label{prob For which k-graphs TM can classify KP-algebras upto graded isomorphism}
Characterise the class of $k$-graphs for which the existence of a pointed $\mathbb{Z}^k$-isomorphism between talented monoids implies that the respective Kumjian--Pask algebras are graded isomorphic. 
\end{prob}

We remark that the above problem is very general. One may work in a particular dimension. For example, if $k=1$, solving the problem will settle the Graded Classification Conjecture as it says that the path categories of all finite directed graphs form the required class. For $k=2$, it is interesting to know whether the relevant class of $2$-graphs can be characterised by means of bi-colored skeletons. Obviously, the bi-colored graph in Example \ref{ex non-isomorphic KP-algebras with isomorphic talented monoids} is a forbidden skeleton in such a plausible characterization.

In Theorem \ref{the lifting theorem}, we obtain a sufficient condition (via a bridging matrix) for lifting homomorphisms between graded $K$-theories of Kumjian--Pask algebras to graded homomorphisms between the Kumjian--Pask algebras. However, we have not yet explored whether the condition is necessary. If it is so, then obviously, not all pointed order-preserving homomorphisms between graded Grothendieck groups can be lifted (one can easily find such a homomorphism in view of Example \ref{ex matrix commutativity is not enough}). On the level of Leavitt path algebras, this lifting can always be done as the functor $K_0^{\gr}$ is full. One way of showing this, among others, is to realise Leavitt path algebras as \emph{Bergman algebras} (see \cite{HLR} for details). So, in order to have a more specific answer to the lifting question (Question \ref{ques lifting question}) for Kumjian--Pask algebras, it may be fruitful to execute further research in the following direction: 
\begin{prob}\label{prob For which k-graphs KP-algebras are Bergman algebras}
Describe the class of $k$-graphs whose Kumjian--Pask algebras can be realised as Bergman algebras. 
\end{prob}

{\bf Acknowledgments.} The authors are thankful to the referee for the meticulous review of the article. They also thank Ralf Meyer for helpful comments that improved the presentation of the paper. Hazrat acknowledges Australian Research Council grant DP230103184. Mukherjee would like to thank University Grants Commission (UGC), India for providing Research Fellowship (ID: 211610022010/ Joint CSIR-UGC NET JUNE 2021). Mukherjee and Sardar also acknowledge the DST-FIST Programme (No: SR/FST/MS-11/2021/101(C)). Part of this work was carried out during visits by Mukherjee and Sardar to Western Sydney University in 2025.

\end{document}